\documentclass[12pt,twoside,reqno]{amsart}
\usepackage{amssymb,amsmath,amsthm,amsfonts,calc,latexsym}
\usepackage{bbm}
\usepackage{hyperref}
\usepackage{mathrsfs}
\usepackage{enumerate}
\usepackage{tikz-cd}
\usepackage[left=3.0cm,top=3cm,right=3.0cm,bottom=3cm]{geometry}

\numberwithin{equation}{subsection}
\newtheorem{Theorem}{Theorem}[subsection]
\newtheorem*{Theorem*}{Theorem}
\newtheorem{Lemma}[Theorem]{Lemma}
\newtheorem{Corollary}[Theorem]{Corollary}
\newtheorem{Proposition}[Theorem]{Proposition}
\newtheorem{Definition}[Theorem]{Definition}
\newtheorem{Remark}[Theorem]{Remark}

\newtheorem{Example}[Theorem]{Example}
\newtheorem*{Example*}{Example}
\newtheorem{Notation}[Theorem]{Notation}

\newcommand{\ba}{\mathbf{a}}

\newcommand{\Bset}{\mathbb{B}}
\newcommand{\Cset}{\mathbb{C}}

\newcommand{\Nset}{\mathbb{N}}
\newcommand{\Rset}{\mathbb{R}}
\newcommand{\Sset}{\mathbb{S}}
\newcommand{\Zset}{\mathbb{Z}}
\newcommand{\cA}{\mathcal{A}}
\newcommand{\cB}{\mathcal{B}}
\newcommand{\cC}{\mathcal{C}}

\newcommand{\cF}{\mathcal{F}}
\newcommand{\cG}{\mathcal{G}}
\newcommand{\cH}{\mathcal{H}}
\newcommand{\cI}{\mathcal{I}}
\newcommand{\cJ}{\mathcal{J}}
\newcommand{\cK}{\mathcal{K}}
\newcommand{\cL}{\mathcal{L}}
\newcommand{\cM}{\mathcal{M}}
\newcommand{\cN}{\mathcal{N}}

\newcommand{\tiota}{\tilde{\iota}}
\newcommand{\tmu}{\tilde{\mu}}
\newcommand{\tOmega}{\tilde{\Omega}}
\newcommand{\trho}{\tilde{\rho}}

\newcommand{\tpsi}{\tilde{\psi}}
\newcommand{\tSigma}{\tilde{\Sigma}}

\newcommand{\txi}{\tilde{\xi}}
\newcommand{\tg}{\tilde{g}}
\newcommand{\tR}{\tilde{R}}

\newcommand{\tcM}{\tilde{\mathcal{M}}}

\newcommand{\calM}{\mathcal{M}}
\newcommand{\1}{\mathbbm{1}}
\newcommand{\alg}{{\operatorname{alg}}}
\newcommand{\Aut}{\operatorname{Aut}}
\newcommand{\End}{\operatorname{End}}

\newcommand{\trace}{\operatorname{tr}}
\newcommand{\gen}{\operatorname{gen}}
\newcommand{\sh}{\operatorname{sh}}
\newcommand{\Id}{\operatorname{Id}}

\newcommand{\perm}{\operatorname{perm}}
\newcommand{\distr}{\operatorname{distr}}
\title{Markovianity and the Thompson monoid $F^{+}$}

\author[C.~K\"ostler]{Claus K\"ostler}
\address{C.~K\"ostler, School of Mathematical Sciences, University College Cork, Cork, Ireland}
\email{claus@ucc.ie}

\author[A.~Krishnan]{Arundhathi Krishnan}
\address{A.~Krishnan, Department of Pure Mathematics, University of Waterloo, Ontario, Canada}
\email{arundhathi.krishnan@uwaterloo.ca} 

\author[S.~J.~Wills]{Stephen J.~Wills}
\address{S.~J.~Wills, School of Mathematical Sciences, University College Cork, Cork, Ireland}
\email{s.wills@ucc.ie} 

\dedicatory{Dedicated to the memory of Vaughan Jones}
\begin{document}
\begin{abstract}
We introduce a new distributional invariance principle, cal\-led `partial spreadability', which
emerges from the representation theory of the Thompson monoid $F^+$ in noncommutative probability
spaces. We show that a partially spreadable sequence of noncommutative random variables is adapted
to a local Markov filtration. Conversely we show that a large class of noncommutative stationary
Markov sequences provides representations of the Thompson monoid $F^+$. In the particular case 
of a classical probability space, we arrive at a de Finetti theorem for stationary Markov
sequences with values in a standard Borel space. 
\end{abstract}
\maketitle

{%
\def\widedotfill{\leaders\hbox to 10pt{\hfil.\hfil}\hfill}
\def\pg#1{\widedotfill\rlap{\hbox to 15pt{\hfill{#1}}}\par}
\rightskip=15pt\leftskip=10pt%
\newcommand{\sct}[2]{\noindent\llap{\hbox to%
    10pt{{#1}\hfill}}~\mbox{#2~}}
\newcommand{\separ}{\vspace{.2em}}
\hrule
\section*{Contents}
\sct{\ref{section:introduction}}{Introduction and some main results}%
\pg{\pageref{section:introduction}}%
\separ
\sct{\ref{section:preliminaries}}{Preliminaries}%
\pg{\pageref{section:preliminaries}}%
\separ
\sct{}{\ref{subsection:basics-on-F-F+} The Thompson group $F$ and its monoid $F^+$}%
\pg{\pageref{subsection:basics-on-F-F+}}%
\separ
\sct{}{\ref{subsection:partialshifts} The partial shifts monoid $S^+$}%
\pg{\pageref{subsection:partialshifts}}%
\separ
\sct{}{\ref{subsection:Markov-maps} Noncommutative probability spaces and Markov maps}%
\pg{\pageref{subsection:Markov-maps}}%
\separ
\sct{}{\ref{subsection:Random Variables and Invariance Principles} Noncommutative random variables and distributional} \\
\separ
\sct{}{\phantom{\ref{subsection:Random Variables and Invariance Principles}}  invariance principles}%
\pg{\pageref{subsection:Random Variables and Invariance Principles}}%
\separ
\sct{}{\ref{subsection:Independence and Markovianity} Noncommutative independence and Markovianity}%
\pg{\pageref{subsection:Independence and Markovianity}}%
\separ
\sct{}{\ref{subsection:Noncommutative Stationary Processes} Noncommutative stationary processes and dilations}%
\pg{\pageref{subsection:Noncommutative Stationary Processes}}%
\separ
\sct{\ref{section:MarkovianityandF+Classical}}{A new distributional invariance principle in classical probability}\\%
\sct{}{theory}%
\pg{\pageref{section:MarkovianityandF+Classical}}%
\separ
\sct{}{\ref{subsection:CSMPR} Algebraization of classical probability theory}%
\pg{\pageref{subsection:CSMPR}}%
\separ
\sct{}{\ref{subsection:DIPinCP} Partial spreadability and de Finetti theorems in classical}\\%
\sct{}{\phantom{\ref{subsection:DIPinCP}} probability}%
\pg{\pageref{subsection:DIPinCP}}%
\separ
\sct{}{\ref{subsection:FunctionsOfMarkov} Functions of classical stationary Markov sequences and their}\\%
\separ
\sct{}{\phantom{\ref{subsection:FunctionsOfMarkov}} algebraization}%
\pg{\pageref{subsection:FunctionsOfMarkov}}%
\separ
\sct{}{\ref{subsection:markov-dilations} A result on Markov dilations in classical probability theory}%
\pg{\pageref{subsection:markov-dilations}}%
\separ
\sct{\ref{section:markovianity-from-F+}}{Markovianity from representations of the Thompson monoid $F^+$}%
\pg{\pageref{section:markovianity-from-F+}}%
\separ
\sct{}{\ref{subsection:generating} Representations with a generating property}%
\pg{\pageref{subsection:generating}}%
\separ
\sct{}{\ref{subsection:s-f-p-a} Commuting squares and Markovianity for shifted fixed point}\\%
\sct{}{\phantom{\ref{subsection:s-f-p-a}} algebras}%
\pg{\pageref{subsection:s-f-p-a}}%
\separ
\sct{}{\ref{subsection:Markov-F} Commuting squares and Markovianity for stationary processes}%
\pg{\pageref{subsection:Markov-F}}%
\separ
\sct{}{\ref{subsection:ncdf} A noncommutative version of de Finetti's theorem}%
\pg{\pageref{subsection:ncdf}}%
\separ
\sct{\ref{section:Constructions of Reps of F+}}{Construction of representations of the Thompson monoid $F^+$}%
\pg{\pageref{section:Constructions of Reps of F+}}%
\separ
\sct{}{\ref{subsection:tensor-product} Tensor product constructions}%
\pg{\pageref{subsection:tensor-product}}%
\separ
\sct{}{\ref{subsection:constr-class-prob} Constructions in classical probability}%
\pg{\pageref{subsection:constr-class-prob}}%
\separ
\sct{}{\ref{subsection:op-alg-construction} Constructions in the framework of  operator algebras}%
\pg{\pageref{subsection:op-alg-construction}}%
\separ
\sct{}{Acknowledgements}%
\pg{\pageref{section:acknowledgements}}%
\separ
\sct{}{References}%
\pg{\pageref{section:bibliography}}
}
\normalsize
\vspace{12pt}
\hrule
\vspace{12pt}

\sloppy
\allowdisplaybreaks

\section{Introduction and some main results}\label{section:introduction}
The goal of this paper is to introduce `partial spreadability' as a new distributional
invariance principle as it emerges from the representation theory of the Thompson monoid $F^+$
and to present first results on the following discovery: the monoid $F^+$ algebraically encodes
Markovianity in an operator algebraic framework of noncommutative probability. Here the attribute
`noncommutative' is understood in the sense of `not-necessarily-commutative', and thus applies to
classical probability, free probability and, more generally, quantum probability. We limit our
considerations in this paper to  representations of the Thompson monoid $F^+$ which connect to 
unilateral Markov shifts and just note that, with some more effort, representations of the Thompson
group $F$ can be seen to connect to bilateral Markov shifts \cite{KK22}. Further generalizations of 
our approach to noncommutative random fields are possible and of interest for future investigations.

Notably, our results are also applicable to subfactor theory where independently, Vaughan Jones
constructed and studied representations for the Thompson group $F$, within the ongoing effort to
obtain a conformal field theory for every finite index subfactor \cite{Jo17,Jo18a,Jo18b,BJ19a,
BJ19b,AJ21}. How Vaughan Jones' approach and our approach are related is of high interest to be 
further investigated, but beyond the scope of the present paper. Furthermore, our approach invites
further investigations of the representation theory of the Thompson group $F$, and its relatives, 
from the probabilistic viewpoint suggested by our new approach. 

To increase the accessibility of our approach, let us focus in this introduction first on
classical probability where  distributional symmetries and invariance principles are known
to provide deep structural results on stochastic processes \cite{Ka05}. Its subject of research
emerged around the early 1930s when de Finetti characterized an exchangeable infinite sequence of
$\{0,1\}$-valued random variables as a mixture of independent identically distributed (i.i.d.) random
variables. This foundational result was extended to random variables with values in a compact
Hausdorff space by Hewitt and Savage \cite{HS55}, and to values in a Borel space by Aldous
\cite{Al85}. Furthermore, Ryll-Nardzewski \cite{RN57} established that the apparently weaker
distributional symmetry of spreadability (also known as `contractability' in the literature) is
equivalent to exchangeability for infinite sequences. 

Already de Finetti considered `partial exchangeability (of the Markov type)' or `Markov
exchangeability', suggesting that this distributional symmetry should characterize a mixture of
infinite sequences of Markov chains (see also \cite{Di88}). Initially Freedman \cite{Fr62} established 
such a characterization for stationary sequences, taking values in a countable set. Consecutively, 
Diaconis and Freedman relaxed the stationarity condition to that of recurrence, to arrive in \cite{DF80}
at a de Finetti theorem for infinite Markov chains. In particular, they show that partial
exchangeability can be implemented by `block-switch transformations' on the path space of the 
stochastic process. A different and more sophisticated approach was taken by Kallenberg for sequences of
random variables with values in an arbitrary measurable space \cite{Ka82}. Invoking stopping time 
arguments, he characterized recurrent `locally $\cF$-homogeneous' sequences to be `conditionally
$\cF$-Markovian' (see \cite[Theorem 2.4]{Ka82}). Furthermore, for random variables with values
in finite or countable sets, de Finetti-type characterization results in terms of hidden Markov models 
are also available for certain exchangeable sequences \cite{Dh63a,Dh63b,Dh64} and, more recently, 
for certain partially exchangeable sequences \cite{FP09,FP16}. 

Quite recently, it was realized in \cite{EGK17} that actually not exchangeability, but spreadability 
is the fundamental distributional invariance principle from the viewpoint of algebraic topology 
and homological algebra. This insight is little apparent from the common definition of
spreadability which states that joint distributions of a sequence of (noncommutative) random variables
are invariant upon passing to a subsequence (see Definition \ref{definition:distsym}). Indeed,
\cite[Theorem 1.2]{EGK17} establishes a new equivalent characterization of spreadability which connects
it to the representation theory of the `partial shifts monoid'
\[
S^{+}:=\langle h_0,h_1,h_2,\ldots \mid h_{k}h_{\ell}=h_{\ell+1}h_{k} 
\text{ for } 0\leq k\le \ell <\infty \rangle^+. 
\]
Roughly phrasing, $S^{+}$ algebraically encodes conditional independence in classical
probability. More precisely -- in fact, this can be taken as an alternative definition -- 
a sequence of random variables $\xi \equiv (\xi_n)_{n \ge 0}$ is spreadable if and only if 
there exists a representation $\varrho_*$ of $S^+$ in the measure-preserving measurable maps 
of the underlying probability space such that
\[
\xi_0 = \xi_0 \circ \varrho_*(h_n), \qquad \xi_n = \xi_0 \circ \varrho_*(h_0^n)
\]
for all $n \ge 1$. This adds to the extended de Finetti theorem a new equivalent characterization of
conditionally independent, identically distributed random variables (see \cite[Theorem 1.2]{EGK17}, 
which is partly stated as Theorem \ref{theorem:cdeF} for the convenience of the reader). 

This alternative definition of spreadability deserves some elaboration. Commonly, a probabilistic 
approach implements distributional symmetries and invariance principles in terms of actions on the
index set of a sequence of random variables. For example, exchangeability requires that 
the distribution of the sequence is invariant under permuting the indices of the sequence, and 
spreadability requires that the distribution is invariant under passing to a subsequence. Here we take
the alternative point of view that the distributional invariance principle is implemented through
a group or a monoid acting on the underlying probability space, as illustrated above for spreadability 
by the partial shifts monoid $S^+$. In fact, such a point of view initiated the development of 
`braidability' in \cite{GK09} as a new distributional symmetry which is intermediate to exchangeability 
and  spreadability in noncommutative probability, and this was further investigated for braided 
parafermions in \cite{BJLW20}. Also, the introduction of quantum exchangeability in \cite{KS09} 
was inspired by this point of view when replacing the action of permutations by co-actions of quantum
permutations.

This alternative definition of spreadability through measure-preserving measurable actions of  
$S^+$ on a probability space is the starting point of our investigations on possible
generalizations of this distributional invariance principle. It is an elementary observation
that the monoid $S^+$ is a quotient of the Thompson monoid 
\[
F^{+}=\langle g_0,g_1,g_2,\ldots \mid g_{k}g_{\ell}=g_{\ell+1}g_{k} 
\text{ for } 0\leq k < \ell <\infty \rangle^+, 
\]
as relations of the form $g_k g_{k} = g_{k+1} g_{k}$ are absent in the presentation of
$F^+$.\footnote{This definition of the Thompson monoid $F^+$ differs from the usual one used
in the literature (see also Subsection \ref{subsection:basics-on-F-F+}). Our choice of
presentation is motivated from the elementary observation that the monoid $S^+$ is a quotient 
of $F^+$.} 
Thus it is intriguing to ask if the representation theory of $F^+$ is connected to a novel
distributional invariance principle which characterizes a larger class of random objects than
those characterized by spreadability. In other words:
\begin{center}
    \emph{Can one characterize a `partially spreadable' sequence of random variables $\xi$?} 
\end{center}
Here `partial spreadability' means by definition that the sequence $\xi \equiv (\xi_n)_{n \ge 0}$
satisfies
\[
\xi_0 = \xi_0 \circ \rho_*(g_n), \qquad  \xi_n = \xi_0 \circ \rho_*(g_0^n)
\]
for all $n \ge 1$, where now $\rho_*$ denotes a representation of the Thompson monoid $F^+$ in
the measure-preserving measurable maps on the underlying probability space, as introduced in
Definition \ref{definition:partspread-c}. We give an affirmative answer to this question in the 
following theorem, the proof of which is given in Subsections \ref{subsection:DIPinCP} and
\ref{subsection:constr-class-prob}, where $(\Omega_0, \Sigma_0)$ denotes a standard Borel space. 
\begin{Theorem}\label{theorem:definetti-intro}
Let $\xi \equiv (\xi_n)_{n\in\Nset_0}$ be a sequence of $(\Omega_0, \Sigma_0)$-valued random
variables on the standard probability space $(\Omega,\Sigma,\mu)$. Then the following are
equivalent:
\begin{enumerate}
     \item[(a)] $\xi$ is maximal partially spreadable;
     \item[(b)] $\xi$ is a stationary Markov sequence.
\end{enumerate}
\end{Theorem}
It is worthwhile to point out for the implication `$(a) \Rightarrow (b)$' that, in general, a partially
spreadable sequence may fail to be Markovian. This motivates the introduction of the strengthened
notion of `maximal partial spreadability' in Definition \ref{definition:partspread-c}. This maximality 
condition requires that the random variable $\xi_0$ generates a $\sigma$-subalgebra of $\Sigma$ which 
equals the intersection of all $\mu$-almost $\rho_*(g_n)$-invariant sets for $n \ge 1$. It takes into 
account the fact that a function of a stationary Markov sequence is still partially spreadable (see
Theorem \ref{theorem:functionMarkovPC}), but may be non-Markovian.

The proof of the converse implication `$(b) \Rightarrow (a)$' requires us to find a representation of
the Thompson monoid $F^+$ which cannot be easily identified when the underlying probability space
is given by the Daniell-Kolmogorov construction for the Markov process. But a natural choice for
such a representation exists when the Markov process is written as an open dynamical system.  Such 
an approach is common for the construction of quantum Markov processes, as pioneered in \cite{AFL82}.
Actually, our setting requires an operator algebraic approach to noncommutative
stationary Markov processes as presented by K\"ummerer \cite{Ku85,Ku86} and 
used by Anantharaman-Delaroche \cite{AD06}, as well as Haagerup and Musat \cite{HM11}.
Furthermore, we refer the interested reader to Attal's expository paper \cite{At10}
for detailed information on the open dynamical system viewpoint for (possibly
non-stationary) classical Markov sequences. 

Having exemplified by Theorem \ref{theorem:definetti-intro} how Markovianity and representations
of the Thompson monoid $F^+$ relate in classical probability, next we turn our attention to developing 
noncommutative versions of this relation. In parts, our approach is stimulated by the 
significant progress made over the past decade in establishing noncommutative de Finetti
type results in the context of free independence 
\cite{Cu09,KS09,Cu10,Cu11,BCS12,DDM14,DK14,FW16,DKW17,DM17}, 
Boolean or monotone independence \cite{Li15,Li18,Li19,CFG20}, fermionic systems
\cite{BJLW20,CF12,CRZ22,Fi22a,Fi22b}, and general operator algebraic settings
\cite{AGCL08,GK09,GK10,Ko10,GK12,CF15,EGK17}.

An immediate idea is to look for a direct transfer of the de Finetti theorem of Diaconis and Freedman
\cite{DF80} or of Kallenberg's characterization in \cite{Ka82} to a noncommutative setting. Such a 
direct route presently does not seem to be available, as the first approach requires the availability 
of certain pathwise transformations and the latter utilizes stopping time arguments. But combining our
approach to distributional invariance principles  and the open dynamical system formulation of 
stationary Markov processes, we are able to suitably transfer Theorem \ref{theorem:definetti-intro} 
to noncommutative probability.
  
We now address our de Finetti type results obtained in the context of `partial spreadability' 
in an operator algebraic setting.  We refer the reader to Section \ref{section:preliminaries} 
for definitions and notation as we will use them below. Here we just remind that the pair
$(\cM,\psi)$ denotes a (noncommutative) probability space which consist of a von Neumann algebra
$\cM$ and a faithful normal state $\psi$ on $\cM$. Such pairs are also known as W*-algebraic
probability spaces in the literature. Furthermore, a noncommutative random variable $\iota_0$ 
from the probability space $(\cA,\varphi)$ into the probability space $(\cM,\psi)$ is given by 
an injective *-homomorphism $\iota_0 \colon \cA \to \cM$ such that $\psi \circ \iota_0 = \varphi$,
and written as $\iota_0 \colon (\cA,\varphi) \to (\cM,\psi)$. (Here we have omitted addressing a
modular condition as done in Subsection \ref{subsection:Random Variables and Invariance Principles}.)
Finally $\End(\cM,\psi)$ denotes certain unital *-homomorphisms on $\cM$, as further detailed in
Definition \ref{definition:endomorphism}. 

Let us comment on our choice of a W*-algebraic setting where von Neumann algebras
are equipped with faithful normal states. Instead, considering unital C*-algebras with states 
as a starting point may seem more appropriate in the context of noncommutative de Finetti theorems.
For example, St{\o}rmer characterized in \cite{St69} symmetric states on the infinite (minimal)
tensor product of a unital C*-algebra with itself. As another more recent example, quantum symmetric
states were characterized in \cite{DKW17} on the infinite universal free product of a unital
C*-algebra with itself. Even though these free or tensor product constructions provide 
noncommutative versions of the Daniell-Kolmogorov construction, they give little
insight on how to associate a representation of the Thompson monoid $F^+$ to a noncommutative
stationary Markov process, for the same reasons as in the classical theory. As the primary goal 
of the present paper is to establish the connection between Markovianity and `probabilistic'
representations of $F^+$, we chose von Neumann algebras with faithful normal states as our 
starting point. This avoids an additional technical overhead, to improve the transparency of our 
new approach, and postpones C*-algebraic approaches to future publications.  

To provide some context to our main results, let us first recall the following definition of
spreadability as a distributional invariance principle which was identified in \cite{Ko10,EGK17}
to be equivalent to its more traditional formulation (as stated in Definition \ref{definition:distsym}(ii)). 
\begin{Definition}[\cite{Ko10,EGK17}]\normalfont \label{definition:spreadable}
A sequence of random variables $\iota \equiv (\iota_{n})_{n\geq 0}: (\cA,\varphi)\to (\cM,\psi)$ 
is \emph{spreadable} if there exists a representation  $\varrho:S^{+}\to \End(\cM,\psi)$ such
that the following localization and stationarity properties are satisfied:
\begin{align}
&&&&\iota_0 &= \varrho(h_n) \iota_0 
&&\text{for all $n \ge 1$};&&&&\tag{L} \label{eq:L-spread}\\
&&&& \iota_n    & = \varrho(h_0^n) \iota_0&& \text{if $n \ge 0$.}   
&&&&  \tag{S} \label{eq:S-spread}
\end{align}
More generally, $\iota$ is said to be \emph{spreadable} if there exists a spreadable sequence 
$\tiota$ such that $\iota \stackrel{\distr}{=}\tiota$. 
\end{Definition}
Replacing the role of the partial shift monoid $S^+$ by the Thompson monoid $F^+$, we are now in
the position to introduce a natural generalization of spreadability as a new distributional
invariance principle. 
\begin{Definition}\normalfont \label{definition:p-s}
A sequence of random variables $\iota \equiv (\iota_{n})_{n\geq 0}: (\cA,\varphi)\to (\cM,\psi)$ 
is \emph{partially spreadable} if there exists a representation  $\rho:F^{+}\to \End(\cM,\psi)$ 
such that the following localization and stationarity properties are satisfied:
\begin{align}
&&&&\iota_0 &= \rho(g_n) \iota_0 
&&\text{for all $n \ge 1$};&&&&\tag{L} \label{eq:L-thompson}\\
&&&& \iota_n    & = \rho(g_0^n) \iota_0&& \text{if $n \ge 0$.}   
&&&&  \tag{S} \label{eq:S-thompson}
\end{align}
More generally, $\iota$ is said to be \emph{partially spreadable} if there exists a partially
spreadable sequence $\tiota$ such that $\iota \stackrel{\distr}{=}\tiota$. 
\end{Definition}
Clearly spreadability implies partial spreadability. Of course, the crucial question is if
partial spreadability allows to develop similar results of de Finetti type as it is the case
for spreadability, especially in the general framework of noncommutative probability. As detailed
in \cite{Ko10}, conditional independence in classical probability is generalized to a geometric
notion which we call here `conditional CS independence' (see Definition
\ref{definition:order/full-independence}) and which is intimately related to Popa's notion of
commuting squares in subfactor theory. This geometric viewpoint on a very general notion of
noncommutative independence emerged out of the investigations of K\"ummerer on the structure of
noncommutative stationary Markov processes \cite{Ku85} and is further justified by the following
noncommutative extended de Finetti theorem.  
\begin{Theorem}[\cite{Ko10}]\label{theorem:ncdf} 
Let $\iota \equiv (\iota_n)_{n\ge 0} \colon (\cA,\varphi) \to (\cM,\psi)$ be a sequence of 
random variables and consider the following conditions:
\begin{enumerate}
\item[(a)]  $\iota$ is spreadable;  
\item[(b)]  $\iota$ is stationary and conditionally CS independent;
\item[(c)]  $\iota$ is identically distributed and conditionally CS independent.
\end{enumerate}
Then one has the following implications:
\[
\text{(a)}   \Longrightarrow \text{(b)}   \Longrightarrow \text{(c)}.   
\]
\end{Theorem}
We refer the interested reader to the introduction of \cite{Ko10} and to \cite{GK09,GK12} to
learn more on why one should not expect an equivalence of these three statements in the general
framework of noncommutative probability, in contrast to the situation of classical probability or
free probability.  Replacing spreadability by partial spreadability we have succeeded to
establish the following main result of de Finetti type.
\begin{Theorem}\label{theorem:extended-de-Finetti}
Let $\iota \equiv (\iota_n)_{n\ge 0} \colon (\cA,\varphi) \to (\cM,\psi)$ be a sequence of 
random variables and consider the following conditions:
\begin{enumerate}
\item[(a)] $\iota$ is partially spreadable;  
\item[(b)]  $\iota$ is stationary and adapted to a local Markov filtration;
\item[(c)]  $\iota$ is identically distributed and adapted to a local Markov filtration.
\end{enumerate}
Then one has the following implications:
\[
\text{(a)}   \Longrightarrow \text{(b)}   \Longrightarrow \text{(c)}.   
\]
\end{Theorem}
Here the notion of `adaptedness to a local Markov filtration' is cast in Definition \ref{definition:markovianity}.
The proof of this main result is the subject of Subsection \ref{subsection:ncdf}. Clearly, the
converse implication from (c) to (b) fails for the obvious reason that an identically distributed
sequence may not be stationary. Also one should not expect that (b) implies (a) in the general
noncommutative setting, for similar reasons of the corresponding failure in Theorem
\ref{theorem:ncdf}. At the time of this writing we do not know if the conditions (a) and (b) are
equivalent in the commutative setting corresponding to classical probability.

Commonly Markovianity is introduced in the 
literature as a property of a stochastic process with respect to a filtration. Our approach is slightly 
more general in two aspects. Similar to random fields, our approach considers `local filtrations' which 
are \emph{partially ordered} families of von Neumann subalgebras, 
or $\sigma$-subalgebras, indexed by `discrete time intervals'. Furthermore these `local filtrations' 
allow us to introduce Markov properties directly, \emph{without} reference to random variables 
(see Definition \ref{definition:markov-filtration} and Definition \ref{definition:markov-filtration-c}, 
respectively).

Let us stress that, as already stated in 
the setting of classical probability (see Subsection \ref{subsection:FunctionsOfMarkov}), a partially 
spreadable sequence may fail to be a Markov sequence. (Loosely phrased, this is connected to the fact 
that commuting square properties of von Neumann subalgebras are not robust under restrictions, compare 
Remark \ref{remark:cs-not-robust}.) But Markovianity can be enforced by strengthening `partial
spreadability' to `maximal partial spreadability' (as done in Definition \ref{definition:max-p-s}).
Normally the random variable $\iota_0$ of a partially spreadable sequence $\iota$ satisfies only
the localization property 
\[
\iota_0(\cA) \subset \bigcap_{n \ge 1} \cM^{\rho(g_n)}, 
\]
where $\cM^{\rho(g_n)} = \{x \in \cM \mid \rho(g_n) (x) = x\}$ denotes the fixed point von
Neumann subalgebra of the represented generator $\rho(g_n)$. This inclusion condition is now
strengthened to the equality
\[
\iota_0(\cA) = \bigcap_{n \ge 1} \cM^{\rho(g_n)}
\]
for the definition of a maximal partially spreadable sequence. Consequently, one arrives at
Theorem \ref{theorem:maxps} which, for the convenience of the reader, we state here slightly
reformulated and aligned with the formulation of the classical de-Finetti-type result in Theorem
\ref{theorem:definetti-intro}:
\begin{Theorem}
Let $\iota$ be a sequence as in Theorem \ref{theorem:extended-de-Finetti} and consider the following
conditions:
\begin{enumerate}
\item[(a)] $\iota$ is maximal partially spreadable;  
\item[(b)]  $\iota$ is a stationary Markov sequence. 
\end{enumerate}
Then the implication $\text{(a)}   \Longrightarrow \text{(b)}$ is valid.   
\end{Theorem}
Again, one cannot expect the converse implication to be true in the full generality of the
noncommutative setting, as reminded above already.   

We are left to outline the content of this paper. Most of our main results are proven in an
operator algebraic framework of noncommutative probability, as introduced in Section
\ref{section:preliminaries}. We provide definitions, notation and some background results on the
Thompson monoid $F^+$ in Subsection \ref{subsection:basics-on-F-F+}, and on the partial shift
monoid $S^+$ in Subsection \ref{subsection:partialshifts}. The basics of noncommutative
probability spaces and Markov maps are provided in Subsection \ref{subsection:Markov-maps}.
Subsection \ref{subsection:Random Variables and Invariance Principles} is devoted to
noncommutative random variables and established distributional invariance principles like
exchangeability, spreadability and stationarity. We present  in Subsection
\ref{subsection:Independence and Markovianity} a geometric notion of noncommutative conditional
independence which essentially resembles a structure known as commuting squares in subfactor
theory. Furthermore we provide the notion of a local Markov filtration which allows to define 
Markovianity on the level of von Neumann subalgebras without any reference to noncommutative random 
variables. Actually these structures underly K\"ummerer's approach to stationary Markov dilations and
can now be seen to emerge from distributional invariance principles in noncommutative
probability. Finally we provide some results on noncommutative stationary processes in Subsection
\ref{subsection:Noncommutative Stationary Processes}. Here we will meet the one-sided version of
noncommutative stationary Markov processes and Markov dilations in the sense of K\"ummerer
\cite{Ku85} and as they are considered within the context of factorizable Markov maps by Musat
and Haagerup \cite{HM11}. 

Section \ref{section:MarkovianityandF+Classical} is devoted to the introduction of `partial
spreadability' as a new distributional invariance principle in classical probability theory. 
Markovianity and representations of the Thompson monoid $F^+$ are addressed in the traditional
language of classical probability. Subsection \ref{subsection:CSMPR} sketches how one
arrives at noncommutative notions of probability spaces and random variables when starting from a
traditional setting. Subsection \ref{subsection:DIPinCP} presents the extended de Finetti
theorem, Theorem \ref{theorem:cdeF}, to create some relevant background for the introduction 
of the new distributional invariance principles of `partial spreadability' and `maximal partial
spreadability' in Definition \ref{definition:partspread-c}. Furthermore we introduce in Definition
\ref{definition:markov-filtration-c} the notion of a local Markov filtration which is slightly 
more general than those usually considered for Markov sequences, but quite familiar in the topic of 
random fields or generalized stochastic processes. This allows us to formulate  several de-Finetti-type 
results for (maximal) partially spreadable sequences, in particular Theorem 
\ref{theorem:cldeFi-with condition} and Theorem \ref{theorem:definetti-intro}, our main results of 
de Finetti type on the characterization of stationary Markov sequences in classical probability. We 
continue in Subsection \ref{subsection:FunctionsOfMarkov}  with providing evidence that the class of 
partially spreadable sequences is much larger than the class of (mixtures of) stationary Markov 
sequences. For this purpose, we discuss functions of stationary Markov chains, in particular in the
context of the algebraization procedure. This yields our main result in Theorem 
\ref{theorem:functionMarkovPC} that functions of partially spreadable sequences remain partially 
spreadable (but may no longer be Markovian, as it is well-known). Finally, we provide in Subsection 
\ref{subsection:markov-dilations} a general result on stationary Markov processes as open dynamical
systems, Theorem \ref{theorem:kuemmerer-onesided}. It builds on the usual Daniell-Kolmogorov
construction for such processes and underpins K\"ummerer's open dynamical systems approach to 
stationary Markov processes \cite{Ku85}. Our presentation follows closely the arguments given by
K\"ummerer in \cite{Ku86}. In particular, we will illustrate this approach by zero-one-valued 
stationary unilateral Markov sequences in Example \ref{example:M-D}. These results will be drawn upon 
in Subsection \ref{subsection:constr-class-prob} for the construction of a representation of the 
Thompson monoid $F^+$. Effectively, the availability of such an alternative construction entails that 
a stationary Markov process is (maximal) partially spreadable. This establishes one direction 
of the claimed equivalence in the de Finetti theorem, Theorem \ref{theorem:definetti-intro}.

Section \ref{section:markovianity-from-F+} investigates representations of the Thompson monoid
$F^+$ in the endomorphisms of noncommutative probability spaces. Subsection
\ref{subsection:generating} introduces the generating property of representations of $F^+$ in
Definition \ref{definition:generating}. This property ensures that the fixed point algebras of the
represented generators of $F^+$ form a tower which generates the noncommutative probability
space, see Proposition \ref{proposition:generating-property}. Effectively this tower of fixed
point algebras equips the noncommutative probability space with a filtration which, using actions
of the represented generators, can be further upgraded to become a local Markov filtration. We show in
Subsection \ref{subsection:s-f-p-a} that `shifted' fixed point algebras of the represented
generators of $F^+$ provide triangular towers of commuting squares. Local Markov filtrations are
obtained from this as a particular property, see Corollary \ref{corollary:markov-3}. 
Subsection \ref{subsection:Markov-F} considers certain noncommutative stationary processes which
are partially spreadable by construction. A main result is Theorem
\ref{theorem:markov-filtration-1} which establishes that such processes are adapted to a local Markov 
filtration. Finally, the proof of the de Finetti theorem for stationary Markov processes, Theorem
\ref{theorem:extended-de-Finetti}, is completed in Subsection \ref{subsection:ncdf}. 

Section \ref{section:Constructions of Reps of F+}  provides some elementary constructions of
representations of the Thompson monoid $F^+$. Here we focus on tensor product constructions in
Subsection \ref{subsection:tensor-product} and meet structures which are familiar from so-called
tensor dilations of Markov operators \cite{Ku85,Ku86}. When specializing these tensor product
constructions to commutative von Neumann algebras, we obtain representations of the Thompson
monoid $F^+$ in the setting of classical probability. This is the subject of Subsection
\ref{subsection:constr-class-prob}, where we will also complete the proof of the de Finetti
theorem for stationary Markov sequences in the classical setting, Theorem
\ref{theorem:definetti-intro}. Finally, we turn our attention in Subsection 
\ref{subsection:op-alg-construction} to constructions in the general framework of operator algebras. 
We introduce certain monoidal extensions of $F^+$ and $S^+$, and investigate their representation 
theory, to adapt and refine K\"ummerer's approach on noncommutative 
Markov processes as perturbations of noncommutative Bernoulli shifts \cite{Ku93}.   
\section{Preliminaries} 
\label{section:preliminaries}
\subsection{\texorpdfstring{The Thompson group $F$ and its monoid $F^+$}{}}
\label{subsection:basics-on-F-F+}
The Thompson group $F$, originally introduced by Richard Thompson in 1965 as a certain 
group of piece-wise linear homeomorphisms on the interval $[0,1]$, is known to have the 
infinite presentation
\begin{equation*}
F=\langle g_0,g_1,g_2,\ldots \mid g_{k}g_{\ell}=g_{\ell+1}g_{k} 
\text{ for } 0\leq k<\ell <\infty \rangle. 
\end{equation*}
We note that these generators $g_k$ of the group $F$ correspond to the inverses of generators
usually used in the literature (e.g. \cite{Be04}). We work throughout with this group
presentation as we are interested in the study of the following monoid. As the defining relations of
this presentation involve no inverse generators, one can associate
to it the monoid  
\begin{equation}\label{eq:F+}
F^{+}=\langle g_0,g_1,g_2,\ldots \mid g_{k}g_{\ell}=g_{\ell+1}g_{k} 
\text{ for } 0\leq k<\ell <\infty \rangle^+, 
\end{equation} 
henceforth referred to as the \emph{Thompson monoid} $F^+$. (We refer the reader to Remark \ref{remark:F^+} on how this monoid relates to the one usually introduced in the literature.)

An element $e \neq g \in F^+$ has the unique normal form
\begin{align} \label{eq:F-monoid-normal-form}
g = g_k^{a_k} g_{k-1}^{a_{k-1}} \cdots   g_1^{a_1} g_0^{a_0}, 
\end{align}
where $a_0, a_1, \ldots, a_{k-1},a_k \in \Nset_0$ with $a_k >0$ for some $k \in \Nset_0$ (see
\cite[Theorem 2.4.4]{Be04}, for example, bearing in mind that we use the relations of the
inverses).   
 
\begin{Definition} \normalfont \label{definition:mn-shift}
Let $m,n \in \Nset_0$ with $m \le n$ be fixed. The \emph{$(m,n)$-partial shift} $\sh_{m,n}$ is
the monoid endomorphism on $F^+$ defined by 
\[
\sh_{m,n}(g_k)  = \begin{cases}
                   g_m &\text{if $k=0$}\\
                   g_{n +k} &\text{if $k \ge 1$}.
                  \end{cases}
\]
\end{Definition}
Each map $\sh_{m,n}$ is well-defined as a monoid endomorphism as it preserves all defining relations of
$F^+$. In particular, the endomorphism $\sh_{1,1}$ satisfies $\sh_{1,1}(g_k) = g_{k+1}$ for all 
$k \ge 0$. This shift is also considered in \cite[Definition 1.6.5]{Be04}.
\begin{Lemma} \normalfont \label{lemma:m-n-shift}
The monoid endomorphisms $\sh_{m,n}$ on $F^+$ are injective for all $m,n \in \Nset_0$.  
\end{Lemma} \normalfont
\begin{proof}
Let $g\in F^+$ have the (unique) normal form as stated in \eqref{eq:F-monoid-normal-form}. Then 
\[
\sh_{m,n}(g)  =   g_{n+k}^{a_k} g_{n+k-1}^{a_{k-1}} \cdots   g_{n+1}^{a_1} g_{m}^{a_0}
\]
is again in normal form. The injectivity of the map $\sh_{m,n}$ is concluded from the uniqueness
of the normal form. 
\end{proof}
\begin{Remark} \label{remark:F^+} \normalfont
We emphasize that $F^+$ differs from the monoid
usually considered in the literature, due to our choice of generators for $F$. Nevertheless,
this monoid $F^+$ is left and right cancellative, and any two elements of $F^+$ admit a left common 
multiple. Thus $F^+$ can be identified with a monoid in the Thompson group $F$ such that any element
$g$ of $F$ can be written as $g = s^{-1}t$, with $s, t \in F^+$. We refer the reader to \cite{DT19} for
further details on related constructions. Finally, we remark that, alternatively, the generators of the
monoid $F^+$ can be obtained as morphisms (in the inductive limit) of the category of finite binary
forests, similarly as done in \cite{Be04, Bu04, Jo18a}, for example. 
\end{Remark}
\subsection{\texorpdfstring{The partial shifts monoid $S^+$.}{}} 
\label{subsection:partialshifts}
Stipulating additional relations to those for the generators in \eqref{eq:F+}, one obtains the
monoid
\begin{equation*}
S^{+}=\langle h_0,h_1,h_2,\ldots \mid h_{k}h_{\ell}=h_{\ell+1}h_{k} 
\text{ for } 0\leq k\le \ell <\infty \rangle^+ 
\end{equation*} 
as a quotient of the Thompson monoid $F^+$. The monoid $S^+$ is \emph{not} a submonoid of the
group with infinite presentation $\langle h_0,h_1,h_2,\ldots \mid h_{k}h_{\ell} = h_{\ell+1}
h_{k} \text{ for } 0\leq k\le \ell <\infty \rangle$, as the latter is isomorphic to the additive
group $\Zset$. Actually the generators $h_k$ of the monoid $S^+$ satisfy relations as they are
familiar for coface maps in simplicial cohomology.  As discussed in \cite{EGK17}, these
generators arise as morphisms in the direct limit of the semi-cosimplicial category $\Delta_S$.
This becomes evident when considering the following representation of $S^+$ which in particular
motivates to address $S^+$ as the \emph{partial shifts monoid} and its generators $h_k$ as
\emph{partial shifts}. 

It is also worthwhile to remark that while $S^+$ is a quotient of the monoid $F^+$, it is \emph{not} 
a sub-monoid of $F^+$. In particular, $F^+$ is cancellative (see \cite{DT19}), whereas $S^+$ is not. 
For instance $h_k h_k= h_{k+1} h_k$ does not imply that $h_k = h_{k+1}$. 
\begin{Lemma}[Partial shifts] \label{definition:partial-shift}
The maps $(\theta_k)_{k \ge 0} \colon \Nset_0 \to \Nset_0$, defined by
\[
\theta_k(n) = \begin{cases} 
               n+1 & \text{ if $n \ge  k$}, \\
               n & \text{ if $n < k$},
              \end{cases}
\]
satisfy the relations $\theta_k \theta_{\ell} = \theta_{\ell+1} \theta_k$ for 
$0 \le k \le \ell < \infty$.  
\end{Lemma}
\subsection{Noncommutative probability spaces and Markov maps}
\label{subsection:Markov-maps}
Throughout, a \emph{(noncommutative) probability space} $(\cM,\psi)$ consists of a 
von Neumann algebra $\cM$ and a faithful normal state $\psi$ on $\cM$. The identity 
of $\cM$ will be denoted by $\1_{\cM}$, or simply by $\1$ when the context is clear.
Throughout, $\bigvee_{i\in I}\cM_i$ denotes the von Neumann algebra generated by the 
family of von Neumann algebras $\{\cM_i\}_{i\in I} \subset \cM$ for $I\subset \Nset_0$.
If $\cM$ is abelian and acts on a separable Hilbert space, then $(\cM,\psi)$ is
isomorphic to $\big(L^\infty(\Omega, \Sigma, \mu), \int_{\Omega}  \cdot \,\, d\mu\big)$
for some standard probability space $(\Omega, \Sigma, \mu)$.
\begin{Definition} \normalfont \label{definition:endomorphism}
An \emph{endomorphism} $\alpha$ of a probability space $(\cM,\psi)$ is a $*$- homomorphism on $\cM$ 
satisfying the following additional properties:
\begin{enumerate}
    \item $\alpha(\1_{\cM})=\1_{\cM}$ (unitality);
    \item  $\psi\circ \alpha=\psi$ (stationarity);
    \item $\alpha$ and the modular automorphism group $\sigma_t^{\psi}$ commute 
           for all $t\in \Rset$ (modularity).
\end{enumerate}
The set of endomorphisms of $(\cM,\psi)$ is denoted by $\End(\cM,\psi)$. We note that an
endomorphism of $(\cM,\psi)$ is automatically injective. Similarly, $\Aut(\cM,\psi)$ denotes 
the automorphisms of  $(\cM,\psi)$.
\end{Definition}
\begin{Definition}\label{definition:MarkovMap}\normalfont
Let $(\cM,\psi)$ and $(\cN,\varphi)$ be two noncommutative probability spaces. A linear map  
$T \colon \cM \to \cN$ is called a \emph{$(\psi,\varphi)$-Markov map} if the following 
conditions are satisfied:
\begin{enumerate}
\item \label{item:mm-i}
$T$ is completely positive;  
\item \label{item:mm-ii}
$T$ is unital;                  
\item \label{item:mm-iii}
$\varphi \circ T = \psi$;
\item \label{item:mm-iv}
$T \circ \sigma_t^{\psi}  = \sigma_t^{\varphi} \circ T$, for all $t \in \Rset$.
\end{enumerate}
\end{Definition}
Here $\sigma_{}^{\psi}$ and $\sigma_{}^{\varphi}$ denote the modular automorphism groups of 
$(\cM,\psi)$ and $(\cN,\varphi)$, respectively. If $(\cM,\psi) = (\cN,\varphi)$, we say that 
$T$ is a $\psi$-\emph{Markov map on $\cM$}. Conditions \eqref{item:mm-i} to \eqref{item:mm-iii} 
imply that a Markov map is automatically normal. The condition \eqref{item:mm-iv} is equivalent 
to the condition that a unique Markov map $T^* \colon (\cN,\varphi) \to (\cM,\psi)$ exists such
that
\[
\psi\big(T^*(y)x\big) = \varphi\big(y\, T(x)\big) \qquad (x \in \cM, y \in \cN).  
\]
The Markov map $T^*$ is called the \emph{adjoint} of $T$ and $T$ is called \emph{self-adjoint} if
$T=T^*$. We note that condition \eqref{item:mm-iv} is automatically satisfied whenever $\psi$ and
$\varphi$ are tracial, in particular for abelian von Neumann algebras $\cM$ and $\cN$. 
\begin{Remark} \normalfont
More generally in some literature, a Markov map $T$ on a von Neumann algebra $\cM$ is understood
to be a unital normal completely positive linear map from $\cM$ to itself. Thus the above notion
of a $\psi$-Markov map $T$ on $\cM$ is more restrictive, as the existence of a \emph{faithful}
normal state $\psi$ with the stationarity condition $\psi \circ T = \psi$ and a modular condition
are stipulated (as also done in \cite[Definition 1.1]{HM11}). In particular, as recurrence for
Markov maps in noncommutative probability is defined via support properties of stationary normal 
states (see \cite{GaKu12}), every non-zero orthogonal projection $p \in \cM$ is positive
recurrent with respect to a Markov map $T$ on $(\cM,\psi)$, as the faithful state $\psi$ has the
support projection $\1_\cM \in \cM$.
\end{Remark}
We recall for the convenience of the reader the definition of conditional expectations in the
present framework of noncommutative probability spaces.
\begin{Definition} \normalfont
Let $(\cM,\psi)$ be a noncommutative probability space, and $\cN$ be a von Neumann subalgebra of
$\cM$. A linear map $E: \cM\to \cN$ is called a \emph{conditional expectation} if it satisfies
the following conditions:
\begin{enumerate}
    \item $E(x)=x$ for all $x\in \cN$;
    \item $\|E(x)\|\leq \|x\|$ for all $x\in \cM$;
    \item $\psi\circ E=\psi$.
\end{enumerate}
\end{Definition}
Such a conditional expectation exists if and only if $\cN$ is globally invariant under the
modular automorphism group of $(\cM,\psi)$ (see \cite{Ta72}, \cite{Ta79} and \cite{Ta03}). The
von Neumann subalgebra $\cN$ is called $\psi$-conditioned if this condition is satisfied. Note
that such a conditional expectation is automatically normal and uniquely determined by $\psi$. In
particular, a conditional expectation is a Markov map and satisfies the module property
$E(axb)=aE(x)b$ for $a,b\in \cN$ and $x\in \cM$.
\subsection{Noncommutative random variables and distributional invariance
principles}\label{subsection:Random Variables and Invariance Principles}
Let $(\cA,\varphi)$ and $(\cM,\psi)$ be two probability spaces. A \emph{(noncommutative) random
variable} $\iota_0$ is an injective *-homomorphism $\iota_0 \colon \cA \to \cM$ satisfying two
additional properties:
\begin{enumerate}
\item $\varphi = \psi \circ \iota_0$;
\item $\iota_0(\cA)$ is $\psi$-conditioned. 
\end{enumerate}
A random variable will also be addressed as the mapping $\iota_0 \colon (\cA,\varphi) \to
(\cM,\psi)$. If $\tiota_0 \colon (\cA,\varphi) \to (\tcM,\tpsi)$ is another random variable, then
$\iota_0$ and $\tiota_0$ have the same moment sequence and thus are identically distributed.
Given the (identically distributed) sequence of random variables 
\[
\iota \equiv (\iota_{n})_{n \in \Nset_0} \colon (\cA,\varphi) \to (\cM,\psi),
\] 
the family $\cA_\bullet \equiv \{\cA_I\}_{I\subset \Nset_0}$,  with von Neumann subalgebras 
$\cA_I := \bigvee_{i \in I} \iota_i(\cA)$, is called the \emph{canonical local filtration (generated 
by $\iota$)}. The sequence $\iota$ is said to be \emph{minimal} if $\cA_{\Nset_0} = \cM$. 
A sequence $\iota$ can always be turned into a minimal sequence by restriction. If 
\[
\tiota \equiv (\tiota_{n})_{n \in \Nset_0} \colon (\cA,\varphi) \to (\tcM,\tpsi)
\] 
is another sequence of random variables, then $\iota$ and $\tiota$ are said to have the
\emph{same distribution}, in symbols 
$(\iota_0, \iota_1, \iota_2, \ldots) \stackrel{\distr}{=}(\tiota_0, \tiota_1, \tiota_2, \ldots)$
or just $\iota \stackrel{\distr}{=} \tiota$, if 
\[
\psi\big(\iota_{k_1}(a_1) \iota_{k_2}(a_2) \cdots \iota_{k_n}(a_n)\big) 
=  \tpsi\big(\tiota_{k_1}(a_1) \tiota_{k_2}(a_2) \cdots \tiota_{k_n}(a_n)\big)
\]
for all $k_1, k_2, \ldots, k_n \in \Nset_0$, $a_1, a_2, \ldots, a_n \in \cA$ and $n \in \Nset$.  
\begin{Definition}\normalfont\label{definition:distsym}
The sequence of random variables $\iota \equiv (\iota_{n})_{n \in \Nset_0} \colon (\cA,\varphi)
\to (\cM,\psi)$ is said to be 
\begin{enumerate}
\item 
\emph{stationary} if $(\iota_0, \iota_1, \iota_2, \ldots) \stackrel{\distr}{=}
(\iota_{n} , \iota_{n+1}, \iota_{n+2}, \ldots)$  for all $n \in \Nset$; 
\item 
\emph{spreadable} if $(\iota_0, \iota_1, \iota_2, \ldots) \stackrel{\distr}{=}
(\iota_{n_0} , \iota_{n_1}, \iota_{n_2}, \ldots)$  for any increasing subsequence 
$(n_0, n_1, n_2, \ldots)$ of $(0,1,2, \ldots)$;
\item 
\emph{exchangeable} if $(\iota_0, \iota_1, \iota_2, \ldots) \stackrel{\distr}{=}
(\iota_{\sigma(0)} , \iota_{\sigma(1)}, \iota_{\sigma(2)}, \ldots)$  for all permutations
$\sigma\in \Sset_\infty$.  
\end{enumerate}
Here $\Sset_\infty$ denotes the group of all finite permutations on the set $\Nset_0$ such that
the Coxeter generator $\sigma_k \in \Sset_\infty$ is the transposition $(k-1,k)$.   
\end{Definition}
It is elementary to verify that one has the following hierarchy of invariance principles:
\[
\text{exchangeability} 
\quad \Longrightarrow \quad \text{spreadability} 
\quad \Longrightarrow \quad \text{stationarity}. 
\]
These three distributional invariance principles can be equivalently reformulated. 
\begin{Proposition}\label{proposition:dip}
Suppose $\iota \equiv (\iota_n)_{n \in \Nset_0} \colon  (\cA,\varphi) \to (\cM,\psi)$ is a
minimal sequence of random variables.
\begin{enumerate}
\item \label{proposition:dip-item-1}
The sequence $\iota$ is stationary if and only if there exists $\alpha \in \End(\cM,\psi)$ such
that $\iota_n = \alpha^n \iota_0$ for all $n \in \Nset$.
\item \label{proposition:dip-item-2}
The sequence $\iota$ is spreadable if and only if there exists a representation $\varrho \colon
S^+ \to \End(\cM,\psi)$ such that $\iota_0 = \varrho(h_k)\iota_0$ for all $k \ge 1$ and $\iota_n
= \varrho(h_0^n) \iota_0$ for all $n \in \Nset$. 
\item \label{proposition:dip-item-3}
The sequence $\iota$ is exchangeable if and only if there exists exists a representation
$\rho_{\perm}\colon \Sset_\infty \to \Aut(\cM,\psi)$ such that $\iota_0 =
\rho_{\perm}(\sigma_k)\iota_0$ for $k \ge 1$ and $\iota_n = \rho_{\perm}(\sigma_n \sigma_{n-1}
\cdots \sigma_1) \iota_0$ for all $n \in \Nset$.   
\end{enumerate}
\end{Proposition}
\begin{proof}
For \eqref{proposition:dip-item-1} see \cite{Ko10}. For \eqref{proposition:dip-item-2} see \cite{Ko10,EGK17}.
For \eqref{proposition:dip-item-3} see \cite{GK09}.
\end{proof}
The equivalent formulation of spreadability in \eqref{proposition:dip-item-2} and the simple observation
that $S^+$ is a quotient of the Thompson monoid $F^+$ catalyzed our introduction of partial
spreadability in Definition \ref{definition:p-s} as a novel distributional invariance principle. This
implies the extended hierarchy:
\begin{align*}
\text{exchangeability} 
\quad &\Longrightarrow \quad \text{spreadability} 
\quad \Longrightarrow \quad \text{partial spreadability} \\
\quad &\Longrightarrow \quad \text{stationarity}. 
\end{align*}   
\subsection{Noncommutative independence  and Markovianity} 
\label{subsection:Independence and Markovianity}
Out of K\"um\-merer's investigations on the structure of noncommutative Markov dilations (see for
example, \cite{Ku85} and \cite{Ku86}), it emerged that Popa's geometric notion of commuting
squares provides a rich framework for noncommutative independence (see \cite{Po89}).  This notion also
manifests itself in the noncommutative extended de Finetti theorem, Theorem
\ref{theorem:ncdf}. After having introduced commuting squares of von Neumann algebras and some of
their properties, as they are well-known in subfactor theory, we reinterpret these geometric
objects from the viewpoint of noncommutative probability theory, to define (conditional)
commuting square (CS) independence. More generally, we use commuting square structures to
introduce in Definition \ref{definition:markovianity} the notion of a `local Markov filtration' in the
noncommutative setting and relate it to more traditional notions of Markovanity for
(noncommutative) stochastic processes. Our approach is motivated by K\"ummerer's notion of a
Markov dilation (see \cite[Subsection 2.2]{Ku85} or for example \cite[Section 4]{HM11})
and furthermore supported by our investigations on distributional invariance principles emerging
from the Thompson monoid $F^+$. 

We recall some equivalent properties as they serve to define commuting squares in subfactor
theory (see for example \cite{GHJ89,JS97}) and as they are familiar from conditional independence
in classical probability. 
\begin{Proposition}\label{proposition:cs}
Let $\cM_0, \cM_1, \cM_2$ be $\psi$-conditioned von Neumann subalgebras of the probability space
$(\cM,\psi)$ such that $\cM_0 \subset (\cM_1 \cap \cM_2)$. Then the following are equivalent:
\begin{enumerate}
\item \label{item:cs-i}
$E_{\cM_0}(xy) = E_{\cM_0}(x) E_{\cM_0}(y)$ for all $x \in \cM_1$ and $y\in \cM_2$;   
\item  \label{item:cs-ii}
$E_{\cM_1} E_{\cM_2} = E_{\cM_0}$;
\item  \label{item:cs-iii}
$E_{\cM_1}(\cM_2) = \cM_0$;
\item  \label{item:cs-iv}
$E_{\cM_1} E_{\cM_2} =   E_{\cM_2} E_{\cM_1}$  and $\cM_1\cap \cM_2 = \cM_0$. 
\end{enumerate}
In particular, it holds that $\cM_0 = \cM_1 \cap \cM_2$ if one and thus all of these four
assertions are satisfied. 
\end{Proposition}
\begin{proof}
The tracial case for $\psi$ is proved in \cite[Prop.~4.2.1.]{GHJ89}. The non-tracial case follows
from this, after some minor modifications of the arguments therein.  
\end{proof}
\begin{Definition}\normalfont
The inclusions 
\[
\begin{matrix}
\cM_2 &\subset &\cM\\
\cup  &        & \cup \\
\cM_0   & \subset  & \cM_1
\end{matrix}
\]
as given in Proposition \ref{proposition:cs} are said to form  a \emph{commuting square (of von
Neumann algebras)}  if one (and thus all) of the equivalent conditions \eqref{item:cs-i} to
\eqref{item:cs-iv} are satisfied in Proposition \ref{proposition:cs}.
\end{Definition}
\begin{Notation} \normalfont \label{notation:indexsets}
We write $I < J$ for two subsets $I, J \subset \Nset_0$ if $i < j$ for all $i \in I$ and $j \in 
J$. The cardinality of $I$ is denoted by $|I|$. For $N \in  \Nset_0$, we denote by $I + N$ the
shifted set $\{i + N \mid i \in  I\}$. Finally, $\cI(\Nset_0)$ denote set of all `intervals' of 
$\Nset_0$, i.e.~sets of the form $[m,n] := \{m, m+1, \ldots, n\}$ or $[m,\infty) :=  \{m, m+1,
\ldots\}$ for $0 \le m \le n < \infty$.  
\end{Notation}
\begin{Definition}\normalfont
Let $(\cM,\psi)$ be a probability space with three $\psi$-conditioned von Neumann subalgebras
$\cM_0$, $\cM_1$ and $\cM_2$.  Then $\cM_1$ and $\cM_2$ are said to be \emph{CS independent over
$\cM_0$} or \emph{conditionally CS independent} if the inclusions
\[
\begin{matrix}
\cM_2 \vee \cM_0 &\subset &\cM\\
\cup  &        & \cup \\
\cM_0   & \subset  & \cM_1 \vee \cM_0
\end{matrix}
\]
form a commuting square. 
\end{Definition}
The inclusion $\cM_0 \subset (\cM_1 \cap \cM_2)$ is \emph{not} assumed in this definition, and
its failure occurs frequently in the context of distributional invariance principles, see for
example \cite[Example 4.6]{Ko10}.
\begin{Definition}\normalfont \label{definition:order/full-independence}
Let $\cN$ be a von Neumann subalgebra of $(\cM,\psi)$.  A family of von Neumann subalgebras
$\{\cA_n\}_{n \in \Nset_0}$ of $(\cM,\psi)$ is called  
\begin{enumerate}
\item
\emph{order CS independent over $\cN$} if $\bigvee_{i \in I} \cA_i$ and $\bigvee_{j \in J} \cA_j$
are CS independent over $\cN$ for any $I, J \subset \Nset_0$ with $I < J $ or $J < I$;   
\item
\emph{full CS independent over $\cN$} if $\cA_I$ and $\cA_J$ are CS independent over $\cN$ for
any $I, J \subset \Nset_0$ with $I \cap J  = \emptyset$.   
\end{enumerate}
The sequence of random variables $\iota \equiv (\iota_n)_{n \in \Nset_0} \colon (\cA, \varphi)
\to (\cM,\psi)$ with canonical local filtration $\{\cA_I\}_{I \subset \Nset_0}$ is called
\begin{enumerate}
\item[(i')]
\emph{order CS independent over $\cN$} if $\cA_I$ and $\cA_J$ are CS independent over $\cN$ for
any $I, J \subset \Nset_0$ with $I < J$ or $J < I$;   
\item[(ii')]
\emph{full CS independent over $\cN$} if $\cA_I$ and $\cA_J$ are CS independent over $\cN$ for
any $I, J \subset \Nset_0$ with $I \cap J  = \emptyset$.   
\end{enumerate}
\end{Definition}
Clearly, full CS independence over $\cN$ implies order CS independence over $\cN$. Occasionally
both of them will be just addressed as conditional CS independence. The noncommutative extended
de Finetti theorem, Theorem \ref{theorem:ncdf}, establishes that a spreadable sequence $\iota$
is full CS independence over the tail algebra $\cN= \bigcap_{n \ge 0} \bigvee_{k \ge n}^{}
\iota_k(\cA)$. 

We address next the basic notions of Markovianity in noncommutative probability. Commonly,
Markovianity is understood as a property of random variables relative to a filtration of the
underlying probability space. Our investigations from the viewpoint of distributional invariance
principles reveal that the phenomenon of `Markovianity' emerges without reference to any
stochastic process already on the level of a family of von Neumann subalgebras, indexed by the
partially ordered set of all `intervals' $\cI(\Nset_0)$. As commonly the index set of a filtration
is understood to be totally ordered \cite{Ve17} and guided by related notions for random Markov 
fields or generalized stochastic processes, we refer to such partially indexed families as `local 
filtrations'. This is motivated by the fact that sequences of random variables are simple 
examples of one-dimensional random fields. 
\begin{Definition}\normalfont
A family of $\psi$-conditioned von Neumann subalgebras $\cM_\bullet \equiv \{\cM_I\}_{I \in
\cI(\Nset_0)}$ of the probability space $(\cM,\psi)$  is called a \emph{local filtration (of
$(\cM,\psi))$} if 
 \begin{align*}
 I \subset  J  \quad \Longrightarrow \quad  \cM_I  \subset \cM_J.&&&& \qquad \text{(Isotony)}
\end{align*}
A local filtration $\cM_\bullet$ is said to be 
\emph{locally minimal} if $\cM_I \vee \cM_J = \cM_K$
whenever $I,J,K \in \cI(\Nset_0)$ with $I \cup J = K$.
\end{Definition}
The isotony property ensures that inclusions are valid as they are assumed for commuting squares.
To be more precise, it holds that
\[
\begin{matrix}
\cM_{I} &\subset &\cM\\
\cup  &        & \cup \\
\cM_{K}   & \subset  & \cM_{J}
\end{matrix}
\]
for $I, J, K \in \cI(\Nset_0)$ with $K \subset  (I \cap J)$. Finally, let $\cN_\bullet \equiv
\{\cN_I\}_{I \in \cI(\Nset_0)}$ be another local filtration of $(\cM,\psi)$. Then $\cN_\bullet$ 
is said to be \emph{coarser} than $\cM_\bullet$  if $\cN_I \subset \cM_I$ for all 
$I \in \cI(\Nset_0)$ and we denote this by $\cN_{\bullet}\prec \cM_{\bullet}$. Occasionally we 
will address $\cN_{\bullet}$ also as a \emph{local subfiltration} of $\cM_{\bullet}$.
\begin{Definition}\normalfont \label{definition:markov-filtration}
Let $\cM_\bullet \equiv \{\cM_I\}_{I \in \cI(\Nset_0) }$ be a local filtration of $(\cM,\psi)$.
\begin{enumerate}
\item 
$\cM_{\bullet}$ is said to be \emph{Markovian} (or a \emph{local Markov filtration}) if $\cM_{[0,n-1]}$
and $\cM_{[n+1,\infty)}$ are CS independent over $\cM_{[n,n]}$ for all $n \ge 1$, i.e.~the 
inclusions 
\begin{eqnarray*}\label{eq:cs-global-mark-filt}
\begin{matrix}
\cM_{[0,n-1]} \vee \cM_{[n,n]} &\subset &\cM\\
\cup  &        & \cup \\
\cM_{[n,n]}   & \subset  & \cM_{[n+1,\infty)} \vee \cM_{[n,n]}
\end{matrix}
\end{eqnarray*}
form a commuting square for any $n \ge 1$.
\item 
$\cM_{\bullet}$ is said to be  \emph{saturated Markovian} (or a \emph{saturated local Markov filtration}), if
$\cM_{[0,n]}$ and  $\cM_{[n,\infty)}$ are  CS independent over $\cM_{[n,n]}$ for all $n \ge 0$, i.e.~ the
inclusions 
\begin{eqnarray*}
\begin{matrix}
\cM_{[0,n]} &\subset &\cM\\
\cup  &        & \cup \\
\cM_{[n,n]}   & \subset  & \cM_{[n,\infty)}
\end{matrix}
\end{eqnarray*}
form a commuting square for each $n \in \Nset_0$. 
\end{enumerate}
\end{Definition}
Cast as commuting squares, (saturated) Markovianity of the local filtration $\cM_\bullet$ has many equivalent
formulations, see Proposition \ref{proposition:cs}. In particular, corresponding to (i) of Definition \ref{definition:markov-filtration},
\begin{align}
&&&& E_{\cM_{[0,n-1]}\vee\cM_{[n,n]}}  E_{\cM_{[n,n]}\vee \cM_{[n+1,\infty)}} 
  & = E_{\cM_{[n,n]}}&& \text{for all $n \ge 1$.}   
&&&&  \tag{M} \label{eq:filt-markov-I}
\end{align}
In the formulation (ii) of Definition \ref{definition:markov-filtration}, it holds that
\begin{align}
&&&& E_{\cM_{[0,n]}}  E_{\cM_{[n,\infty)}} & = E_{\cM_{[n,n]}}&& \text{for all $n \ge 0$.}   
&&&&  \tag{M'} \label{eq:filt-markov-II}
\end{align}
Here $E_{\cM_I}$ denotes the $\psi$-preserving normal conditional expectation from $\cM$ onto
$\cM_I$. 
\begin{Lemma} \label{lemma:markov-I-II}
A saturated local Markov filtration $\cM_\bullet$ is a Markovian.  A locally minimal local 
Markov filtration $ \cM_\bullet$ is saturated. 
\end{Lemma}
In other words, if $\cM_\bullet$ is locally  minimal, then $\cM_{\bullet}$ has the Markov property \eqref{eq:filt-markov-I} if and only if it satisfies \eqref{eq:filt-markov-II}.
\begin{proof}
The isotony property of local filtrations ensures the inclusions 
$\cM_{[0,n-1]} \vee \cM_{[n,n]}  \subset \cM_{[0,n]} $ and
$\cM_{[n+1,\infty)} \vee \cM_{[n,n]}  \subset \cM_{[n,\infty)}$. Thus
\begin{align*}
E_{\cM_{[0,n-1]}\vee\cM_{[n,n]}}  &E_{\cM_{[n,n]}\vee \cM_{[n+1,\infty)}} \\
&=  E_{\cM_{[0,n-1]}\vee\cM_{[n,n]}}   E_{\cM_{[0,n]}} 
   E_{\cM_{[n,\infty]}} E_{\cM_{[n,n]}\vee \cM_{[n+1,\infty)}} \\
&= E_{\cM_{[0,n-1]}\vee\cM_{[n,n]}}   E_{\cM_{[n,n]}} E_{\cM_{[n,n]}\vee \cM_{[n+1,\infty)}}\\ 
&=   E_{\cM_{[n,n]}}.
\end{align*}
Consequently \eqref{eq:filt-markov-II} implies \eqref{eq:filt-markov-I}. Local minimality of 
$\cM_\bullet$ ensures $\cM_{[0,n-1]} \vee \cM_{[n,n]} = \cM_{[0,n]}$ and
$\cM_{[n+1,\infty)} \vee \cM_{[n,n]} = \cM_{[n,\infty)}$. That the property \eqref{eq:filt-markov-I} 
implies \eqref{eq:filt-markov-II} is immediate under this minimality assertion. 
\end{proof}
In the following we will not distinguish explicitly between the properties \eqref{eq:filt-markov-I} and
\eqref{eq:filt-markov-II} for local Markov filtrations. Also we will address 'saturated Markovianity' 
just as 'Markovianity', as our main results aim anyway at establishing the apparently stronger property
\eqref{eq:filt-markov-II} and, moreover, both these properties are equivalent for locally minimal local
filtrations. In particular, local Markov filtrations arising from noncommutative stationary Markov 
processes as in Subsection \ref{subsection:Noncommutative Stationary Processes} are indeed locally 
minimal.
\begin{Remark}\normalfont \label{remark:failure-mark-filt}
Markovanity of the local filtration $\cM_\bullet$, as introduced in Definition
\ref{definition:markov-filtration}, may not transfer to a local subfiltration $\cN_\bullet$. To be more
precise, the inclusions
\[
\begin{matrix}
\cN_{[0,n]} &\subset &\cN \\
\cup  &        & \cup  \\
\cN_{[n,n]}   & \subset  & \cN_{[n,\infty)} 
\end{matrix}
\]
may not form a commuting square, even though the corresponding inclusions of the local filtration 
$\cM_{\bullet}$ do. In classical probability, this observation is connected to the fact that a function
of a Markov process may not be Markovian. We address this phenomenon more in detail in Subsection 
\ref{subsection:FunctionsOfMarkov} from the viewpoint of noncommutative probability. 
\end{Remark}
\begin{Definition}\normalfont\label{definition:markovianity}
Let $\iota \equiv (\iota_n)_{n \in \Nset_0} \colon (\cA, \varphi) \to (\cM, \psi)$ be a sequence
of random variables with canonical local filtration $\cA_\bullet \equiv \big\{\cA_I := \bigvee_{n \in
I} \iota_n(\cA)\big\}_{I \in \cI(\Nset_0)}$. Furthermore let $\cM_\bullet \equiv \{\cM_I\}_{I \in
\cI(\Nset_0) }$ be another local filtration of $(\cM,\psi)$.
\begin{enumerate}
\item 
$\iota$ is said to be \emph{adapted to (the local filtration) $\cM_\bullet$} if $\cA_I \subset \cM_I$ 
for all $ I \in \cI(\Nset_0)$.  
\item 
$\iota$ is said to be $\cM_\bullet$-\emph{Markovian} (or an $\cM_\bullet$-\emph{Markov sequence})
if $\iota$ is adapted to $\cM_\bullet$ and
\[
E_{\cM_{[0,n]}} \iota_{n+1} = E_{\cA_{[n,n]}} \iota_{n+1} \qquad (n \ge 0).
\]
An $\cM_\bullet$-\emph{Markov} sequence $\iota$ is just said to be \emph{Markovian} (or a
\emph{Markov sequence}) if $\cM_\bullet = \cA_\bullet$.
\end{enumerate}
\end{Definition}
It is elementary to verify that an $\cM_\bullet$-Markovian sequence $\iota$ is
$\cN_\bullet$-Markovian whenever $\cA_\bullet \prec \cN_\bullet \prec \cM_\bullet$. In
particular, any $\cM_\bullet$-\emph{Markov} sequence is Markovian. We emphasize that there exist
non-Markovian sequences $\iota$ which are adapted to a local Markov filtration $\cM_\bullet$.
`Trivial' examples for such sequences are provided in Remark \ref{remark:warning-mark-filt}, and
`less trivial' classical examples are the topic of Subsection \ref{subsection:FunctionsOfMarkov}.
Let us remark here without further ado, that such non-Markovian sequences are especially of
relevance for distributional invariance principles, as the latter are often about the characterization
of so-called `mixtures of stochastic processes'.  

The next result states a sufficient condition ensuring that the Markovianity of a local filtration is
inherited by a sequence.
\begin{Lemma}\label{lemma:Markovfiltseq}
Suppose the sequence $\iota$ is adapted to the local Markov filtration $\cM_\bullet$ such that the
inclusions
\begin{eqnarray*}
\begin{matrix}
\cA_{[0,\infty]} &\subset &\cM\\
\cup  &        & \cup \\
\cA_{[n,n]}   & \subset  & \cM_{[n,n]}
\end{matrix}
\end{eqnarray*}
form a commuting square for all $n \ge 0$. Then $\iota$ is ($\cM_\bullet$-)Markovian.  
\end{Lemma}
In particular, if the canonical (local) filtration of $\iota$ is Markovian as in Definition
\ref{definition:markov-filtration}, then $\iota$ is a Markov sequence  as in Definition
\ref{definition:markovianity} (ii).
\begin{proof}
This is immediate from the equations
\begin{eqnarray*}
E_{\cM_{[0,n]}} E_{\cA_{[n+1,n+1]}}   
&=& E_{\cM_{[0,n]}} E_{\cM_{[n,\infty)}} E_{\cA_{[n+1,n+1]}}  \\
&=& E_{\cM_{[n,n]}} E_{\cA_{[n+1,n+1]}}  
= E_{\cM_{[n,n]}}  E_{\cA_{[0,\infty)}}  E_{\cA_{[n+1,n+1]}} \\
&=& E_{\cA_{[n,n]}} E_{\cA_{[n+1,n+1]}}.   
\end{eqnarray*}
\end{proof}
\begin{Remark}\normalfont \label{remark:warning-mark-filt}
The Markovianity of a local filtration in Definition \ref{definition:markovianity} (i) should be 
understood with some care, as it permits `trivial' filtrations. For example, given the probability 
space $(\cM,\psi)$, the local filtration $\cM_\bullet \equiv  \{\cM_I = \cM\}_{I \in \cI(\Nset_0)}$ 
is Markovian and any sequence of random variables is adapted to it. Note also that a constant
sequence $\iota$ (with $\iota_n = \iota_0$ for all $n \ge 0$) is adapted to a local Markov filtration
$\cM_\bullet$ whenever $\iota_0(\cA) \subset \cM_{[n,n]}$ for all $n \ge 0$. 
\end{Remark}
\begin{Remark}\normalfont
Presently we do not know in the full generality of our noncommutative framework if the 
canonical local filtration of a Markov sequence is Markovian (in the sense of Definition
\ref{definition:markov-filtration}). But this is true under additional algebraic conditions on the 
Markov sequence which fully include the classical case, see Subsection \ref{subsection:DIPinCP}. 
\end{Remark}
\subsection{Noncommutative stationary processes and dilations}
\label{subsection:Noncommutative Stationary Processes}
We introduce unilateral noncommutative stationary processes, as they underly the
approach to distributional invariance principles in \cite{Ko10,GK09}. Furthermore we present
Anantharaman-Delaroche's notion of factorizable Markov maps \cite{AD06} and  unilateral
versions of dilations of Markov maps as they are subject of K\"ummerer's approach to
noncommutative stationary Markov processes \cite{Ku85}. The existence of such dilations is
actually equivalent to the factoralizability of Markov maps \cite{HM11}.
\begin{Definition}\normalfont \label{definition:process-sequence}
A \emph{(unilateral) stationary process} $(\cM,\psi, \alpha, \cA_0)$ consists of a probability
space $(\cM,\psi)$, a $\psi$-conditioned subalgebra $\cA_0 \subset \cM$, and an endomorphism
$\alpha\in \End(\cM,\psi)$. The sequence
\[
(\iota_n)_{n \ge 0}\colon  (\cA_0, \psi_0) \to (\cM,\psi), 
\qquad \iota_{n} := \alpha^n|_{\cA_0}=\alpha^n\iota_0,
\]
is called the \emph{sequence of random variables associated to}  $(\cM,\psi, \alpha, \cA_0)$.
Here $\psi_0$ denotes the restriction of $\psi$ from $\cM$ to $\cA_0$ and $\iota_0$ denotes the
inclusion map of $\cA_0$ in $\cM$.
\end{Definition}
\begin{Definition} \normalfont \label{definition:property}
A stationary process $(\cM, \psi, \alpha, \cA_0)$  is said to have property ‘A’  if its
associated sequence of random variables 
$(\iota_n)_{n \ge 0} \colon (\cA_0,\psi_0) \to (\cM,\psi)$ has property ‘A’. For example,  
$(\cM, \psi, \alpha, \cA_0)$ is  minimal if  $(\iota_n)_{n \ge 0}$  is minimal.  
\end{Definition}
\begin{Definition}\normalfont  \label{definition:ncms}
The stationary process $\big( \cM,\psi, \alpha, \cA_0)$ is called a 
\emph{(unilateral noncommutative) stationary Markov process} if its canonical local filtration
\[
\Big\{\cA_I:=  \bigvee_{i \in I} \alpha^i \iota_0(\cA_0)\Big\}_{I \in \cI(\Nset_0)}
\]
is Markovian. If this process is minimal, then the endomorphism $\alpha$ is also called a 
\emph{Markov shift} with generator $\cA_0$.   

The associated $\psi_0$-Markov map $T=\iota_0^*\alpha \iota_0$, where $\iota_0$ is the inclusion
map of $\cA_0$ in $\cM$ and $\psi_0$ the restriction of $\psi$ to $\cA_0$, is often called the
\emph{transition operator} of the given Markov process.
\end{Definition}
Anantharaman-Delaroche has introduced in \cite[Definition 6.2]{AD06} the notion of factorizable
maps which we recall here in the setting relevant to our approach. 
\begin{Definition}[\cite{AD06}]\label{definition:factorizable} \normalfont
Let $(\cA,\varphi)$ be a probability space. A $\varphi$-Markov map $T$ on $\cA$ is called
\emph{factorizable} if there exists a probability space $(\cM,\psi)$ and random variables
$\iota_1, \iota_2\colon(\cA,\varphi)\to (\cM,\psi)$ such that $T=\iota_2^*\iota_1$.
\end{Definition}
The factorizability of a Markov map is actually equivalent to the existence of certain dilations
of a Markov map as they are studied by K\"ummerer in \cite{Ku85}. Here we are interested in the
following straightforward unilateral modifications of the original bilateral notions of dilation 
and Markov dilation in \cite{Ku85}. 
\begin{Definition}[\cite{Ku85}]\normalfont \label{definition:dilation}
Let $(\cA,\varphi)$ be a probability space. A $\varphi$-Markov map $T$ on $\cA$ is said to admit
a \emph{(unilateral state-preserving) dilation} if there exists a probability space $(\cM,\psi)$,
an endomorphism $\alpha\in \End(\cM,\psi)$ and a $(\varphi,\psi)$-Markov map $\iota_0:\cA\to \cM$
such that, for all $n \in \Nset_0$, 
\begin{eqnarray*}
T^n=\iota_0^*\alpha^n \iota_0.
\end{eqnarray*}
Such a dilation of $T$ is denoted by the quadruple $(\cM,\psi,\alpha,\iota_0)$ and is said to be
\emph{minimal} if $\cM=\bigvee_{n\in \Nset_0} \alpha^{n}\iota_0(\cA)$. The quadruple 
$(\cM,\psi,\alpha,\iota_0)$ is called a \emph{dilation of first order} if the equality 
$T=\iota_0^* \alpha \iota_0$ alone holds.
\end{Definition}
Actually it follows from the case $n=0$ that the $(\varphi,\psi)$-Markov map $\iota_0$ is a
random variable from $(\cA,\varphi)$ to $(\cM,\psi)$ such that $\iota_0\iota_0^*$ is the
$\psi$-preserving conditional expectation from $\cM$ onto $\iota(A)$. 
\begin{Definition}[\cite{Ku85}]\normalfont \label{definition:markovdilation}
The dilation $(\cM,\psi,\alpha,\iota_0)$ of the $\varphi$-Markov map $T$ on $\cA$ (as introduced
in Definition \ref{definition:dilation}) is said to be a \emph{(unilateral state-preserving) Markov
dilation} if the  local filtration $\big\{\cA_I := \bigvee_{n \in I} \alpha^n\iota_0(\cA)\big\}_{I \in 
\cI(\Nset_0)}$ is Markovian. 
\end{Definition}
\begin{Remark} \normalfont
A dilation of a $\varphi$-Markov map $T$ on $\cA$ may not be a Markov dilation. 
This is discussed in \cite[Section 3]{KuSchr83} where it is shown that Varilly has
constructed a dilation in \cite{Va81} which is not a Markov dilation. We are grateful 
to B.~K\"ummerer for bringing this to our attention \cite{Ku21}. 
\end{Remark}

\begin{Definition}\normalfont
\label{definition:tensordilation}
Let $(\cA, \varphi)$ be a probability space and $T$ be a $\varphi$-Markov map on $\cA$. A dilation (of
first order) $(\cM,\psi,\alpha,\iota_0)$  of $T$ is called a \emph{tensor dilation} if the conditional
expectation $\iota_0\iota_0^* \colon \cM \to \iota_0(\cA)$ is of tensor type, that is, there exists a
von Neumann subalgebra $\cC$ of $\cM$ with faithful normal state $\chi$ such that 
$\cM= \cA \otimes \cC$ and $(\iota_0\iota_0^*)(a \otimes x) = \chi(x)(a \otimes \1_{\cC})$ 
for all $a\in \cA$, $x\in \cC$.
 \end{Definition}

The following result was obtained by Haagerup and Musat in \cite{HM11}. We reformulate it here,
using unilateral instead of bilateral notions of dilations. 
\begin{Theorem}[\cite{HM11}]
Let $T$ be a $\varphi$-Markov map on the von Neumann algebra $\cA$ equipped with a  faithful normal 
state $\varphi$. Then the following are equivalent:
\begin{enumerate}
    \item[(a)] $T$ is factorizable (in the sense of Definition \ref{definition:factorizable}).
    \item[(b)] $T$ admits a dilation (in the sense of Definition \ref{definition:dilation}).
    \item[(c)] $T$ admits a Markov dilation (in the sense of Definition \ref{definition:markovdilation}).
\end{enumerate}
\end{Theorem}
We refer the reader to the proof given in \cite[Theorem 4.4]{HM11}, as its relevant arguments
transfer to our present setting.  Note that clearly (c) implies (b). Furthermore (b) implies (a)
by putting $\iota_1 := \alpha \iota$ and  $\iota_2 := \iota$ such that $T = \iota_2^* \iota_1$. 
So the main task is to prove that (a) implies (c), which makes use of \cite[Theorem 6.6]{AD06}.

Let us next relate the above unilateral notions of dilations and stationary processes.  It is
immediate that a dilation $\big(\cM,\psi,\alpha,\iota_0\big)$ of the $\varphi$-Markov map $T$
on $\cA$ gives rise to the stationary process $\big(\cM,\psi, \alpha, \iota_0(\cA)\big)$.
Furthermore this stationary process is Markovian if and only if the dilation is a Markov
dilation, as evident from the definitions. Conversely, a stationary Markov process
yields a dilation (and thus a Markov dilation) as it was shown by K\"ummerer in the
setting of bilateral stationary Markov processes (see \cite[Proposition 2.2.7]{Ku85}). We
adapt next its proof to the unilateral setting, as we will need this result in Theorem
\ref{theorem:TensorMarkov}.  
 \begin{Proposition}\label{proposition:dilation}
Let $(\cM,\psi, \alpha,\cA_0)$ be a stationary Markov process and $T=\iota_0^*\alpha \iota_0$ be
the corresponding transition operator where $\iota_0$ is the inclusion map of $\cA_0$ into $\cM$.
Then $(\cM,\psi,\alpha,\iota_0)$ is a dilation of $T$. In other words, the following diagram
commutes for all $n\in \Nset_0$:
	\[
	\begin{tikzcd}
	(\cA_0, \psi_0) \arrow[r, "T^n"] \arrow[d, "\iota_0"]
	& (\cA_0, \psi_0) \arrow[d, leftarrow, "\iota_0^*"] \\
	(\cM,\psi) \arrow[r, "\alpha^n"]
	& (\cM,\psi) 
	\end{tikzcd}.
	\]
Here $\psi_0$ denotes the restriction of $\psi$ to $\cA_0$.
\end{Proposition}
\begin{proof}
For $I\subset \Nset_0$, let $\cA_I=\bigvee_{n \in I}\alpha^n(\iota_0)(\cA_0))$ and let $E_{\cA_I}$
be the unique normal $\psi$-preserving conditional expectation onto $\cA_I$, which we write as
$P_{I}$ for brevity. Then the intertwiner relation $P_{I+k}\circ \alpha^k=\alpha^k\circ P_I$,
with $k\in \Nset_0$ and $I\subset \Nset_0$, can be seen using the following argument. Since
$\cA_{I+k} = \alpha^k(\cA_{I})$, it holds for $y \in \cA_{I}$ and $x \in \cM$ that 
\begin{align*}
\psi\big(\alpha^k(y) P_{I+k}\alpha^k(x)\big) 
&=\psi\big(\alpha^k(y)\alpha^k(x)\big)
= \psi\big(\alpha^k(yx)\big)
= \psi\big(yx\big)\\
&= \psi\big(yP_{I}(x)\big)
= \psi\big(\alpha^k(y) \alpha^k P_{I}(x)\big).
\end{align*}
As the functionals $\{\psi\big(\alpha^k(y) \cdot \big) \mid y \in \cA_I\}$, considered as
elements in the predual of $\cA_{I+k}$, are norm dense, we conclude $P_{I+k}\alpha^k(x) =
\alpha^k P_{I}(x)$ for all $x \in \cM$. In particular, we get 
\begin{align}\label{eq:dilation}
P_{[k-1,k-1]}\alpha^k \iota_0
&= P_{[k-1,k-1]}\alpha^{k-1}\alpha \iota_0 \notag\\
&= \alpha^{k-1}P_{[0,0]}\alpha \iota_0
= \alpha^{k-1} \iota_0 T \qquad  (k\in \Nset),
\end{align}
where we have used $P_{[0,0]}=\iota_0 \iota_0^*$ for the last equality. Now we prove the dilation
property by induction. We know that $\iota_0^*\alpha^n \iota_0=T^n$ is true for $n=0,1$. Suppose
$\iota_0^*\alpha^n \iota_0=T^n$ for some $n\in \Nset_0$. Then
\begin{alignat*}{3}
\iota_0^*\alpha^{n+1} \iota_0&= \iota_0^* P_{[0,0]}\alpha^{n+1} \iota_0 
&\qquad & \text{(as $\iota_0^*\iota_0=\text{Id}$ and $\iota_0\iota_0^*=P_{[0,0]})$}\\
&= \iota_0^* P_{[0,0]}P_{[0,n]}\alpha^{n+1} \iota_0 
&\qquad & \text{(as $\cA_{[0,0]}\subset \cA_{[0,n]}$)}\\
&= \iota_0^* P_{[0,0]}P_{[n,n]}\alpha^{n+1}\iota_0 
&\qquad & \text{(by Markovianity)}\\
&= \iota_0^*\alpha^n \iota_0 T 
&\qquad &\text{(by Equation \eqref{eq:dilation})}	\\
&= T^n T=T^{n+1} & \qquad & \text{(by induction hypothesis)}.
\end{alignat*}
\end{proof}
The dilation $(\cM,\psi,\alpha,\iota_0)$ in Proposition \ref{proposition:dilation} is of course a 
Markov dilation, as its local filtration is that of the stationary Markov process. 

The following formula for `pyramidally time ordered correlations' is obtained by an application
of the Markov property  and the intertwiner relation (see for example
\cite[Subsection 2.2]{Ku86} or \cite[Subsection 2.1]{Go04}).
\begin{Lemma}\label{lemma:qregression}
Let $(\iota_n)_{n\in \Nset_0}:(\cA,\varphi)\to (\cM,\psi)$ be a stationary Markov sequence of
random variables with  transition operator $T=\iota_0^* \alpha \iota_0$, where $\alpha\in
\End(\cM,\psi)$ such that $\iota_n=\alpha^n\iota_0, n\in \Nset_0$. Suppose $k_1<k_2<\cdots <k_n$.
Then for all $a_1,\ldots,a_n,b_1,\ldots,b_n\in \cA$,
\begin{align*} 
\psi\big(\iota_{k_1}(a_1^*)\cdots &\iota_{k_n}(a_n^*)\iota_{k_n}(b_n)
                        \cdots \iota_{k_1}(b_1)\big)\\
&=\varphi\big(a_1^*T^{k_2-k_1}(a_2^* T^{k_3-k_2}(a_3\cdots
            T^{k_n-k_{n-1}}(a_n^*b_n)\cdots)b_2)b_1\big).
\end{align*}
\end{Lemma}
\begin{proof}
Let $P_{I}$ be the the unique normal $\psi$-preserving conditional expectation onto $\cA_I$,
where  $\cA_I=\bigvee_{n \in I}\alpha^n(\iota_0(\cA))$. Then using that
$(\cM,\psi,\alpha,\iota_0)$ is a dilation of $T$, we get
\begin{align}\label{eq:repdilation} 
P_{[p,p]}\alpha^q \iota_0
&=P_{[p,p]}\alpha^p\alpha^{q-p}\iota_0 \notag\\
&=\alpha^p P_{[0,0]}\alpha^{q-p}\iota_0
=\alpha^p\iota_0 T^{q-p} \qquad (p < q). 
\end{align}
Now we compute 
\begin{alignat*}{3}
   &\psi\Big(\iota_{k_1}(a_1^*)\cdots 
     \iota_{k_n}(a_n^*)\iota_{k_n}(b_n)\cdots \iota_{k_1}(b_1)\Big)\\
   &= \psi\Big(\alpha^{k_1}(\iota_0(a_1^*))
     \cdots \alpha^{k_n}(\iota_0(a_n^*))\alpha^{k_n}(\iota_0(b_n))
     \cdots \alpha^{k_1}(\iota_0(b_1))\Big)\\
   &=\psi\Big(\iota_0(a_1^*)\alpha^{k_2-k_1}(\iota_0(a_2^*))
      \cdots \alpha^{k_n-k_1}(\iota_0(a_n^*b_n)) \cdots (\iota_0(b_1))\Big)\\
\begin{split} 
      &=\psi\Big(P_{[0,k_{n-1}-k_{1}]}(\iota_0(a_1^*)
      \cdots \alpha^{k_{n-1}-k_{1}}(\iota_0(a_{n-1}^*)) \qquad \times \\
      & \qquad \qquad \qquad\qquad \alpha^{k_n-k_{1}} 
      (\iota_0(a_n b_n^*))\alpha^{k_{n-1}-k_1}(\iota_0(b_{n-1}))\cdots (\iota_0(b_1)))\Big) 
\end{split}\\
\begin{split} 
      &=\psi\Big( \iota_0(a_1^*) \cdots \alpha^{k_{n-1}-k_{1}}(\iota_0(a_{n-1}^*)) 
      P_{[0,k_{n-1}-k_{1}]}(\alpha^{k_n-k_{1}}(\iota_0(a_n^*b_n))) \qquad \times\\
      & \qquad \qquad \qquad \qquad  \alpha^{k_{n-1}-k_1}
            (\iota_0(b_{n-1}))\cdots \iota_0(b_1)\Big)
\end{split}\\
\begin{split}
   &=\psi\Big( \iota_0(a_1^*)\cdots \alpha^{k_{n-1}-k_{1}}(\iota_0(a_{n-1}^*)) 
   P_{[k_{n-1}-k_1,k_{n-1}-k_{1}]}
       (\alpha^{k_n-k_{1}}(\iota_0(a_n^*b_n))) \qquad \times\\
    & \qquad \qquad \qquad \qquad \alpha^{k_{n-1}-k_1}(\iota_0(b_{n-1}))
       \cdots \iota_0(b_1)\Big) 
\end{split} \\
\begin{split}
   &=\psi\Big( \iota_0(a_1^*)\cdots
      \alpha^{k_{n-1}-k_{1}}(\iota_0(a_{n-1}^*))
      \alpha^{k_{n-1}-k_{1}}\iota_0
      T^{k_n-k_{n-1}}(a_n^*b_n) \qquad \times\\
      & \qquad \qquad \qquad \qquad \alpha^{k_{n-1}-k_{1}}(\iota_0(b_{n-1}))\cdots 
      \iota_0(b_1)\Big)
\end{split} \\
   &=\psi\Big( \iota_0(a_1^*)\cdots \alpha^{k_{n-1}-k_{1}}
      \iota_0(a_{n-1}^*T^{k_n-k_{n-1}}(a_n^* b_n)b_{n-1})\cdots\iota_0(b_1)\Big)\\
   &\phantom{=}\vdots \\
   &=\psi\Big(\iota_0 (a_1^* T^{k_2-k_1}(a_2^*\cdots
      T^{k_{n-1}-k_{n-2}}(a_{n-1}^*T^{k_n-k_{n-1}}(a_n^*b_n)b_{n-1})\cdots b_2)b_1) \Big)\\
   &=\varphi\Big(a_1^* T^{k_2-k_1}(a_2^*\cdots
     T^{k_{n-1}-k_{n-2}}(a_{n-1}^*T^{k_n-k_{n-1}}(a_n^*b_n)b_{n-1})\cdots b_2)b_1\Big).
\end{alignat*}
Here, the fifth line is obtained from the fourth by the module property of conditional
expectations, the sixth line is obtained from the fifth by the Markov property, and the seventh
from the sixth by \eqref{eq:repdilation} as $k_{n}-k_1>k_{n-1}-k_1$. We note that as our
random variables are all unital, we can assume without loss of generality that each
$k_{l+1}=k_{l}+1$ for each $l\in \{1,\ldots,n-1\}$.
\end{proof}
Within our general framework of noncommutative probability, we do not presently know if the
canonical local filtration of a stationary Markov sequence is Markovian (in the sense of Definition
\ref{definition:markov-filtration}).  However, this is indeed the case if one considers
commutative von Neumann algebras or, more generally, von Neumann algebras which are `pyramidally
generated' by a sequence of random variables. Let us make this algebraic condition more precise:
\begin{Definition}\normalfont
Suppose $\iota\equiv (\iota_n)_{n\in \Nset_0}:(\cA,\varphi)\to (\cM,\psi)$ is a sequence of
random variables with canonical local filtration $\cA_\bullet \equiv (\cA_I)_{I \in \cI(\Nset_0)}$. 
The von Neumann algebra $\cA_{\Nset_0}$ is said to be \emph{pyramidally generated} by the random
variables if elements of the form 
\[
\iota_{n}(a_0^*) \iota_{n+1}(a_1^*) \cdots \iota_{n+p-1}(a_{p-1}^*)  
\iota_{n+p}(a_{p}^*b_p) \iota_{n+p-1}(b_{p-1}) \cdots  \iota_{n+1}(b_{1}) \iota_{n}(b_{0})
\]
are weak*-total in $\cA_{[n,n+p]}$ for all $n,p \in \Nset_0$, where $a_1, \ldots a_{p},
b_1,\ldots, b_p \in \cA$. 
\end{Definition}
\begin{Proposition}\label{proposition:onestepMarkov}
Suppose $\iota\equiv (\iota_n)_{n\in \Nset_0}:(\cA,\varphi)\to (\cM,\psi)$ is a stationary Markov
sequence with canonical local filtration $\cA_\bullet \equiv (\cA_I)_{I \in \cI(\Nset_0)}$. If
$\cA_{\Nset_0}$ is pyramidally generated then $\cA_\bullet$ is Markovian.
\end{Proposition}
If $\cM$ (and hence $\cA$) is commutative then $\cA_{\Nset_0}$ is pyramidally generated and thus
$\cA_\bullet$ is Markovian.
\begin{proof}
By \eqref{eq:dilation} we have that $P_{[k-1,k-1]}\iota_{k}=\iota_{k-1} T$ for all 
$k\in \Nset$. We conclude from this that, for $n,p\in \Nset_0$ and $x,y\in \cA_{[n,n+p]}, 
a\in \cA$,
\begin{alignat*}{2}
P_{[0,n]}&\big(x\iota_{n+p+1}(a)y\big)\\
&= P_{[0,n]}P_{[0,n+p]}\big(x\iota_{n+p+1}(a)y\big)\\
&=P_{[0,n]}\big(xP_{[0,n+p]}\iota_{n+p+1}(a)y\big)\\
&=P_{[0,n]}\big(x\iota_{n+p}T(a)y \big).
\end{alignat*}
A repeated application of the above gives for $a_1,\ldots,a_p,b_1,\ldots,b_p\in \cA$,
\begin{align*} 
P_{[0,n]}&(\iota_{n}(a_0^*)\cdots  \iota_{n+p}(a_{p}^*)
         \iota_{n+p+1}(a_{p+1}^*b_{p+1})   \iota_{n+p}(b_{p}) \cdots \iota_{n}(b_0))\\
&=P_{[0,n]}\big(\iota_{n}(a_0^*T(a_1^*\cdots T(a_{p+1}^*b_{p+1})\cdots b_1)b_0)\big)\\
&=P_{[n,n]}\big(\iota_{n}(a_0^*T(a_1^*\cdots T(a_{p+1}^*b_{p+1})\cdots b_1)b_0)\big)\\
&=P_{[n,n]}(\iota_{n}(a_0^*)\cdots  \iota_{n+p}(a_{p}^*) 
           \iota_{n+p+1}(a_{p+1}^*b_{p+1})   \iota_{n+p}(b_{p}) 
            \cdots \iota_{n}(b_{0})).
\end{align*}
So $P_{[0,n]}$ and $P_{[n,n]}$ coincide on a weak$^*$-total subset of $\cA_{[n,n+p+1]}$ for every
$p\geq 0$. Since $\bigcup_{p \ge 0}\cA_{[n,n+p+1]}$ is weak$^*$-dense in $\cA_{[n,\infty)}$ it
follows $P_{[0,n]}P_{[n,\infty)}=P_{[n,n]}$ by another approximation argument. Thus
$\cA_{\bullet}$ is a local Markov filtration.
\end{proof}
\begin{Remark}\label{remark:nonstationaryMarkov}\normalfont
The assumption of stationarity of the Markov sequence in Proposition \ref{proposition:onestepMarkov} 
can be dropped so that the result still remains true. To see this, consider a (possibly)
inhomogeneous Markov sequence $\iota\equiv (\iota_n)_{n\in \Nset_0} \colon (\cA,\varphi)\to
(\cM,\psi)$ with $\varphi$-Markov maps $T_{n+1}:=\iota^*_n \iota_{n+1}^{}$ as transition
operators (see \cite[Subsection 2.1]{Go04}) on $\cA$ which satisfy
\[
P_{[n,n]}\iota_{n+1}=\iota_n T_{n+1} 
\]
for all $n\ge 0$. Here $P_I$ denotes the $\psi$-preserving conditional expectation onto $\cA_I$
as in Proposition \ref{proposition:onestepMarkov}. Replacing appropriately the single transition 
operator $T$ from the stationary setting by the transition operators $T_n$, all arguments
of the proof of Proposition \ref{proposition:onestepMarkov} transfer to the setting of inhomogeneous
Markov sequences. Thus the canonical local filtration of a (not necessarily stationary) Markov sequence
is Markovian (in the sense of Definition \ref{definition:markov-filtration}) if $\cA_{\Nset_0}$
is pyramidally-generated.
\end{Remark}
We close this subsection by providing a noncommutative notion of operator-valued Bernoulli
shifts, as we will meet them when constructing representations of the partial shift monoid $S^+$
in Section \ref{section:Constructions of Reps of F+}. The definition of such shifts stems from
investigations of K\"ummerer on the structure of noncommutative Markov processes in \cite{Ku85},
and such shifts can also be seen to emerge from the noncommutative extended de Finetti theorem in
\cite{Ko10} (see Theorem \ref{theorem:ncdf}). 
\begin{Definition}\normalfont \label{definition:ncbs}
The minimal stationary process $\big( \cM,\psi, \beta, \cB_0)$ is called a \emph{(unilateral
noncommutative full/ordered) Bernoulli shift} with \emph{generator} $\cB_0$ if $\cM^{\beta}
\subset \cB_0$ and $\{\beta^n(\cB_0)\}_{n \in \Nset_0}$ is  full/order CS independent over
$\cM^\beta$.   
\end{Definition}
It is easy to see that a Bernoulli shift $( \cM,\psi, \beta, \cB_0)$  is a minimal
stationary Markov process where the corresponding transition operator $\iota_0^*\beta \iota_0$ 
is a conditional expectation (onto $\cM^{\beta}$, the fixed point algebra of $\beta$). 
Here $\iota_0$ denotes the inclusion map of $\cB_0$ into $\cM$.
\section{A new distributional invariance principle in classical probability theory}
\label{section:MarkovianityandF+Classical}
Initially the deep connection between Markovianity and certain representations of the Thompson
monoid $F^+$ emerged from  investigations on distributional symmetries and invariance principles
in noncommutative probability. These pioneering results are of  course also available when
restricting our framework to that of classical probability. We review in Subsection 
\ref{subsection:CSMPR} the algebraization procedure of random variables, in particular to
make our approach more accessible to readers from traditional probability theory. Continuing in
Subsection \ref{subsection:DIPinCP}, we discuss some of our main definitions and results on
`partial spreadability', using their reformulations in terms of properties of classical random
variables. Among these reformulated results is the already stated Theorem \ref{theorem:definetti-intro},
a de Finetti theorem for stationary Markov sequences. Subsection
\ref{subsection:FunctionsOfMarkov} presents evidence that functions of stationary Markov
sequences are partially spreadable. Finally, Subsection \ref{subsection:markov-dilations} is
about constructing stationary Markov sequences as Markov dilations in classical probability, as
featured in K\"ummerer's approach to noncommutative stationary Markov processes. Such an
alternative construction turns out to be key for the construction of a representation of the
Thompson monoid $F^+$, as needed for the implementation of partial spreadability. 
\subsection{Algebraization of classical probability theory}
\label{subsection:CSMPR}
We review how one arrives at noncommutative notions of probability spaces, random variables
and stationary processes when starting from a traditional probabilistic setting. 

Let $(\Omega, \Sigma, \mu)$ be a standard probability space. Then $\cL := L^\infty(\Omega,
\Sigma, \mu)$, the Lebesgue space of essentially bounded $\Cset$-valued measurable functions, is
a commutative von Neumann algebra and, with $f \in \cL$, the Lebesgue integral $\trace_\mu(f) :=
\int_\Omega f d\mu$ is a faithful normal tracial state on $\cL$. In other words, $\trace_\mu(f)$
is the expectation of the essentially bounded $\Cset$-valued random variable $f \in \cL$ such
that $\trace_\mu(f^*f) = \trace_\mu(|f|^2) = 0$  implies $f=0$ (in the Lebesgue sense). The pair
$(\cL,\trace_\mu)$ is the standard example for a noncommutative probability space coming from
classical probability theory. Conversely, a noncommutative probability space $(\cM, \psi)$ with 
a commutative von Neumann algebra $\cM$ can be seen to be isomorphic to this standard example,
provided $\cM$ has a separable predual.    

Now let $(\Omega_0, \Sigma_0)$ be a standard Borel space and consider the 
$(\Omega_0, \Sigma_0)$-valued random variable $\xi_0$ on $(\Omega, \Sigma, \mu)$. Denote by
$\mu_0 := \mu \circ \xi_0^{-1}$ the pushforward measure of $\mu$ and by 
$\trace_{\mu_0} := \int_\Omega \, \cdot \,d\mu_0$ the induced (tracial) state on $\cL_0 :=
L^\infty(\Omega_0,\Sigma_0,\mu_0)$. Then $\iota_0(f) := f \circ \xi_0$ defines an injective
*-homomorphism from $\cL_0 = L^\infty(\Omega_0,\Sigma_0,\mu_0)$ into $\cL =  L^\infty(\Omega,
\Sigma, \mu)$ such that $\trace_\mu \circ \iota_0 = \trace_{\mu_0}$. Altogether we have arrived 
at the algebraization of the random variable $\xi_0$ to the noncommutative random variable 
$\iota_0 \colon (\cL_0, \trace_{\mu_0}) \to (\cL, \trace_{\mu})$. Note that the noncommutative
probability spaces $(\cL_0, \trace_{\mu_0})$ and its image under $\iota_0$ in $(\cL,
\trace_{\mu})$ are isomorphic, and hence we frequently identify $\cL_0$ with
$L^{\infty}(\Omega, \Sigma^{\xi_0},\mu_{\xi_0})$, where $\Sigma^{\xi_0}:=\sigma\{\xi_0^{-1}(A)
\mid A\in \Sigma_0\}$ and $\mu_{\xi_0}$ is the restriction of $\mu$ to $\Sigma^{\xi_0}$.

Conversely, if  $\tiota_0 \colon (\cA,\varphi) \to (\cM,\psi)$ is an injective *-homomorphism, 
where $\cM$ (and thus $\cA$) is a commutative von Neumann algebra with separable predual, then 
there exist two standard probability spaces $(\Omega, \Sigma, \mu)$ and 
$(\Omega_0, \Sigma_0, \mu_0)$, and an $(\Omega_0, \Sigma_0)$-valued random variable $\xi_0$ on
$(\Omega, \Sigma, \mu)$ with $\mu \circ \xi_0^{-1} = \mu_0$ such that the noncommutative random
variables $\tiota_0$ and $\iota_0$ are the same, up to isomorphisms between the involved standard
probability spaces. 

This algebraization procedure, roughly phrasing, puts the emphasis on the $\sigma$-algebra
generated by a random variable and less on the random variable itself. Effectively, an
(unbounded) random variable $\xi_0$ may be replaced by an (essentially bounded) random variable
$f \circ \xi_0 = \iota_0(f)$, as long as $f \in \cL_0$ is chosen such that $\xi_0$ and $f \circ
\xi_0$ generate the same $\sigma$-algebra. Of course, this observation for a single random
variable $f \circ \xi_0$ extends immediately to multivariate settings, as the family of 
functions $\{f_i\}_{i \in I} \subset \cL_0$ yields the family of (bounded) random variables 
$\{\iota_0(f_i)  =  f_i \circ \xi_0\}_{i \in I}$.    

We are finally left to address the algebraization procedure for sequences of random
variables. Here we will constrain our investigations to sequences given by identically
distributed random variables. This simplification improves the transparency of our
approach, since it allows us to realize each noncommutative random variable as an
injective mapping with the same domain. More explicitly, we will consider a sequence of
identically distributed $(\Omega_0, \Sigma_0)$-valued random variables $\xi \equiv
\{\xi_n\}_{n =0}^\infty$  on $(\Omega, \Sigma, \mu)$ and denote its canonical local filtration
by $\Sigma^\xi_\bullet \equiv \{\Sigma^{\xi}_I\}_{I \subset {\Nset_0}}^{}\subset \Sigma$, where 
\[ 
\Sigma^{\xi}_I:=\sigma\{\xi_i^{-1}(A)\mid A\in \Sigma_0, i\in I\}.
\]
The $\sigma$-algebra $\Sigma^{\xi}_{\Nset_0}$ is also denoted simply by $\Sigma^{\xi}$ and the
measure $\mu|_{\Sigma^{\xi}_{\Nset_0}}$ by $\mu_{\xi}$. Note also that  $\Sigma^{\xi}_{[n,n]}$ is
often written as $\Sigma^{\xi_n}$ or $\sigma\{\xi_n\}$.

We remind that in the setting of the standard probability space $(\Omega, \Sigma, \mu)$, a
$\sigma$-subalgebra of $\Sigma$ (generated by random variables) is understood to be completed
with respect to sets of $\mu$-measure 0.   Since all random variables $\xi_n$ are identically
distributed we can identify the pushforward measures $\mu_n = \mu \circ \xi_n^{-1}$ and 
$\mu_0 = \mu \circ \xi_0^{-1}$. Thus the algebraization procedure yields the sequence of
noncommutative random variables
\[
\iota \equiv (\iota_{n})_{n \ge 0} \colon (\cL_0, \trace_{\mu_0}) \to   (\cL, \trace_{\mu}), 
\]
where $\iota_n (h) = h \circ \xi_n$. As before, $\cL_0$ can be canonically identified with
$L^{\infty}(\Omega,\Sigma^{\xi_{0}}, \mu_{\xi_0})$ as a von Neumann subalgebra of $\cL$.
Here $\mu_{\xi_0}$ denotes the restrictions of $\mu$ to $\Sigma^{\xi_{0}}$.
\subsection{Partial spreadability and de Finetti theorems in classical probability}
\label{subsection:DIPinCP}
Let $\xi \equiv \{\xi_n\}_{n =0}^\infty$ be a sequence of identically distributed
$(\Omega_0, \Sigma_0)$-valued random variables on $(\Omega, \Sigma, \mu)$ and denote by 
$ \iota \equiv (\iota_{n})_{n \ge 0} \colon (\cL_0, \trace_{\mu_0}) \to   (\cL,\trace_{\mu})$
the corresponding sequence of noncommutative random variables, obtained from $\xi$ through
the algebraization procedure. It is elementary to see that $\xi$ is stationary (in the wide
sense) if and only if $\iota$ is stationary (in the sense of Definition \ref{definition:distsym}). 
Furthermore we may assume $\Sigma = \Sigma^{\xi}_{\Nset_0}$ for this stationary sequence.
Consequently, there exists a $\mu$-preserving $\Sigma$-measurable map $\eta$ on $\Omega$ such
that $\xi_n = \xi_0 \circ \eta^n$ for $n \in \Nset$. This map $\eta$ lifts to an endomorphism
$\alpha$ on the von Neumann algebra $\cL$  such that $\alpha(f) = f \circ \eta$ and
$\trace_{\mu} \circ \alpha = \trace_\mu$. Thus the stationary sequence $\xi$, or its
algebraization $\iota$, yields the stationary process $(\cL, \trace_\mu, \alpha, \cL_0)$ 
(in the sense of Definition \ref{definition:process-sequence}). Conversely, if $(\cL, \trace_\mu,
\alpha, \cL_0)$ is a stationary process with a commutative von Neumann algebra $\cL$ acting on a
separable Hilbert space, then the endomorphism $\alpha$ can be implemented by a
measure-preserving measurable map $\eta$ as described before. 

This algebraization procedure also applies to distributional invariance principles other than
stationarity, like exchangeability or spreadability. Taking into account their characterisations
in Proposition \ref{proposition:dip}, the extended de Finetti theorem can be casted as follows, 
adapting in wide parts the formulation in Kallenberg's monograph \cite{Ka05}.
\begin{Theorem}\label{theorem:cdeF}
Suppose $\xi \equiv (\xi_n)_{n \ge 0}$ is a sequence of $(\Omega_0, \Sigma_0)$-valued random
variables on the standard probability space $(\Omega, \Sigma, \mu)$. Then the following are
equivalent:
\begin{enumerate}
    \item[(a)]
    $\xi$ is exchangeable;
    \item[(b)]
    $\xi$ is spreadable;
    \item[(b')]
    there exist $\Sigma^\xi$-measurable $\mu_{\xi}$-preserving maps $\{\eta_{n}\}_{n \ge 0}$ on $\Omega$ 
    satisfying the relations $\eta_k  \eta_\ell = \eta_{\ell+1} \eta_{k}$ for $0 \le k \le \ell < \infty$ 
    such that $\xi_0 = \xi_0 \circ \eta_k$  for all $k> 0$ (localization) and $\xi_n = \xi_0 \circ \eta_0^n$ 
    for all $n \ge 0$ (stationarity);
\item[(c)]
    $\xi$ is stationary and conditionally independent.     
\item[(d)]
    $\xi$ is identically distributed and conditionally independent. 
\end{enumerate}
Here the conditioning is taken with respect to the tail $\sigma$-algebra of the sequence $\xi$. 
\end{Theorem}
The equivalence of conditions (a) and (d) is the content of the traditional de Finetti theorem,
and the equivalent characterization by condition (b) is attributed to Ryll-Nardzewski. We refer
the reader to \cite{Ka05} for further information of this extended version of the de Finetti
theorem and note that, usually, the clearly equivalent condition (c) is omitted in the published
literature. The reformulation of spreadability by condition (b') emerges from \cite{Ko10,EGK17}
where it was discovered that spreadability connects to the representation theory of the partial
shifts monoid $S^+$ (see Subsection \ref{subsection:partialshifts}), and that this monoid $S^+$
appears in the inductive limit of the category of semicosimplicial probability spaces.

As spreadability can be implemented through representations of the partial shifts monoid $S^+$,
it is natural to implement a new distributional invariance principle through representations of
the Thompson monoid $F^+$, as done in our framework of noncommutative probability. From the view
point of classical probability, it is of interest to further investigate the following
reformulations of this new distributional invariance principle in noncommutative probability. 
\begin{Definition}\normalfont\label{definition:partspread-c}
Let $\xi \equiv (\xi_n)_{n \ge 0}$ be a sequence of $(\Omega_0, \Sigma_0)$-valued random
variables on the standard probability space $(\Omega, \Sigma, \mu)$. 
\begin{enumerate}
\item 
$\xi$ is said to be \emph{partially spreadable} if there exists $\Sigma^{\xi}$-measurable
$\mu_{\xi}$-preserving maps $\{\eta_n\}_{n \ge 0}$ on $\Omega$ satisfying the relations $\eta_k
\eta_{\ell} = \eta_{\ell+1} \eta_{k}$ for $0 \le k < \ell < \infty$ such that $\xi_0 = \xi_0
\circ \eta_k$  for all $k> 0$ (localization) and $\xi_n = \xi_0 \circ \eta_0^n$ for all $n \ge 0$
(stationarity).
\item 
The partially spreadable sequence $\xi$ from (i) is said to be \emph{maximal partially
spreadable} if $\sigma\{\xi_0\} = \bigcap_{k\ge 1}\cI_{\eta_k}$. Here $\cI_{\eta_k}$ denotes the
$\sigma$-subalgebra of $\mu$-almost $\eta_k$-invariant sets in $\Sigma$. 
\end{enumerate}
More generally, a sequence $\txi \equiv (\txi_n)_{n \ge 0}$ of $(\Omega_0, \Sigma_0)$-valued
random variables on a probability space $(\tOmega, \tSigma, \tmu)$ is said to be
\emph{(maximal) partially spreadable} if it has the same distribution as a (maximal) partially
spreadable sequence $\xi$.
\end{Definition}
Thus partial spreadability is implemented by a representation of the Thompson monoid 
$F^+$ in the measure-preserving measurable maps on the probability space 
$(\Omega, \Sigma, \mu)$. It is worthwhile to point out the subtle fact that this
representation may not restrict to the probability space $(\Omega, \Sigma^\xi, \mu_{\xi})$,
where $\Sigma^\xi$ denotes here the $\sigma$-algebra generated by $\xi$ and $\mu_{\xi}$ the
restriction of $\mu$ to  $\Sigma^\xi$. This is in contrast to exchangeability, spreadability
and stationarity, which are robust under such restrictions to the minimal setting of the 
underlying probability space.  

We provide next the classical reformulation of a local Markov filtration from Definition
\ref{definition:markov-filtration}, as it will be used to generalize the extended de Finetti
theorem, Theorem \ref{theorem:cdeF}.  We stress again that Markovianity of a local filtration, as a
property of a (partially ordered) family of $\sigma$-subalgebras indexed by `time intervals' $I
\subset \Nset_0$, is a prori not in reference to any stochastic process. 
\begin{Notation}\normalfont
Suppose $\tSigma$ is a $\sigma$-subalgebra of $\Sigma$. Then $\mathbb{E}_{\tSigma}$ denotes the
$\mu$-preserving (normal) conditional expectation from $L^\infty(\Omega, \Sigma,\mu)$ onto
$L^\infty(\Omega,\tSigma,\mu|_{\tSigma})$.  
\end{Notation}
Commonly a filtration of a probability space is understood to be a non-decreasingly (or non-increasingly) 
totally ordered family of $\sigma$-subalgebras of $\Sigma$ \cite{Ve17}. As we will need to consider 
families of $\sigma$-subalgebras which are indexed by partially ordered sets, we will refer to such 
families as `local filtrations', a well-established notion in the context of random fields. Recall that
$\cI(\Nset_0)$ denotes the set of `intervals' in $\Nset_0$ (see Notation \ref{notation:indexsets}).
\begin{Definition}\normalfont \label{definition:markov-filtration-c}
A local filtration $\cF_\bullet  \equiv \{\cF_I\}_{I\in \cI(\Nset_0)} \subset \Sigma$
of $(\Omega, \Sigma, \mu)$ is said to be \emph{Markovian} (or a \emph{local Markov filtration}) if
$\cF_{[0,n-1]}$ and $\cF_{[n+1,\infty)}$ are conditionally independent over $\cF_{[n,n]}$ for all 
$n \ge 1$. In other words, it holds the Markov property 
\begin{align}
&&&& \mathbb{E}_{\cF_{[0,n-1]} \vee \cF_{[n,n]}}  \mathbb{E}_{\cF_{[n,n]} \vee \cF_{[n+1,\infty)}} 
 & = \mathbb{E}_{\cF_{[n,n]}}&& \text{for all $n \ge 1$.}   
&&&&  \tag{M$_\text{c}$} \label{eq:filt-markov-I-c} 
\end{align}
Here $\cF_I \vee \cF_J$ denotes the $\sigma$-algebra generated by $\cF_I$ and $\cF_J$. 
\end{Definition}
An apparently slightly stronger definition of Markovianity, corresponding to saturated
Markovianity in Definition \ref{definition:markov-filtration}, is given by the conditions
\begin{align}
&&&& \mathbb{E}_{\cF_{[0,n]}}  \mathbb{E}_{\cF_{[n,\infty)}} 
 & = \mathbb{E}_{\cF_{[n,n]}}&& \text{for all $n \ge 0$.}   
&&&&  \tag{M$_\text{c}^\prime$} \label{eq:filt-markov-II-c} 
\end{align}
It is easily verified that \eqref{eq:filt-markov-II-c}  implies \eqref{eq:filt-markov-I-c}, 
using the isotony properties of local filtrations, see also Lemma \ref{lemma:markov-I-II}.  
Both Markov conditions \eqref{eq:filt-markov-II-c} and \eqref{eq:filt-markov-I-c} are equivalent 
if the local filtration is \emph{locally minimal}, i.e. it holds $\cF_I \vee \cF_J = \cF_{K}$ 
for any $I,J,K \in \cI(\Nset_0)$ with $I \cup J = K$. 

In the following we will not distinguish explicitly between the Markov properties 
\eqref{eq:filt-markov-I-c} and \eqref{eq:filt-markov-II-c}, as we will throughout establish  
the stronger property \eqref{eq:filt-markov-II-c} which is anyway equivalent to 
\eqref{eq:filt-markov-I-c} for locally minimal local filtrations.  
\begin{Remark}\normalfont \label{remark:cs-not-robust}
Suppose the local filtration $\cG_{\bullet} \equiv \{\cG_{I}\}_{I \in \cI(\Nset_0)} \subset \Sigma$ is
coarser than the local Markov filtration $\cF_\bullet$. Then  $\cG_{\bullet}$ is not necessarily
Markovian. A sufficient condition so that $\cG_\bullet$ inherits Markovianity from $\cF_\bullet$
is 
\[
\mathbb{E}_{\cG_{[0,\infty)}} \mathbb{E}_{\cF_{[n,n]}}
= \mathbb{E}_{\cG_{[n,n]}}  \qquad (n \in \Nset_0),
\]
as it is easily verified:
\begin{align*}
\mathbb{E}_{\cG_{[0,n]}}  \mathbb{E}_{\cG_{[n,\infty)}}
 & = \mathbb{E}_{\cG_{[0,n]}} \mathbb{E}_{\cF_{[0,n]}}  
      \mathbb{E}_{\cF_{[n,\infty)}} \mathbb{E}_{\cG_{[n,\infty)}} 
= \mathbb{E}_{\cG_{[0,n]}}
   \mathbb{E}_{\cF_{[n,n]}} \mathbb{E}_{\cG_{[n,\infty)}} \\
& = \mathbb{E}_{\cG_{[0,n]}} \mathbb{E}_{\cG_{[0,\infty)}} 
     \mathbb{E}_{\cF_{[n,n]}} \mathbb{E}_{\cG_{[n,\infty)}}
= \mathbb{E}_{\cG_{[0,n]}} \mathbb{E}_{\cG_{[n,n]}} \mathbb{E}_{\cG_{[n,\infty)}} 
= \mathbb{E}_{\cG_{[n,n]}}.
\end{align*} 
\end{Remark}
\begin{Definition}\normalfont \label{definition:markovianity-c}
Suppose the standard probability space $(\Omega, \Sigma, \mu)$ is equipped with a local filtration
$\cF_\bullet  \equiv \{\cF_I\}_{I\in \cI(\Nset_0)} \subset \Sigma$. Let $\xi \equiv 
(\xi_n)_{n\in \Nset_0}$ be a sequence of $(\Omega_0,\Sigma_0)$-valued random variables on
$(\Omega, \Sigma, \mu)$ with canonical local filtration $\Sigma^\xi_\bullet \subset \Sigma$. 
\begin{enumerate}
\item 
$\xi$ is said to be \emph{adapted to (the local filtration) $\cF_\bullet$}, if $\Sigma^\xi_I \subset
\cF_{I}$ for all $I \in \cI(\Nset_0)$. 
\item
$\xi$ is said to be \emph{$\cF_\bullet$-Markovian} (or an \emph{$\cF_\bullet$-Markov sequence}), 
if $\xi$ is adapted to $\cF_{\bullet}$ and 
\begin{equation}
\label{eq:classicalMarkov1}
\mathbb{E}_{\cF_{[0,n]}} \mathbb{E}_{\sigma\{\xi_{n+1}\}}
= \mathbb{E}_{\sigma\{\xi_{n}\}}\mathbb{E}_{\sigma\{\xi_{n+1}\}}
\end{equation}
A \emph{$\Sigma^\xi_\bullet$-Markovian} sequence $\xi$ is just said to be \emph{Markovian} (or a
\emph{Markov sequence}). 
\end{enumerate}
\end{Definition}
$\cF_\bullet$-Markovianity of a sequence makes use only of the totally ordered family 
$\{\cF_{[0,n]}\}_{n \ge 0}$ in $\cF_\bullet$ which is a filtration
in the common sense. This gives the usual notion of an $\{\cF_{[0,n]}\}_{n \ge 0}$-Markovian sequence, 
see for example \cite{Ka97}.
\begin{Lemma} \label{lemma:ms-mf-equivalence}
Let $\xi \equiv (\xi_n)_{n\in \Nset_0}$ be a sequence of $(\Omega_0,\Sigma_0)$-valued random
variables on  $(\Omega, \Sigma, \mu)$ with canonical local filtration $\Sigma^\xi_\bullet \subset
\Sigma$. Then the following are equivalent:
\begin{enumerate}
    \item[(a)]  $\xi$ is a Markov sequence;
    \item[(b)]  $\Sigma^\xi_\bullet$ is a local Markov filtration. 
\end{enumerate}
\end{Lemma}
\begin{proof}
The implication `(b) $\Rightarrow$ (a)' is clear by taking $\cF_{\bullet}\equiv
\Sigma^{\xi}_{\bullet}$ in \eqref{eq:classicalMarkov1}. The converse implication `(a)
$\Rightarrow$ (b)' follows immediately from Proposition \ref{proposition:onestepMarkov} and Remark
\ref{remark:nonstationaryMarkov}, as it formulates the same in the language of noncommutative
probability. 
\end{proof}
Suppose now that the standard probability space $(\Omega, \Sigma, \mu)$ is equipped with a
representation $\rho_*$ of the Thompson monoid $F^+ = \langle g_0, g_1, g_2, \ldots \mid g_k
g_\ell = g_{\ell+1} g_{k} \text{ for }  0\leq k<\ell <\infty \rangle^+$ in the $\mu$-preserving
$\Sigma$-measurable maps on $\Omega$. Let us put $\eta_k := \rho_*(g_k)$ for brevity and recall
that $\cI_{\eta_k}$ denotes the $\sigma$-subalgebra of $\mu$-almost $\eta_k$-invariant sets in 
$\Sigma$. Putting  $\cJ_n := \bigcap_{k \ge n+1} \cI_{\eta_k}$ we obtain the tower
of $\sigma$-subalgebras
\[
\cJ_{-1} \subset \cJ_0 \subset \cJ_1 \subset \cdots \subset
\cJ_\infty := \sigma\{\cJ_n \mid n \ge 0  \} \subset \Sigma.
\]
Note that $\cJ_{-1}$ is the $\sigma$-subalgebra of $\mu$-almost $\rho_*(F^+)$-invariant sets in $\Sigma$. 
This tower is a filtration which can  be further upgraded to the local filtration $\cF_\bullet^+ \equiv
\{\cF_I^{+}\}_{I \in \cI(\Nset_0)}^{}$ by putting
\[
\cF_{[0,n]}^{+} := \cJ_n, \qquad  \cF_{[m, m+n]}^{+} := \eta_0^m(\cJ_n),
\qquad \cF_{[m, \infty)}^{+} := \eta_0^m(\cJ_\infty),
\]
where $\eta_0(\cJ) := \sigma\{\eta_0(J) \mid J \in \cJ\}$ for a 
$\sigma$-subalgebra $\cJ \subset \Sigma$. 

We will assume for the remainder of this subsection that the representation $\rho_*$ of the
Thompson monoid $F^+$, and thus the local filtration $\cF^+_\bullet$, has the so-called generating 
property $\cJ_\infty = \cF^+_{\Nset_0} = \Sigma$. This can always be achieved by restriction of
$\sigma$-algebra and measure of the standard probability space $(\Omega,\Sigma,\mu)$ when
necessary, as shown in Subsection \ref{subsection:generating} in the general noncommutative
setting. This ensures, by an application of the mean ergodic theorem, that
$\cJ_n=\cI_{\eta_{n+1}}$ for every $n\geq -1$. We obtain as a main result that the fixed point
$\sigma$-algebras $\{\cJ_n\}_{n \ge 0}$ of the represented Thompson monoid $F^+$ encode
Markovianity. 
\begin{Theorem}\label{theorem:cMarkovFilt}
The local filtration $\cF_\bullet^{+}$ is Markovian (in the sense of Definition 
\ref{definition:markov-filtration-c}). 
\end{Theorem}
Actually this result is the classical reformulation of Corollary \ref{corollary:markov-2} which
emerges from applications of the von Neumann mean ergodic theorem  in Theorem
\ref{theorem:cs-markov-1}.  The former establishes, here reformulated in the language of
probability theory, that the two $\sigma$-subalgebras $\cJ_{m+k}$ and $\eta_0^{k}(\cJ_n)$
are conditionally independent over $\eta_0^{k}(\cJ_m)$ for $0 \le m \le n < \infty$ and $k\ge 0$.
Already some instances of these  conditional independences between shifted fixed point
$\sigma$-algebras suffice to encode the Markovianity of the local  filtration $\cF_\bullet^+$. We 
refer the reader to Subsection \ref{subsection:s-f-p-a} for a more detailed investigation of
these conditional independences in the general setting of noncommutative probability. 

Partial spreadability as a new distributional invariance principle of a sequence $\xi$
(as casted in Definition \ref{definition:partspread-c}) provides a sufficient condition to ensure that
the canonical local filtration of $\xi$ is adapted to the local filtration of fixed point $\sigma$-algebras
$\cF^+_\bullet$. Altogether we arrive at the following result which takes into account the
stationarity of a partially spreadable sequence, as evident from Definition
\ref{definition:partspread-c} of this distributional invariance principle. For clarity of its
formulation, the sequence $\xi$ and the representation $\rho_*$ of the Thompson monoid $F^+$ are
assumed to be realized on the same probability space. 
\begin{Theorem}\label{theorem:cl-ps-Markov}
A partially spreadable sequence $\xi$ is stationary and adapted to the local Markov filtration
$\cF_{\bullet}^{+}$. 
\end{Theorem}
For the convenience of the reader from classical probability, let us we sketch the proof which
essentially follows the arguments from Remark \ref{remark:adaptedness} in the general setting of
noncommutative probability.
\begin{proof}
Since the sequence $ \xi \equiv (\xi_n)_{n \ge 0}$ and the representation $\rho_*$ of the
Thompson monoid $F^+$ are assumed to be realized on the the same standard probability  space
$(\Omega, \Sigma, \mu)$, partial spreadability ensures $\xi_0 \circ \eta_n=\xi_0$ and 
$\xi_n = \xi_0 \circ \rho_*(g_0^n) = \xi_0 \circ \eta_0^n$ for all $n \ge 1$. The latter
condition implies the stationarity of the sequence. As the local  filtration $\cF^+_\bullet$ is
Markovian by Theorem \ref{theorem:cMarkovFilt}, we are left to verify the adaptedness of the 
sequence $\xi$ to $\cF^+_{\bullet}$. Indeed, it follows from the localization properties
$\xi_0 \circ \eta_n= \xi_0$ (for $n\geq 1$) that $\sigma\{\xi_0\}\subset \bigcap_{k \ge 1}
\cI_{\eta_k}=  \cJ_0$. Using the relations of the generators of $F^+$, one concludes further 
that $\Sigma^{\xi}_{[0,n]}\subset \cJ_{n}$ for every $n\geq 0$. This in turn gives the 
inclusions $\Sigma^\xi_{[m,n]}\subset \eta_0^m(\cJ_{n-m})=\cF^+_{[m,n]}$ for all $0
\leq m \leq n $ and, by an approximation argument, $\Sigma^\xi_{[m,\infty]}\subset
\cF^+_{[m,\infty]}$ for all $m\geq 0$. Overall, we arrive at $ \Sigma^{\xi}_{I} \subset \cF^+_I$
for all $I\in \cI(\Nset_0)$, that is, the adaptedness of the sequence $\xi$ to the local filtration
$\cF^+_{\bullet}$.
\end{proof}
We stress that a partially spreadable sequence $\xi$ may not be Markovian. The inclusions of
$\sigma$-subalgebras 
\begin{eqnarray} \label{eq:sigma-alg-incl}
\begin{matrix}
\cF^{+}_{[n,n]} &\subset &\cF^{+}_{[0,\infty)}\\
\cup  &        & \cup \\
\sigma\{\xi_n\}   & \subset  & \Sigma^{\xi}
\end{matrix}
\end{eqnarray}
are always valid for all $n \ge 0$, but may not ensure that the sequence $\xi$ is Markovian. A
sufficient condition for the latter is that $\Sigma^\xi$ and $ \cF^{+}_{[n,n]}$ are
conditionally independent over $\sigma\{\xi_n\}$ for all $n \ge 0$, for example. Such conditions,
ensuring Markovianity of a sequence, are studied in greater detail in Subsection
\ref{subsection:Markov-F} in the operator algebraic framework of commuting squares.  

All inclusions \eqref{eq:sigma-alg-incl} are valid of course in the particular situation where 
the canonical local filtration $\Sigma_\bullet^\xi$ of the sequence and the local Markov filtration
$\cF^+_\bullet$ coincide, i.e.~$\Sigma_I^{\xi} = \cF^+_I$ for all `intervals' $I \in
\cI(\Nset_0)$. Combining Lemma \ref{lemma:ms-mf-equivalence} and Theorem \ref{theorem:cl-ps-Markov}
one arrives immediately at the following result.
\begin{Corollary}
A partially spreadable sequence $\xi$ is Markovian if $\Sigma^\xi_\bullet \equiv \cF^+_\bullet$. 
\end{Corollary}
Actually Theorem \ref{theorem:cl-ps-Markov} can be strengthened by considering the local filtration 
$\cG_\bullet^+ \equiv \big\{  \cG_I^+  
   := \sigma\{\eta_0 ^{i}(\cJ_0)\mid i\in I\}\big\}_{I \in \cI(\Nset_0)}$
which is coarser than $\cF_\bullet^+$. 
\begin{Theorem}\label{theorem:cl-ps-Markov-strong}
A partially spreadable sequence $\xi$ is stationary and adapted to the local Markov filtration
$\cG_{\bullet}^{+}$. 
\end{Theorem}
The proof of this strengthened result is given in Subsection \ref{subsection:Markov-F} in the
noncommutative framework. As before, one needs to stipulate additional conditions to ensure that
$\xi$ is Markovian. Sufficient conditions are now obtained by replacing $\cF_\bullet^+$ by
$\cG_\bullet^+$ in \eqref{eq:sigma-alg-incl} and stipulating conditional independence over
$\sigma\{\xi_n\}$ of $\Sigma^\xi$ and $ \cG^{+}_{[n,n]}$ for all $n \ge 0$. Actually these
conditions can be further strengthened so that one arrives at the following result which
essentially features Theorem \ref{theorem:markov-filtration-1} in the traditional language of
probability theory. 
\begin{Theorem}\label{theorem:cldeFi-with condition}
Suppose the partially spreadable sequence $\xi$ is such that the
$\sigma$-algebras $\Sigma^{\xi}$ and $\cJ_0$ are conditionally independent
over $\sigma\{\xi_0\}$. Then $\xi$ is a stationary Markov sequence. 
\end{Theorem}
Clearly this conditional independence condition in above theorem is satisfied whenever
$\sigma({\xi_0}) = \cJ_0$. This is precisely the additional condition which upgrades `partial
spreadability' to `maximal partial spreadability', as an inspection of Definition
\ref{definition:partspread-c} easily reveals. 

The above results have not so far addressed whether a stationary Markov sequence can be seen to
be (maximal) partially spreadable. An affirmative answer to this question in classical
probability requires the construction of a representation of the Thompson monoid $F^+$ from a
stationary Markov sequence. An intermediate step towards the construction of such a
representation is provided by the construction of a certain Markov dilation, as it is discussed
in Subsection \ref{subsection:markov-dilations} below. Altogether this allows us to establish a
de Finetti theorem for stationary Markov sequences in Theorem \ref{theorem:commdeFinetti} which 
was stated in its classical reformulation as Theorem \ref{theorem:definetti-intro}. We state it 
here again for the convenience of the reader.
\begin{Theorem*}
Let $\xi \equiv (\xi_n)_{n\in\Nset_0}$ be a sequence of $(\Omega_0, \Sigma_0)$-valued random
variables on the standard probability space $(\Omega,\Sigma,\mu)$. Then the following are
equivalent:
\begin{enumerate}
     \item[(a)] $\xi$ is maximal partially spreadable;
     \item[(b)] $\xi$ is a stationary Markov sequence.
\end{enumerate}
\end{Theorem*}
A proof of this theorem is obtained from transferring the proof of Theorem \ref{theorem:commdeFinetti}  
to the traditional setting in terms of classical random variables. 
\begin{Remark} \normalfont
The extended de Finetti theorem characterizes spreadable sequences of random variables as
mixtures of sequences of conditionally independent, identically distributed random variables. 
A similar complete characterization of partially spreadable sequences of random variables 
is presently an open problem. It is tempting to speculate that any stationary sequence $\xi$ is
partially spreadable. Such a conjecture is supported by the fact that a function of a stationary
Markov sequence may fail to be Markovian. On the other hand, a function of a partially spreadable
sequence yields again a partially spreadable sequence. We will further discuss this phenomenon in
Subsection \ref{subsection:FunctionsOfMarkov}. Finally, let us briefly remark that mixtures of
stationary Markov sequences may also be considered to emerge from certain functions of a Markov
sequence.  
\end{Remark}
\subsection{Functions of classical stationary Markov sequences and their algebraization}
\label{subsection:FunctionsOfMarkov}
We motivate  why partial spreadability characterizes a much larger class of stationary stochastic
processes than that of stationary Markov sequences. For this purpose,  consider the sequence
$\zeta \equiv \{\zeta_n:=f\circ \xi_n\}_{n=0}^{\infty}$ for some (measurable) function 
$f:(\Omega_0,\Sigma_0)\to (\Omega_0,\Sigma_0)$ of the stationary sequence $\xi$ (as above), 
and denote by $\Sigma^\zeta_{\bullet} \equiv \{\Sigma^{\zeta}_I\}_{I \subset {\Nset_0}}^{}$ the
canonical local filtration of $\zeta$.  Then it is well-known that stationarity and stochastic
independence of $\xi$ transfer to the sequence $\zeta$, but Markovianity may not (for example see
\cite{BR58} and \cite[page 59]{Ro74}). If the function $f$ is injective then the Markovianity of
$\Sigma^{\xi}_{\bullet}$ implies the Markovianity of $\Sigma^{\zeta}_{\bullet}$. However, in
general we only get the inclusions $ \Sigma^{\zeta}_{\bullet}\subset
\Sigma^{\xi}_{\bullet}\subset \Sigma$. Thus the local filtration $\Sigma^{\zeta}_{\bullet}$ may be
coarser than the local Markov filtration $\Sigma^{\xi}_{\bullet}$ and, in general, one can not
expect the sequence $\zeta$ to be Markovian, that is, the canonical local filtration
$\Sigma^{\zeta}_{\bullet}$ may not be Markovian. But partial spreadability is robust when passing
from $\xi$ to $\zeta$, in contrast to Markovanity. 
\begin{Proposition}\label{proposition:partial-spread-functions}
Suppose $\xi$ is partially spreadable. Then $\zeta$ is partially spreadable.
\end{Proposition}
\begin{proof}
Since $\zeta_n = f \circ \xi_n$, it follows from the partial spreadability of $\xi$ that  
$f \circ \xi_0 = f \circ \xi_0 \circ \eta_k$ and $f \circ \xi_n = f \circ \xi_0 \circ \eta_0^n$
for all $k, n >0$. 
\end{proof}
Consequently, the class of partially spreadable sequences of random variables is much larger than
the class of stationary Markov sequences. We recall that $(\Omega_0,\Sigma_0)$ denotes a standard
Borel space and $(\Omega, \Sigma, \mu)$ a standard probability space. 
\begin{Theorem}\label{theorem:functionMarkovPC}
Let $\xi \equiv (\xi_n)_{n\in\Nset_0}$ be a sequence of $(\Omega_0, \Sigma_0)$-valued random
variables on $(\Omega,\Sigma,\mu)$. If $\xi$ is a stationary Markov sequence then its function 
sequence $\zeta \equiv (\zeta_n:= f\circ \xi_n)_{n=0}^\infty$ is partially spreadable for any
(measurable) function $f \colon (\Omega_0,\Sigma_0) \to  (\Omega_0,\Sigma_0)$. 
\end{Theorem}
\begin{proof}
A stationary Markov sequence is (maximal) partially spreadable by Theorem
\ref{theorem:definetti-intro}. Thus Proposition \ref{proposition:partial-spread-functions}
applies.
\end{proof}
\begin{Remark}\normalfont
Maximal partial spreadability of the sequence $\xi$ may not be
inherited by $\zeta$, for the same reason that a function of a
Markov sequence may be non-Markovian. 
\end{Remark}
We close this subsection by addressing how passing from the sequence $\xi$ to its function
sequence $\zeta$ reappears as restricting the corresponding noncommutative random variables.
Continuing our discussion from above, upon the algebraization of classic random variables, the
inclusions of local filtrations
\[
\Sigma^{\zeta}_{\bullet}\subset  \Sigma^{\xi}_{\bullet}\subset \Sigma
\]
become inclusions of Lebesgue spaces of essentially bounded functions: 
\[
L^{\infty}(\Omega, \Sigma^{\zeta},\mu_\zeta)\subset L^{\infty}(\Omega, \Sigma^{\xi},\mu_\xi)
\subset L^{\infty}(\Omega, \Sigma,\mu),
\]  where $\mu_{\zeta}$ denotes the restrictions of $\mu$ to $\Sigma^{\zeta}$. Thus taking a 
function of a sequence of random variables becomes restriction of a sequence of noncommutative
random variables. Altogether this entails that a Markovian sequence of noncommutative random
variables 
\[
\iota \colon (\cL_0, \trace_{\mu_0}) \to   (\cL, \trace_{\mu}) 
\]
restricts to a sequence 
\[
\iota^{\zeta} \colon 
(\widetilde{\cL}_{0}, \trace_{\widetilde{\mu}_{0}}) \to   (\cL, \trace_{\mu}), 
\]
which may fail to be Markovian with respect to its canonical local filtration. Here we have put 
$\widetilde{\cL}_{0} := L^{\infty}(\Omega,\Sigma^{{\zeta}_{0}}, \widetilde{\mu}_{0}) \subset
\cL_0= L^{\infty}(\Omega,\Sigma^{{\xi}_{0}},\mu_0)$, where $\widetilde{\mu}_{0}$ denotes the
restriction of $\mu_{0}:=\mu_{\xi_0}$ from $\Sigma^{\xi_{0}}$ to $\Sigma^{\zeta_{0}}$.
Consequently, in reference to Definition \ref{definition:markov-filtration-c}, the Markovianity of the
local filtration $\Sigma^\xi_\bullet$ may not pass to the local filtration $\Sigma^\zeta_\bullet$. 
Of course, this may be reformulated in terms of commuting squares of commutative von Neumann algebras,
see Remark \ref{remark:failure-mark-filt}.
\subsection{A result on Markov dilations in classical probability theory}
\label{subsection:markov-dilations}
Presently it is an open problem if a representation of the Thompson monoid $F^+$ can be constructed 
directly based on the usual Daniell-Kolmogorov construction for stationary Markov processes. This issue is 
bypassed by rewriting a Markov process as an open dynamical system, a construction which is well-known within
quantum probability \cite{AFL82}. Here we review this construction in the framework of classical probability
along the presentation given in the habilitation thesis of K\"ummerer \cite{Ku86}.
This construction is underlying the construction of a representation of the Thompson monoid $F^+$ in 
Subsection \ref{subsection:constr-class-prob}. 
\begin{Notation} \normalfont
For the remainder of this subsection the noncommutative probability space $(\cL, \trace_\lambda)$
is given by the Lebesgue space of essentially bounded functions $\cL := L^\infty([0,1],\lambda)$ 
and $\trace_\lambda := \int_{[0,1]} \cdot\,  d\lambda$ as the faithful normal state on $\cL$. 
Here $\lambda$ denotes the Lebesgue measure on the unit interval $[0,1] \subset \Rset$. 
\end{Notation}
\begin{Theorem}[{\cite[4.4.2]{Ku86}}] \label{theorem:kuemmerer-onesided}
Let $T$ be a $\varphi$-Markov map on $\cA$, where $\cA$ is a commutative von 
Neumann algebra with separable predual. Then there exists 
$\alpha \in \End(\cA \otimes \cL, \varphi \otimes \trace_\lambda)$ 
such that $(\cA\otimes \cL, \varphi\otimes \trace_{\lambda}, \alpha, \iota_0)$ is a Markov
dilation of $T$. That is, $(\cA\otimes \cL, \varphi\otimes \trace_{\lambda}, \alpha, \cA\otimes
\1_{\cL})$ is a stationary Markov process, and for all $n \in \Nset_0$, 
\[
T^n   =   \iota_0^* \, \alpha^n \iota_0,
\]
where $\iota_0 \colon (\cA,\varphi) \to (\cA \otimes \cL, \varphi \otimes \trace_\lambda)$
denotes the canonical embedding $\iota_0(a) = a \otimes \1_\cL$ such
that $E_0 := \iota_0 \circ \iota_0^* $ is the $\varphi \otimes \trace_\lambda$-preserving 
normal conditional expectation from $\cA \otimes \cL$ onto $\cA \otimes \1_{\cL}$.
\end{Theorem}
For the convenience of the reader we outline next a proof of this folkore result in
noncommutative probability. Our arguments follow almost verbatim those given by K\"ummerer in his
habilitation thesis \cite{Ku86}, adapting therein arguments from the setting of bilateral to that
of unilateral noncommutative stationary processes. 
\begin{proof}
The Daniell-Kolmogorov construction gives a dilation $(\cM,\psi,\beta,\kappa)$ of the 
$\varphi$-Markov operator $T$ on $\cA$ in the following way. A state $\psi_{\operatorname{alg}}$
on the infinite algebraic tensor product $\cM_{\operatorname{alg}}:= \odot_{\Nset_0} \cA$ is
defined through 
\[ 
\psi_{\operatorname{alg}}(a_0\otimes a_1\otimes 
      \cdots \otimes a_n \otimes \1_\cA \otimes \cdots)
=\varphi\big(a_0T(a_1T(a_2\cdots T(a_{n-1}T(a_n))\cdots))\big).
\]
Next one applies the GNS construction to the pair $(\cM_{\alg}, \psi_{\alg})$, a so-called
*-algebraic probability space, to obtain the noncommutative probability space $(\cM,\psi)$ (as
introduced in Subsection \ref{subsection:Markov-maps}) such that 
$\cM = \Pi_{\psi}(\cM_\alg)^{\prime\prime}$ and $\psi \big(\Pi_{\psi}(x)\big) 
= \langle 1_\psi , \Pi_\psi(x) 1_\psi \rangle = \psi_\alg(x)$. Here $(\cH_\psi,\Pi_\psi,1_\psi)$
denotes the GNS triple of $(\cM_\alg, \psi_\alg)$ and $\Pi_{\psi}(\cM_\alg)^{\prime\prime}$ the
double commutant of $\Pi_{\psi}(\cM_\alg)$ in $\cB(\cH_\psi)$. We remark that the normal state
$\psi$ on $\cM$ is automatically faithful, as $\psi$ is tracial. The right shift defined on
$\cM_\alg = \odot_{\Nset_0} \cA$ by $a_0\otimes a_1\otimes \dots  \mapsto \1_\cA\otimes
a_0\otimes a_1\otimes \cdots$ extends in the GNS representation to an endomorphism $\beta$ of
$(\cM,\psi)$ such that $\beta\big(\Pi_\psi(a_0 \otimes a_1 \otimes \cdots  )\big) =
\Pi_\psi(\1_\cA \otimes a_0 \otimes a_1 \otimes \cdots  )$. 
Moreover $\kappa (a) := \Pi_{\psi}(a\otimes \1_\cA\otimes \1_\cA\otimes \cdots)$ defines a
noncommutative random variable $\kappa$ from $(\cA,\varphi)$ into $(\cM,\psi)$ such that 
$(\cM,\psi,\beta,\kappa)$ is a minimal Markov dilation of the $\varphi$-Markov map $T$ 
on $\cA$. We stress here that the endomorphism $\beta$ is simply the right shift on tensor
products, and all information about the Markov map $T$ is encoded into the state $\psi$.

As $(\cM,\psi,\beta,\kappa)$ is a dilation of the $\varphi$-Markov map $T$
on $\cA$, so is its amplification  to $(\cM \otimes \cL, \psi \otimes
\trace_{\lambda}, \beta \otimes \Id_\cL, \kappa_\cM \circ \kappa)$,  where
$\kappa_\cM$ is the canonical embedding of $\cM$ into $\cM\otimes \cL$.  
We show next that the two noncommutative probability spaces 
$(\cM \otimes \cL, \psi \otimes \trace_{\lambda})$ and $(\cA \otimes \cL,
\varphi \otimes \trace_{\lambda})$ are isomorphic. More explicitly, we will show
the existence of an isomorphism $\pi \colon \cM \otimes \cL \to \cA \otimes
\cL$ such that $\pi\big(\kappa(a)\otimes \1_\cL)= a\otimes \1_\cL =
\iota_0(a)$ and $\psi \otimes \trace_\lambda = (\varphi \otimes \trace_\lambda)
\circ \, \pi$. Consequently $\alpha := \pi \circ \beta \circ \pi^{-1}$ will
define an endomorphism of $(\cA \otimes \cL, \varphi \otimes \trace_{\lambda})$ 
such that $(\cA \otimes \cL, \varphi \otimes \trace_\lambda, \alpha, \iota_0)$ is a 
Markov dilation of the $\varphi$-Markov map $T$ on $\cA$. 

As $\cA$ is a commutative von Neumann algebra with separable predual, it can
identified with $L^{\infty}(\Gamma,\mu)$ where $\Gamma$ is a standard Borel
space and $\mu$ is a probability measure on $\Gamma$ which induces $\varphi$
(that is, $\varphi=\int_{\Gamma} \cdot \, d\mu$). Moreover $\cM$ also has
separable predual, since $\cM =\bigvee\{\beta^n \kappa(\cA)\mid n\in
\Nset_0\}$ by the minimality of the dilation.  Consequently we can assume that
$\cM$ is faithfully represented on a separable Hilbert space $\cH$ with cyclic
separating vector $\xi$ inducing the state $\psi$.

Applying the theory of desintegration (\cite[Chapter IV]{Ta79}), we obtain a measurable field 
of commutative von Neumann  algebras $\{\{\cM(\gamma), \cH(\gamma)\}  \colon \gamma\in \Gamma\}$
such that 
\[
\{\cM,\cH\}=\int^{\oplus}_{\Gamma} \{\cM(\gamma), \cH(\gamma)\} \, d\mu(\gamma),
\]
with $\kappa(\cA)\cong L^{\infty}(\Gamma,\mu)$ as the diagonal subalgebra. Let 
\[
\xi=\int_{\Gamma}^{\oplus} \xi(\gamma) \, d\mu(\gamma), \qquad 
\psi =\int_{\Gamma}^{\oplus} \psi(\gamma) \, d\mu(\gamma)
\] 
be the corresponding desintegration of the vector $\xi$ and the state $\psi$
(see \cite[Chapter IV, Definitions 8.14, 8.15, 8.33]{Ta79}). It follows from the uniqueness
of desintegration (\cite[Chapter IV, Theorem  8.34]{Ta79} that, for $x\in \cA(\gamma)$, 
\[
\varphi(\gamma)(x)=\langle \xi(\gamma),x \xi(\gamma)\rangle  
\]
for $\mu$-almost all $\gamma \in \Gamma$. In particular, $\xi(\gamma)$ is cyclic and separating
for $\mu$-almost all $\gamma\in \Gamma$. Similarly, with $ \cL = L^\infty([0,1], \lambda)$ and
$\cK = L^2([0,1], \lambda)$,  one has the unique desintegration
\[
\{\cM\otimes \cL, \cH\otimes\cK\}
   =\int_{\Gamma}^{\oplus} \{\cM(\gamma)\otimes \cL, \cH(\gamma)\otimes \cK \}\, d\mu(\gamma),
\]
and 
\[\psi\otimes \trace_{\lambda}
    =\int_{\Gamma}^{\oplus} \psi(\gamma) \otimes \trace_{\lambda}\, d\mu(\gamma).
\]
For almost all $\gamma\in \Gamma$, the noncommutative probability space $\big(\cM(\gamma) 
\otimes \cL, \varphi(\gamma)\otimes \trace_{\lambda}\big)$ has a von Neumann algebra without
atoms acting on a separable Hilbert space. Hence, $\cM(\gamma)\otimes \cL$ is isomorphic to $\cL$
for almost all $\gamma\in \Gamma$ (\cite[Chapter III, Theorem 1.22]{Ta79}). Under this
isomorphism, the state $\psi(\gamma)\otimes \trace_{\lambda}$ is mapped into some faithful normal
state on $\cL$ and we can even find an isomorphism mapping $\psi(\gamma)\otimes \trace_{\lambda}$
into $\trace_{\lambda}$ (see \cite[Theorem 15.3.9]{Ro68}). Altogether we get an isomorphism
$\pi(\gamma)$ between $(\cM(\gamma)\otimes \cL, \psi(\gamma)\otimes \trace_{\lambda})$ and
$(\cL,\trace_{\lambda})$ for almost all $\gamma \in \Gamma$. In particular, $\pi(\gamma)$ is even
a spatial isomorphism between the pairs $\{\cM(\gamma) \otimes \cL, \cH(\gamma) \otimes \cK\}$
and $\{\cL, \cK\}$. Therefore, applying \cite[Section IV, Theorem 8.28]{Ta79}, we obtain an
isomorphism between
\[
\left(\int_{\Gamma}^{\oplus}\{\cM(\gamma)\otimes\cL,\, \cH(\gamma)\otimes\cK\} \,d\mu(\gamma) \,,
      \int_{\Gamma}^{\oplus} \psi(\gamma) \otimes \trace_\lambda\,d\mu(\gamma) \,  \right )
\]
and
\[
\left(\int_{\Gamma}^{\oplus}\{\cL,\,\cK \}\, d\mu(\gamma)\, ,\,  
      \int_{\Gamma}^{\oplus} \trace_\lambda\,d\mu(\gamma) \,\right )
\]
which, by \cite[Chapter IV, Theorem 8.30]{Ta79}, in turn is isomorphic to
\[
(L^{\infty}(\Gamma,\mu)\otimes \cL,\lambda),\varphi\otimes \trace_{\lambda})
   =(\cA\otimes \cL,\varphi\otimes \trace_{\lambda}).
\]
This establishes the existence of an isomorphism $\pi$ between the noncommutative probability
spaces  $(\cM \otimes \cL, \psi \otimes \trace_\lambda)$ and 
$(\cA \otimes \cL, \varphi \otimes \trace_{\lambda})$. 
\end{proof}
As pointed out by K\"ummerer in \cite{Ku86}, Theorem \ref{theorem:kuemmerer-onesided} ensures the existence
of a Markov dilation, but does not provide an explicit construction.  Also, its proof rests on the 
traditional Daniell-Kolmogorov construction for a stationary process. But a concrete construction is available
for stationary Markov chains with values in a finite set which, moreover, is independent of the 
Daniell-Kolomogorov construction. For the convenience of the reader, we demonstrate this for the algebraic
reformulation of a stationary $\{0,1\}$-valued Markov chain, reproducing the arguments of K\"ummerer in 
\cite[4.4.3]{Ku86}.   
\begin{Example}\normalfont \label{example:M-D}
An algebraic model for a coin toss is given by the noncommutative probability space 
$(\Cset^2,\varphi_q)$, where $\Cset^2$ represents the two-dimensional unital *-algebra 
of all $\Cset$-valued functions on the set $\{0,1\}$, and for some fixed $q \in (0,1)$, 
\[
\varphi_q\big( \ba \big) = q a_1 + (1-q) a_2,  
\qquad  
\text{with $\ba = \begin{bmatrix} a_1\\a_2 \end{bmatrix} \in \Cset^2$,}
\]
defines a faithful (tracial) state $\varphi_q$ on $\Cset^2$. Now consider the transition matrix 
$T \colon \Cset^2 \to \Cset^2$ given by
\[
T = \begin{bmatrix} 1-p_1 & p_1 \\ p_2 & 1-p_2 \end{bmatrix}
\]
for some fixed $p_1, p_2 \in [0,1]$. For some fixed $0 <q < 1$, this map $T$ is a $\varphi_q$-Markov map (in the sense of 
Definition \ref{definition:MarkovMap}) if and only if the pair $(p_1, p_2)$ satisfies the following two conditions:
\begin{enumerate}
    \item 
    $p_2 (1-q) = p_1 q$  \qquad (stationarity of the map);
    \item
    either $p_1= p_2 = 0$ or
    $p_1, p_2 >0$ (faithfulness of the state).
\end{enumerate}
Note that $p_1= p_2=0$ corresponds to $T$ being the identity matrix, and $p_1=p_2=1$ to $T$ being the flip, both of which are automorphisms on $\Cset^2$. Thus these two cases yield trivial stationary Markov dilations. For notational convenience we will exclude these two cases in the following discussion.  

To start the construction of a stationary Markov dilation (in the sense of K\"ummerer), 
let the noncommutative
probability space $(\cL,\trace_{\lambda})$ be given by 
 \[
\cL  = L^\infty([0,1], \lambda)  
\quad \text{and}  \quad 
\trace_{\lambda} = \int_{[0,1]} \hspace{-12pt} \cdot  \,\,d\lambda,
\]
where $\lambda$ denotes the Lebesgue measure on $[0,1]$. The goal is to construct an automorphism 
$C\in \Aut(\Cset^2 \otimes \cL,  \varphi_q \otimes \trace_{\lambda})$ such that 
$\iota^* \circ  C  \circ \iota(\ba) = T(\ba)$, where the embedding 
$\iota \colon  \Cset^2 \to \Cset^2 \otimes \cL$ is given by $\iota (\ba) = \ba \otimes \1_\cL$. 
For this purpose, let 
\begin{align*}
(\Omega_1, \lambda_1)&:= \big([0,1], q \lambda\big)\\    
(\Omega_2, \lambda_2)&:= \big([0,1], (1-q) \lambda\big)
\end{align*}
and define on the disjoint union $\Omega_1 \sqcup \Omega_2$ the probability measure $\mu$ such that 
\[
\mu(A) = \lambda_1(A \cap \Omega_1)
+ \lambda_2(A \cap \Omega_2).
\]
In the following we canonically identify the two (isomorphic) noncommutative probability spaces 
$(\Cset^2 \otimes \cL,  \varphi_q \otimes \trace_{\lambda})$ and
$\big(L^\infty(\Omega_1 \sqcup \Omega_2, \mu),  \int_{\Omega_1 \sqcup \,\Omega_2} \, \cdot\, d\mu \big)$ 
such that
\[
 \varphi_q \otimes \trace_{\lambda} (\ba \otimes f)
 = a_1 \int_{\Omega_1} f d\lambda_1 + 
   a_2 \int_{\Omega_2} f d\lambda_2. 
\]
Thus, under this identification, the embedding $\iota$ reads as 
\[
\iota(\ba) = a_1 \chi_{\Omega_1}^{} + a_2 \chi_{\Omega_2}^{},  
\]
where $\chi_A$ denotes the characteristic function of the measurable set $A \in \Omega_1 \sqcup \Omega_2$,
and its adjoint $\iota^*$ becomes
\[
\iota^*(f)
= \begin{bmatrix}
\int_{\Omega_1} f d\lambda_1\\
\int_{\Omega_2} f d\lambda_2
\end{bmatrix}.
\]
For the construction of the automorphism $C$, we further partition
\begin{align*}
\Omega_1 & = \Omega_{1,1} \sqcup \Omega_{1,2} :=  [0,1-p_1] \sqcup (1-p_1,1],\\
\Omega_2 &=  \Omega_{2,1} \sqcup \Omega_{2,2} :=[0,p_2) \sqcup [p_2,1],
\end{align*}
to obtain
\begin{alignat*}{6}
\mu(\Omega_{1,1}) & = (1-p_1) q, 
&\qquad& 
\mu( \Omega_{1,2}) & = & p_1 q,&\\
\mu( \Omega_{2,1}) & =p_2(1-q), 
&\qquad&
\mu(\Omega_{2,2}) &=& (1-p_2) (1-q).&
\end{alignat*}
This way of partitioning is motivated by
\[
\mu( \Omega_{1,2})  =  \mu( \Omega_{2,1}) 
\quad \Longleftrightarrow \quad 
p_1 q= p_2 (1-q),
\]
where the righthand side of this equivalence is the concrete form of the stationarity condition 
$\varphi_q \circ T = \varphi_q$. Thus there exists a bijective $\mu$-preserving measurable map 
$\tau$ on $\Omega_1 \sqcup \Omega_2$ such that 
\[
\tau(\Omega_{1,1})  =  \Omega_{1,1}, \qquad
\tau(\Omega_{2,2})  =  \Omega_{2,2},\qquad 
\tau(\Omega_{1,2})  =  \Omega_{2,1}.
\]
We choose $\tau$ to be the identity map when restricted to $\Omega_{1,1} \cup \Omega_{2,2}$. Let $C$ 
denote the automorphism on $L^\infty(\Omega_1 \sqcup \Omega_2, \mu)$ induced by $\tau$, such that
$C(\chi_\cA^{}) = \chi_{\tau(A)}^{}$. We claim that  $\iota^* \circ  C  \circ \iota(\ba) = T(\ba)$. Indeed,
\begin{align*}
\iota^* C \iota(\mathbf{a}) 
& =  \iota^* C \big(a_1 \chi_{\Omega_1}^{} + a_2 \chi_{\Omega_2}^{}\big)\\
& =  \iota^* C \big(a_1\chi_{\Omega_{1,1}}^{} + a_1\chi_{\Omega_{1,2}}^{}
                    + a_2 \chi_{\Omega_{2,1}}^{}+ a_2 \chi_{\Omega_{2,2}}^{}\big)\\
& =  \iota^* \big(a_1\chi_{\Omega_{1,1}}^{} + a_1\chi_{\Omega_{2,1}}^{}
                    + a_2 \chi_{\Omega_{1,2}}^{}+ a_2 \chi_{\Omega_{2,2}}^{}\big)\\
&= \begin{bmatrix}
a_1 (1-p_1) + a_2 p_1  \\
a_1  p_2 + a_2 (1-p_2) 
\end{bmatrix} 
 = T (\ba).
\end{align*}
As the invariance property $\int_{\Omega_1 \sqcup \Omega_2}C (f) d\mu = \int_{\Omega_1 \sqcup \Omega_2}f d\mu$
is evident from the definition of $C$, we have verified that $C\in \Aut(\Cset^2 \otimes \cL,  
\varphi_q \otimes \trace_{\lambda})$ (after the canonical identification made above).  

We are left to provide a probability space $(\cM, \psi)$ and an endomorphism 
$\alpha \in \End(\cM, \psi)$ such
that $(\cM, \psi, \alpha, \iota)$ is a Markov dilation of $T$.
Let $\cM := \Cset^2 \otimes \cL^{\otimes_{\Nset_0}}$ and 
$\psi := \varphi_q \otimes \trace_{\lambda}^{\otimes_{\Nset_0}}$. Then the endomorphism 
$\alpha$ on $\Cset^2 \otimes \cL^{\otimes_{\Nset_0}}$
is uniquely defined by the $\Cset$-linear extension of the map
\[
\ba \otimes f_0 \otimes f_1 \otimes \cdots  
\mapsto 
C(\ba \otimes \1_{\cL}) \otimes f_0 \otimes f_1 \otimes \cdots 
\]
such that $\iota^* \alpha^n \iota = T^n$ for all $n \in \Nset_0$.
We will draw on this approach for the concrete construction of
noncommutative stationary Markov processes in Section
\ref{section:Constructions of Reps of F+}. 
\end{Example}

\begin{Remark}\normalfont
Using the terminology of K\"ummerer \cite{Ku86}, the construction of the
automorphism $C$ in Example \ref{example:M-D} provides us with a `tensor
dilation of first order'. Such a `dilation of first order' can always be further
upgraded to a Markov dilation, as illustrated in Example \ref{example:M-D}.
Note that Theorem \ref{theorem:kuemmerer-onesided}, as a result
on the existence of a Markov dilation, includes of course
the existence of a `Markov dilation of first order' which can be further
amplified as done in Example \ref{example:M-D}. We will take this
viewpoint in Theorem \ref{theorem:F+-gen-compression}.
\end{Remark}

\section{\texorpdfstring{Markovianity from representations of the Thompson monoid $F^+$}{}}
\label{section:markovianity-from-F+}
The noncommutative de Finetti theorem, Theorem \ref{theorem:ncdf}, rests on the result that
representations of the partial shift monoid $S^+$ on noncommutative probability spaces provide
rich structures of commuting squares, in particular as they are underlying the notion of
noncommutative Bernoulli shifts in Definition \ref{definition:ncbs}. Here we investigate
commuting square structures as they emerge from representations of the Thompson monoid $F^+$ on
noncommutative probability spaces. Our investigations reveal that certain commuting squares, as
available in triangular towers of inclusions, already encode Markovianity. As the partial shift
monoid $S^+$ is a quotient of the Thompson monoid $F^+$, our approach yields de-Finetti-type
results for noncommutative stationary Markov processes. 

Let us fix some notation, as it will be used throughout this section. We assume that the
probability space $(\cM,\psi)$ is equipped with the representation $\rho \colon F^+ \to
\End(\cM,\psi)$. For brevity of notion, especially in proofs, the represented generators of $F^+$
are also denoted by
\[
\alpha_n := \rho(g_n)  \in \End(\cM,\psi),   
\]
with fixed point algebras given by  $\cM^{\alpha_n} := \{x \in \cM \mid \alpha_n (x) = x\}$, 
for  $0 \le n < \infty$.  Furthermore the intersections of fixed point algebras
\[
\cM_n := \bigcap_{k \ge n +1} \cM^{\alpha_k}
\] 
give the tower of von Neumann subalgebras 
\[
\cM_{-1} \subset \cM_0 \subset \cM_1 \subset  \cM_2 \subset   \ldots  
\subset \cM_\infty  := \bigvee_{n \ge 0} \cM_n  \subset \cM. 
\]
From the viewpoint of noncommutative probability theory, this tower provides a filtration of the 
noncommutative probability space $(\cM,\psi)$. In particular, we will see in Subsection 
\ref{subsection:s-f-p-a} that the inclusions 
 \[
\begin{matrix}
\cM_{m} &\subset &\cM \\
\cup  &        & \cup \\
\alpha_0^m(\cM_0)   & \subset  & \alpha_0^m(\cM_\infty)
\end{matrix}
\]
form commuting squares which encode Markovianity. Consequently the canonical local filtration of 
a stationary process $(\cM, \psi, \alpha_0, \cA_0)$ will be seen to be a local subfiltration of a local
Markov filtration whenever the $\psi$-conditioned von Neumann subalgebra $\cA_0$ is
well-localized, to be more precise: contained in the intersection of fixed point algebras
$\cM_0$. It is worthwhile to emphasize that, depending on the choice of the generator $\cA_0$,
the canonical local filtration of this stationary process may not be Markovian. Subsection
\ref{subsection:Markov-F} investigates in detail conditions under which the canonical local filtration
of a stationary process $(\cM, \psi, \alpha_0, \cA_0)$ is Markovian. Finally, Subsection 
\ref{subsection:ncdf} provides the proof of Theorem \ref{theorem:extended-de-Finetti}, a
noncommutative de Finetti theorem as appropriate for noncommutative stationary Markov processes.  
\subsection{Representations with a generating property}
\label{subsection:generating}
An immediate consequence of the relations between generators of the Thompson monoid $F^+$ is 
the adaptedness of the endomorphism $\alpha_0$ to the tower of (intersected) fixed point
algebras:
\[
\alpha_0(\cM_{n}) \subset \cM_{n+1} \qquad \text{for all $n \in \Nset_0$}.
\]
Thus, generalizing terminology from classical probability, the random variables
\begin{alignat*}{2}
\iota_0 &:= \Id|_{\cM_0} &\colon \cM_0 \to \cM_0 \subset \cM\\
\iota_1 &:= \alpha_0|_{\cM_0} &\colon \cM_0 \to \cM_1 \subset \cM\\
\iota_2 &:= \alpha^2_0|_{\cM_0} &\colon \cM_0 \to \cM_2 \subset \cM\\
 & \qquad \vdots\\
\iota_n &:= \alpha^n_0|_{\cM_0} &\colon \cM_0 \to \cM_n \subset \cM
\end{alignat*}
are adapted to the filtration $\cM_0 \subset \cM_1 \subset \cM_2 \subset \ldots$ and $\alpha_0$
is the time evolution of the stationary process $(\cM,\psi, \alpha_0, \cM_0)$. We refer the
reader to \cite[Chapter 3]{Go04} for more information on the general philosophy of adapted
endomorphims and to \cite[Appendix A]{GK09} or \cite{EGK17} on how adaptedness 
is of relevance within the context of distributional symmetries and invariance principles. 

Clearly, at most, the von Neumann subalgebra $\cM_\infty$ can be generated by this sequence of
random variables $(\iota_n)_{n \ge 0}$. An immediate question is whether a representation of the 
Thompson monoid $F^+$ restricts to the von Neumann subalgebra $\cM_\infty$. 
\begin{Definition}\label{definition:generating} \normalfont
The representation  $\rho \colon F^+ \to \End(\cM,\psi)$ is said to have the \emph{generating
property} if $\cM_\infty = \cM$.  
\end{Definition}
As shown in Proposition \ref{proposition:generating-property} below, this generating property
entails that each intersected fixed point algebra $\cM_n  = \bigcap_{k \ge n+1} \cM^{\alpha_k}$ 
equals the single fixed point algebra $\cM^{\alpha_{n+1}}$. Thus the generating property 
tremendously simplifies the form of the tower $\cM_0 \subset \cM_1 \subset \ldots$, and our next 
result shows that this can always be achieved by restriction.     
\begin{Proposition}\label{proposition:generatingrestriction}
The representation $\rho:F^{+}\to \End(\cM,\psi)$ restricts to the generating representation 
$\rho_{\gen}:F^{+}\to \End(\cM_{\infty},\psi_\infty)$ such that $\alpha_n (\cM_\infty) \subset
\cM_\infty$ and $E_{\cM_\infty} E_{\cM^{\alpha_n}} = E_{\cM^{\alpha_n}} E_{\cM_\infty}$ for all
$n \in \Nset_0$.  Here $\psi_\infty$ denotes the restriction of the state $\psi$ to $\cM_\infty$.
$E_{\cM^{\alpha_n}}$ and  $E_{\cM_{\infty}}$ denote the unique $\psi$-preserving normal
conditional expectations onto $\cM^{\alpha_n}$ and $\cM_{\infty}$ respectively.
\end{Proposition}
\begin{proof}
We show that $\alpha_i (\cM_n ) \subset \cM_{n+1}$ for all $i,n \ge 0$. Let $x \in \cM_n$. If 
$i \ge n+1$ then $\alpha_i (x) = x$  is immediate from the definition of $\cM_n$.  If $i < n+1$
then, using the relations for the generators of the Thompson monoid, 
$\alpha_i(x)= \alpha_i\alpha_{k+1} (x) =   \alpha_{k+2}\alpha_i(x)$
for any $k \ge n$, thus $\alpha_i(x) \in \cM_{n+1}$.  Consequently $\alpha_i$ maps  
$\bigcup_{n \ge 0} \cM_n$ into itself for any $i \in \Nset_0$. Now a standard approximation
argument shows that $\cM_\infty$ is invariant under $\alpha_i$ for any $i \in \Nset_0$. 
Consequently the representation $\rho$ restricts to $\cM_\infty$ and, of course, this restriction
$\rho_{\gen}$ has the generating property. 

Since $\cM_\infty$ is globally invariant under the modular automorphism group of $(\cM,\psi)$,
there exists the (unique) $\psi$-preserving normal conditional expectation $E_{\cM_\infty}$ from
$\cM$ onto $\cM_\infty$. In particular, $\rho_{\gen}(g_n) = \alpha_n|_{\cM_\infty}$ commutes
with the modular automorphism group of $(\cM_\infty, \psi_\infty)$ which ensures
$\rho_{\gen}(g_n)\in  \End(\cM_{\infty},\psi_\infty)$. Finally that $E_{\cM_\infty}$ and
$E_{\cM^{\alpha_n}}$ commute is concluded from 
\[
E_{\cM_\infty}  \alpha_n E_{\cM_\infty} =   \alpha_n E_{\cM_\infty},
\]
which implies $  E_{\cM^{\alpha_n}}  E_{\cM_\infty} =    E_{\cM_\infty}E_{\cM^{\alpha_n}}$ by
routine arguments, and an application of the mean ergodic theorem (see for example 
\cite[Theorem 8.3]{Ko10}), 
\[
 E_{\cM^{\alpha_n}} = \lim_{N \to \infty} \frac{1}{N} \sum_{i=1}^{N} \alpha_n^i, 
\]    
where the limit is taken in the pointwise strong operator topology. 
\end{proof}
\begin{Lemma}\label{lemma:generating}
With the notations as above, $\cM_k=\cM^{\alpha_{k+1}}\cap \cM_{\infty}$ for all $k\in \Nset_0$.
\end{Lemma}
\begin{proof}
For the sake of brevity of notation, let $Q_n=E_{\cM^{\alpha_n}}$ denote the $\psi$-preserving
normal conditional expectation from $\cM$ onto $\cM^{\alpha_n}$. By the definition of $\cM_k$ 
and $\cM_\infty$, it is clear that $\cM_k\subset\cM^{\alpha_{k+1}}\cap \cM_{\infty}$. In order to
show the reverse inclusion, it suffices to show that $Q_n Q_k |_{\cM_\infty}=Q_k |_{\cM_\infty},
0\leq k<n<\infty$. We claim that, for $0\leq k < n$, 
\[
	Q_n Q_k |_{\cM_\infty} = Q_k |_{\cM_\infty}  \quad \Longleftrightarrow \quad Q_k Q_n Q_k |_{\cM_\infty} = Q_k |_{\cM_\infty}. 
\]
Indeed this equivalence is immediate from 
\begin{align*}
	\psi\big( (Q_nQ_k - Q_k)(y^*) (Q_nQ_k - Q_k)(x)\big) 
	& = \psi\big( y^* (Q_kQ_n - Q_k) (Q_nQ_k - Q_k)(x)\big) \\
	& =  \psi\big( y^* (Q_k - Q_kQ_n Q_k)(x)\big)
\end{align*}
for all $x,y \in \cM_\infty$.  We are left to prove $Q_k Q_n Q_k |_{\cM_\infty}= Q_k
|_{\cM_\infty}$ for $k < n$. For this purpose we express the conditional 
expectations $Q_k$ and $Q_n$ as mean ergodic limits in the pointwise strong operator topology 
and calculate
\begin{align*}
Q_{k}Q_{n}Q_{k}|_{\cM_\infty}
&=\lim_{M\to \infty}\lim_{N\to \infty}
   \frac{1}{MN}\sum_{i=1}^{M}\sum_{j=1}^{N} \alpha_{k}^{i}\alpha_{n}^{j}Q_{k}|_{\cM_\infty}\\
&=\lim_{M\to \infty}\lim_{N\to\infty}
   \frac{1}{MN} \sum_{i=1}^{M}\sum_{j=1}^{N}\alpha_{n+i}^{j}\alpha_{k}^{i}Q_{k}|_{\cM_\infty}\\
&=\lim_{M\to \infty}\lim_{N\to\infty}\frac{1}{MN}
   \sum_{i=1}^{M}\sum_{j=1}^{N}\alpha_{n+i}^{j}Q_{k}|_{\cM_\infty}\\
&=\lim_{M\to \infty}\frac{1}{M}
   \sum_{i=1}^{M}Q_{n+i}Q_{k}|_{\cM_\infty} = Q_{k}|_{\cM_{\infty}}.
\end{align*}
Here the last equality follows because for $x\in  \cM_\infty$, also $Q_k x\in  \cM_\infty$ and 
so it holds that $Q_{n+i} Q_k(x)=Q_k (x)$ for $i$ sufficiently large, thus 
\[
	\lim_{M\to\infty}\frac{1}{M}\sum_{i=1}^{M}Q_{n+i}=\operatorname{Id}
\]
in the pointwise strong operator topology.
\end{proof}
\begin{Corollary}\label{corollary:commsquare}
The following set of inclusions forms a commuting square for every $n\in \Nset_0$:
\[
	\begin{matrix}
	\cM^{\alpha_{n+1}} &\subset &\cM\\
	\cup  &        & \cup \\
	\cM_{n}   & \subset  & \cM_{\infty}
	\end{matrix}
\]
\end{Corollary}
\begin{proof}
Let $Q_n$ and $E_{\cM_{\infty}}$ be the $\psi$-preserving normal conditional expectation from
$\cM$ onto $\cM^{\alpha_n}$ and $\cM_{\infty}$ respectively for $n\in \Nset_0$. For 
$n\in \Nset_0$, by Proposition \ref{proposition:generatingrestriction},
$Q_{n+1}E_{\cM_{\infty}}=E_{\cM_{\infty}}Q_{n+1}$ and by Lemma \ref{lemma:generating},
$\cM_n=\cM^{\alpha_{n+1}}\cap \cM_{\infty}$. By (iv) of Proposition \ref{proposition:cs}, 
we get a commuting square.
\end{proof}
\begin{Proposition}\label{proposition:generating-property}
If the representation $\rho:F^{+}\to \End(\cM,\psi)$ has the generating property then the
following equality holds for all $n \in \Nset_0$:
\[
\cM_n = \cM^{\alpha_{n+1}}.
\]
In other words, one has the tower of fixed point algebras
\[
\cM^{\rho(F^+)} \subset \cM^{\rho(g_{0})} \subset \cM^{\rho(g_{1})} \subset \cM^{\rho(g_{2})}
\subset \ldots \subset \cM =  \bigvee_{n \ge 0} \cM^{\rho(g_n)}. 
\]
\end{Proposition}
\begin{proof}
If the representation $\rho$ is generating, then $\cM_\infty=\cM$. Hence
$\cM_n=\cM^{\alpha_{n+1}}$ for all $n\in \Nset_0$ as a consequence of Lemma
\ref{lemma:generating}.
\end{proof}
\begin{Remark} \normalfont
Suppose that the representation  $\rho:F^{+}\to \End(\cM,\psi)$ satisfies the additional
relations $\rho(g_n) \rho(g_n) = \rho(g_{n+1}) \rho(g_n)$ for all $n \in \Nset_0$, as it is the
case for representations of the partial shifts monoid $S^+$. Then the inclusions 
$\cM^{\rho(g_{n})} \subset \cM^{\rho(g_{n+1})}$, and consequently $\cM_n = \cM^{\rho(g_{n+1})}$,
are immediate without stipulating the generating property of the representation, since $x = 
\rho(g_n)(x)$ implies $x = \rho(g_n^2)(x) = \rho(g_{n+1})\rho(g_n)(x) =  \rho(g_{n+1})(x)$ for
all $x \in \cM$ and $n\in \Nset_0$.     
\end{Remark}
\subsection{Commuting squares and Markovianity for shifted fixed point algebras}
\label{subsection:s-f-p-a}
The following intertwining properties will be crucial for obtaining local Markov filtrations from
representations of the Thompson monoid $F^+$.
\begin{Proposition}\label{proposition:intertwining-property}
Suppose $\rho \colon F^+ \to \End(\cM,\psi)$ is a  (not necessarily generating) representation 
of $F^+$.   Then with $\alpha_n=\rho(g_n)$, the following equality holds:
\[
\alpha_k Q_n  = Q_{n+1} \alpha_k  
\]
for all $0 \le k < n < \infty$. Here $Q_n$ denotes the $\psi$-preserving normal conditional
expectation from $\cM$ onto the fixed point algebra $\cM^{\alpha_n}$  of the represented 
generator $\alpha_n \in \End(\cM,\psi)$. 
\end{Proposition}
\begin{proof}
An application of the mean ergodic theorem and the relations between the generators of the
Thompson monoid $F^+$ yield that, for $k < n$,
\[
\alpha_k Q_n 
= \lim_{N \to \infty} \frac{1}{N} \sum_{i=1}^{N-1} \alpha_k \alpha_n^i 
=  \lim_{N \to \infty} \frac{1}{N} \sum_{i=0}^{N-1} \alpha_{n+1}^i \alpha_k 
= Q_{n+1} \alpha_k. 
\]
Here the limits are taken in the pointwise strong operator topology. 
\end{proof}
\begin{Theorem}\label{theorem:cs-markov-1}
Suppose $\rho \colon F^+ \to \End(\cM,\psi)$ is a generating representation with $\alpha_k :=
\rho(g_k)$ for all $k \in \Nset_0$. Then  each cell in the following  triangular tower is a
commuting square:
\begin{eqnarray*}
\setcounter{MaxMatrixCols}{20}
\begin{matrix}
\cM_{0}  &\subset&   \cM_{1} & \subset & \cM_{2} & \subset &\cM_{3} &  \subset  & \cdots & \subset & \cM_{\infty} = \cM\\
        &&          \cup  &         & \cup  &         & \cup &            &        & & \cup  \\
        &&   \alpha_0\big(\cM_{0}\big)&\subset&\alpha_0 \big(\cM_{1}\big)&\subset&\alpha_0\big(\cM_{2}\big)&\subset& \cdots & \subset & \alpha_0\big(\cM_{\infty}\big)\\
         &&               &         & \cup  &         & \cup   &               & & & \cup \\
              &&&&     \alpha_0^2\big(\cM_{0}\big)& \subset &  \alpha_0^2 \big(\cM_{1}\big) & \subset  & \cdots & \subset &   \alpha_0^2\big(\cM_{\infty}\big)\\
       &&&&&&   \cup           &         \       &&  & \cup \\
        &&&&&&   \vdots           &     \vdots    &   &   \vdots    &   \vdots 
\end{matrix}
\setcounter{MaxMatrixCols}{10}
\end{eqnarray*}
In particular, $\cM_{n+1} \cap \alpha_0\big(\cM_{n+1}\big)  = \alpha_0\big(\cM_n\big)$  
for all $n \ge 0$.  
\end{Theorem}
\begin{proof}
Let $0 \le m < n < \infty$ and  $k\ge 1$.  We verify first all inclusions as they appear in the
diagram
\begin{align} \label{eq:cs-markov-1}
\begin{matrix}
\cM_{m+k} &\subset &\cM_{n+k} \\
\cup  &        & \cup \\
\alpha_0^k(\cM_m)   & \subset  & \alpha_0^k(\cM_n)
\end{matrix}.
\end{align}
Indeed, the definition of $\cM_n$ ensures the claimed horizontal inclusions in this diagram. 
The vertical inclusions in the diagram follow from the  intertwining properties 
$\alpha_0^kQ_{n+1} = Q_{n+1+k} \alpha_0^k$  (see Proposition
\ref{proposition:intertwining-property}). For $n= \infty$, all inclusions are easily concluded 
by routine approximation arguments.\\
We show next that the above diagram is a commuting square. Indeed, as $\rho$ is generating,
$\cM_\ell=\cM^{\alpha_{\ell+1}}$, for all $\ell \in \Nset_0$, and $E_{\cM_\ell}=Q_{\ell+1}$,
where $E_{\cM_\ell}$ is the conditional expectation onto $\cM_\ell$. Hence for any $x \in \cM_n$,
\[
E_{\cM_{m+k}} \alpha_0^k(x) = Q_{m+k+1} \alpha_0^k(x) 
           = \alpha_0^k  Q_{m +1} (x)  = \alpha_0^k E_{\cM_{m}} (x).
\]
This ensures that $E_{\cM_{m+k}} (\alpha_0^k(\cM_n)) = \alpha_0^k(\cM_m)$.  Thus the above 
inclusions form a commuting square by Proposition \ref{proposition:cs} and, in particular,  
it holds that $\alpha_0^k(\cM_m) = \cM_{m+k}\cap \alpha_0^k(\cM_n)$.

Finally, the commuting square properties of more general cells in the triangular tower of
inclusions are deduced from those in \eqref{eq:cs-markov-1}, since commuting square properties
are preserved when acting with the endomorphism $\alpha_0$  on all four corners of the diagram.  
\end{proof}
\begin{Corollary} \label{corollary:markov-1}
Suppose $\rho \colon F^+ \to \End(\cM,\psi)$ is a  generating representation with $\alpha_k :=
\rho(g_k)$ for all $k \in \Nset_0$. Let $0 \le m \le n < \infty $ be fixed.  Then each cell in
the following  triangular tower is a commuting square:
\begin{eqnarray*}
\setcounter{MaxMatrixCols}{20}
\begin{matrix}
\cM_{m}  &\subset&   \cM_{n+1} & \subset & \cM_{n+2}  & \subset & \cdots & \subset & \cM_{\infty} = \cM\\
        &&          \cup  &         & \cup  &       &  &       &  \cup  \\
        &&   \alpha_m\big(\cM_{m}\big)&\subset&\alpha_m \big(\cM_{n+1}\big)&   \subset& \cdots & \subset & \alpha_m\big(\cM_{\infty}\big)\\
         &&               &         & \cup  &        &  &       &  \cup \\
              &&&&     \alpha_m^2\big(\cM_{m}\big)& \subset  & \cdots & \subset &   \alpha_m^2\big(\cM_{\infty}\big)\\
       &&&&&&              &           &          \cup \\
        &&&&   & &             &     \vdots                & \vdots 
\end{matrix}
\setcounter{MaxMatrixCols}{10}
\end{eqnarray*}
In particular, $\cM_{n+1} \cap \alpha_m\big(\cM_{n+1}\big)  = \alpha_m\big(\cM_{m}\big)$ and   
$\cM_{n+k+1} \cap \alpha_m\big(\cM_{n+k+1}\big)  = \alpha_m\big(\cM_{n+k}\big)$  
for all $k \ge 1$.  
\end{Corollary}
\begin{proof}
Consider the representation $\rho_{m,n} := \rho \circ \sh_{m,n} \colon F^+ \to \End(\cM,\psi)$ 
where  $\sh_{m,n}$ denotes the $(m,n)$-partial shift as introduced in Definition
\ref{definition:mn-shift}.  We observe that $\rho_{m,n}(g_0) = \rho(g_m)$ and
$\rho_{m,n}(g_k) = \rho(g_{n+k})$ for all $k \ge 1$.  In particular this ensures that
$\rho_{m,n}$ inherits the generating property from the representation $\rho$. Thus Theorem 
\ref {theorem:cs-markov-1} applies to $\rho_{m+1,n+1}$  and all claimed properties are immediate
since $\cM_{m} = \cM^{\rho(g_{m+1})}= \cM^{\rho_{m+1,n+1}(g_0)}$ and 
$\cM_{n+k} = \cM^{\rho(g_{n+k+1})} = \cM^{\rho_{m+1,n+1}(g_{k})}$ for $k \ge 1$. 
\end{proof}
The triangular tower of $\alpha_0$-shifted fixed point algebras (as given in Theorem
\ref{theorem:cs-markov-1}) can also be addressed through a local filtration indexed by `intervals'.
This reveals that Markovianity (as introduced in Definition \ref{definition:markov-filtration})
corresponds to specific commuting squares in the triangular tower. 
\begin{Corollary} \label{corollary:markov-2}
Suppose $\rho \colon F^+ \to \End(\cM,\psi)$ is a  generating representation. The family of von
Neumann subalgebras $\cM_\bullet^\rho \equiv \{\cM_I^\rho\}_{I \in \cI(\Nset_0)}^{}$ of
$(\cM,\psi)$, with
\[
\cM^\rho_{[0,n]}:= \cM_n, \qquad \cM^{\rho}_{[m,m+n]} 
:= \rho(g_0^m)(\cM_n),  \qquad     \cM^{\rho}_{[m, \infty)} := \rho(g_0^m)(\cM_\infty),
\]
defines a local Markov filtration.    
\end{Corollary}
\begin{proof}
First we check the isotony property to verify that this family of subalgebras forms a local filtration.
Let $\alpha_n=\rho(g_n), n\in \Nset_0$, as before. Suppose $[m,m+n]\subset [k,k+\ell]$, 
we will show that $\cM^{\rho}_{[m,m+n]}\subset \cM^{\rho}_{[k,k+\ell]}$, that is,
$\alpha_0^m(\cM_{n})\subset \alpha_0^k(\cM_\ell)$. As $[m,m+n]\subset [k,k+\ell]$, we must have
$m\geq k$ and $n\leq\ell$. Hence for $x\in \cM_n$, we can write 
$\alpha_0^m(x)=\alpha_0^k\alpha_0^{m-k}(x)$, so it suffices to show that 
$\alpha_0^{m-k}(x)\in \cM_\ell$. Let 
$p\geq \ell+1\geq (m-k)+n+1$, then $\alpha_p\alpha_0^{m-k}(x) = \alpha_0^{m-k}\alpha_{p-(m-k)}(x)
= \alpha_0^{m-k}(x)$ as $p-(m-k)\geq n+1$.

Let $P^\rho_I$ denote the $\psi$-preserving normal conditional expectation from $\cM$ onto
$\cM^\rho_I$. This local filtration is Markovian if $P^\rho_{[0,m]}P^\rho_{[m,n]}=P^\rho_{[m,m]}$ for
$0\leq m\leq n$, which is implied by the definition of $\cM^\rho_I$ and the following cell of
inclusions that is a  commuting square as a consequence of Theorem \ref{theorem:cs-markov-1}:
\[
\begin{matrix}
\cM_{m} &\subset &\cM \\
\cup  &        & \cup \\
\alpha_0^m(\cM_0)   & \subset  & \alpha_0^m(\cM_n)
\end{matrix}.
\]
\end{proof}
\begin{Corollary} 
\label{corollary:markov-3}
Suppose $\rho \colon F^+ \to \End(\cM,\psi)$ is a  generating representation and consider the
$(m,n)$-shifted representation $\rho_{m,n} := \rho \circ \sh_{m,n}$ for some fixed 
$0 \le m \le n < \infty$. Then the family of von Neumann subalgebras 
$\cM_\bullet^{\rho_{m,n}} \equiv \{\cM_I^{\rho_{m,n}^{}}\}_{I \in \cI(\Nset_0)}^{}$ of
$(\cM,\psi)$, with
\[
\cM^{\rho_{m,n}^{}}_{[0,\ell]}:= \cM_{\ell+n}, \qquad \cM^{\rho_{m,n}^{}}_{[k,k+\ell]} 
                              := \rho(g_m^k)(\cM_{\ell+n}),  \qquad   
\cM^{\rho_{m,n}^{}}_{[k, \infty)} := \rho(g_m^k)(\cM_\infty),
\]
defines a local Markov filtration.    
\end{Corollary}
\begin{proof}
The case $m=n=0$ corresponds to Corollary \ref{corollary:markov-2}. Its proof directly transfers
to the general case $0 \le m \le n$, after relabeling the involved objects and morphisms
according to the $(m,n)$-shifted representation. 
\end{proof}
\subsection{Commuting squares and Markovianity for stationary processes.}
\label{subsection:Markov-F}
Given the  representation $\rho \colon F^+ \to \End(\cM,\psi)$, with represented generators
$\alpha_n := \rho(g_n)$, for $n \in \Nset_0$, and intersected fixed point algebras 
\[
\cM_n := \bigcap_{k \ge n+1} \cM^{\alpha_k},
\]
let $\cA_0 \subset \cM_0$ be a von Neumann subalgebra of $(\cM,\psi)$.  Then 
$(\cM,\psi, \alpha_0, \cA_0)$ is a (unilateral noncommutative) stationary process with 
generating algebra $\cA_0$. Its canonical local filtration is denoted by 
$\cA_\bullet \equiv \{\cA_{I}\}_{I\in \cI(\Nset_0)}$, where 
\[ 
\cA_I := \bigvee_{i \in I}\alpha_0^i(\cA_0),  
\]
and an `interval' $I \in \cI(\Nset_0)$ is written as 
$[m,n] := \{i  \in \Nset_0 \mid m \le i \le n\}$  or
$[m,\infty) := \{i \in \Nset_0 \mid m \le i \}$.

Furthermore $P_I$ will denote the $\psi$-preserving normal conditional expectation 
from $\cM$ onto $\cA_I$. Note that the endomorphism $\alpha_0$ acts covariantly on the local 
filtration, i.e.~ $\alpha_0(\cA_I) = \cA_{I+1}$ for all $I \in \cI(\Nset_0)$, where 
$I+1 := \{i+1\mid i \in I\}$. 

We record a simple, but important, observation obtained from the relations of $F^+$ on stationary
processes to which we will frequently appeal. 
\begin{Proposition}\label{proposition:fixed-point}
Let $(\cM,\psi,\alpha_0,\cA_0)$ be the (unilateral noncommutative) stationary process with
$\cA_0\subset \cM_0$ as above. Then it holds that $\cA_{[0,n]}\subset \cM_n$ for all 
$n \in \Nset_0$.
\end{Proposition}
\begin{proof}
As $\cA_0\subset \cM_0$, it holds that $\alpha_{n}(x) = x$ for any $x \in \cA_0$ and 
$n \in \Nset$. Thus using the defining relations of $F^+$ we get for $0\le k \le n$ and 
$n+1 \le  l$,
	\[
	\alpha_{l} \alpha_0^k(x) =  \alpha_0^k \alpha_{l-k}(x) = \alpha_0^k(x).
	\]
Hence $\cA_{[0,n]} \subset \cM_{n} $ for all $n\in \Nset_0$. 
\end{proof}
\begin{Remark} \label{remark:adaptedness} \normalfont
The canonical local filtration  $\cA_\bullet$ is coarser than the local filtration $\cM_\bullet^\rho 
\equiv \{\cM_{I}^\rho\}_{I \in \cI(\Nset_0)}^{}$ which is Markovian whenever $\rho$ is 
generating by Corollary \ref{corollary:markov-2}. Indeed this follows from the endomorphism
$\alpha_0$ acting covariantly on the local filtration. More explicitly, we conclude that
$\cA_{[m,n]}=\alpha_0^m(\cA_{[0,n-m]}) \subset \alpha_0^m(\cM_{n-m}) = \cM_{[m,n]}^\rho$ 
for $0 \le m \le n$. Furthermore it holds that
\[
\cA_{[m,\infty)} = \bigvee_{n\ge m}\cA_{[m,n]} \subset \bigvee_{n\ge m}\cM^{\rho}_{[m,n]} 
                 = \cM^{\rho}_{[m,\infty)}.
\] 
\end{Remark}
We next observe that the generating property of the representation $\rho$ can be concluded 
from the minimality of a stationary process. 
\begin{Proposition}\label{proposition:minimality-generating}
Suppose the representation $\rho:F^+\to \End(\cM,\psi)$ and $\cA_0\subset \cM_0$ are given. If
the stationary process $(\cM,\psi,\alpha_0,\cA_0)$ is minimal, then $\rho$ is generating.
\end{Proposition}
\begin{proof}
The minimality of the stationary process
$(\cM,\psi,\alpha_0,\cA_0)$ ensures $\cA_{[0,\infty]}
=\bigvee_{i \in \Nset_0}\alpha_0^i(\cA_0)=\cM$.
By Proposition \ref{proposition:fixed-point}, $\cA_{[0,n]}\subset \cM_n$ for all $n \in \Nset_0$.
Thus $\cM = \bigvee_{n \ge 0}\cA_{[0,n]}\subset \bigvee_{n \ge 0} \cM_n = \cM_\infty$. We
conclude from this that the representation $\rho$ has the generating property, 
i.e.~$\cM_\infty = \cM$.
\end{proof}
In the following results, it is not assumed that the stationary process is minimal or that the
representation $\rho$ is generating unless explicitly mentioned.
\begin{Theorem}\label{theorem:markov-filtration-1}
Suppose $\rho \colon F^+ \to \End(\cM,\psi)$ is a representation. Let $\alpha_n:=\rho(g_n)$ as
before, and let $\cA_0 \subset \cM_{0}$ and $\cA_{[0,\infty)} 
:= \bigvee_{n \in \Nset_0}\alpha_0^n(\cA_0) $ be von Neumann subalgebras of $(\cM,\psi)$ such
that the inclusions
\[
\begin{matrix}
\cM^{\alpha_1} &\subset &\cM \\
\cup  &        & \cup \\
\cA_0    & \subset  & \cA_{[0,\infty)}
\end{matrix}
\]
form a commuting square. Then the family of von Neumann subalgebras 
$\cA_\bullet \equiv \{\cA_{I}^{}\}_{I\in \cI(\Nset_0)}^{}$, with  
\[ 
\cA_I := \bigvee_{i \in I}\alpha_0^i(\cA_0) 
\]
is a local Markov filtration and $\big(\cM,\psi, \alpha_0, \cA_0  \big)$ is a stationary Markov
process.  
\end{Theorem}
\begin{proof}
Let $Q_n$ and $P_I$ denote the $\psi$-preserving normal conditional expectations from $\cM$ onto
$\calM^{\alpha_n}$ and $\cA_I$ respectively as before. Note that the commuting square condition
implies $Q_1 P_{[0,\infty)} = P_{[0,0]}$. From Proposition \ref{proposition:fixed-point},
$\cA_{[0,n]} \subset \cM_n \subset \cM^{\alpha_{n+1}}$ for all $n\in \Nset_0$. Hence we get 
\begin{alignat*}{3}
P_{[0,n]} \alpha_0^n P_{[0,\infty)} 
&=  P_{[0,n]} Q_{n+1} \alpha_0^n P_{[0,\infty)}   
    &\qquad & \text{(since $\cA_{[0,n]} \subset \cM^{\alpha_{n+1}}$)}\\ 
&=  P_{[0,n]} \alpha_0^n Q_1 P_{[0,\infty)}   &\qquad &\text{(by intertwining property)}\\
&=  P_{[0,n]} \alpha_0^n P_{[0,0]} P_{[0,\infty)}   
    & \qquad & \text{(by commuting square condition)}\\
&=  \alpha_0^n P_{[0,0]} P_{[0,\infty)}   
    & \qquad & \text{(as $\cA_{[n,n]} \subset \cA_{[0,n]}$)}\\
&=  P_{[n,n]}\alpha_0^n P_{[0,0]} P_{[0,\infty)}   
    & \qquad & \text{(since $\cA_{[n,n]}= \alpha_0^n(\cA_0)$)}\\
&= P_{[n,n]} \alpha_0^n Q_1 P_{[0,\infty)}   
    &\qquad & \text{(by commuting square condition)}\\
&=  P_{[n,n]} Q_{n+1} \alpha_0^n P_{[0,\infty)}   
    &\qquad &\text{(by intertwining property)}\\
&= P_{[n,n]} \alpha_0^n P_{[0,\infty)} 
    &\qquad & \text{(since $\cA_{[n,n]} \subset \cM^{\alpha_{n+1}}$)}.
\end{alignat*}
Altogether we have shown that $P_{[0,n]}P_{[n,\infty)}=P_{[n,n]}$, which is the required
Markovianity for the local filtration $\{\cA_{I}\}_{I\in  \cI(\Nset_0)}$.
\end{proof}
\begin{Corollary}\label{corollary-markov-filtration-0}
Suppose $\rho\colon F^+ \to \End(\cM,\psi)$ is a representation with $\alpha_0 = \rho(g_0)$.  
Then the quadruple $\big(\cM,\psi, \alpha_0, \cM_{0}\big)$ is a stationary Markov process. 
\end{Corollary}
\begin{proof}
We know from Corollary \ref{corollary:commsquare} that the following is a commuting square:
\[
\begin{matrix}
\cM^{\alpha_1} &\subset &\cM \\
\cup  &        & \cup \\
\cM_0    & \subset  & \cM_{\infty}
\end{matrix}.
\]
Let $\{\cM_I\}_{I\in \cI(\Nset_0)}$ denote the local filtration given by  $\cM_I=\bigvee_{i\in I}
\alpha_0^{i}(\cM_0)$ and $P_I$ be the corresponding conditional expectations. As
$\cM_{[0,n]}\subset \cM_{n}$ for all $n\in \Nset_0$, it is easily verified that
$\cM_{[0,\infty)}\subset \cM_\infty$. Let $P_0:=P_{[0,0]}$ be the $\psi$-preserving conditional
expectation from $\cM$ onto $\cM_0$. Then from the commuting square above, we have 
$E_{\cM_{\infty}}Q_1=P_{0}$, where $E_{\cM_\infty}$ is of course the conditional expectation 
onto $\cM_\infty$. This in turn gives $P_{[0,\infty)}Q_1=P_{[0,\infty)} E_{\cM_{\infty}}Q_1 =
P_{[0,\infty)} P_0=P_0$. Hence we get that $\cM_0$ is a von Neumann subalgebra of $\cM$ such that
\[
	\begin{matrix}
	\cM^{\alpha_1} &\subset &\cM \\
	\cup  &        & \cup \\
	\cM_0    & \subset  & \cM_{[0,\infty)}
	\end{matrix}
\]
forms a commuting square. By Theorem \ref{theorem:markov-filtration-1}, 
$\big(\cM,\psi, \alpha_0,\cM_{0}\big)$ is a stationary Markov process.
\end{proof}
\begin{Corollary}\label{corollary:markov-filtration-0-MN}
Suppose $\rho\colon F^+ \to \End(\cM,\psi)$ is a representation with $\alpha_m = \rho(g_m)$, 
for $m\in \Nset_0$. Then the quadruple $\big(\cM,\psi, \alpha_m, \cM_{n}\big)$ is a stationary
Markov process  for any $0 \le m \le n < \infty$. 
\end{Corollary}
\begin{proof}
Consider the representation $\rho_{m,n} := \rho \circ \sh_{m,n} \colon F^+ \to \End(\cM,\psi)$ 
where  $\sh_{m,n}$ denotes the $(m,n)$-partial shift as introduced in Definition
\ref{definition:mn-shift}.  We observe that $\rho_{m,n}(g_0) = \rho(g_m)$ and
$\rho_{m,n}(g_k) = \rho(g_{n+k})$ for all $k \ge 1$. In particular we get 
$\bigcap_{k \ge 1}\cM^{\rho_{m,n}(g_k)}
=\bigcap_{k \ge 1} \cM^{\rho(g_{k+n})}
=\bigcap_{k \ge n+1} \cM^{\rho(g_k)}=\cM_n.$ 
Thus Corollary \ref{corollary-markov-filtration-0} applies for the $(m,n)$-shifted representation
$\rho_{m,n}$ and its application completes the proof.  
\end{proof}
\begin{Corollary}\label{corollary:markov-filtration-MN}
Suppose  $\rho \colon F^+ \to \End(\cM,\psi)$ is a generating representation. Then the quadruple 
$\big(\cM,\psi, \alpha_m, \cM^{\alpha_{n+1}}\big)$ is a stationary Markov process for any 
$0 \le m \le n < \infty$. 
\end{Corollary}
\begin{proof}
If the representation $\rho$ is generating, then $\cM^{\alpha_{n+1}}=\cM_n$. Hence the result follows by Corollary \ref{corollary:markov-filtration-0-MN}.
\end{proof}
\begin{Remark}\normalfont \label{remark:markov-filtration-1}
The commuting square assumption in Theorem \ref{theorem:markov-filtration-1} may not be 
satisfied for a noncommutative stationary process $(\cM,\psi, \alpha_0, \cA_0)$ if one only
demands that the generator $\cA_0$ is a $\psi$-conditioned von Neumann subalgebra of the fixed
point algebra $\cM^{\alpha_1}$. Consequently the canonical local filtration of the resulting
noncommutative stationary processes may not be Markovian.
\end{Remark}
\begin{Theorem}\label{theorem:markov-filtration-2}
Let the probability space $(\cM,\psi)$ be equipped with the representation 
$\rho \colon F^+ \to \End(\cM,\psi)$ and the local filtration 
$\cA_\bullet \equiv \{\cA_I\}_{I \in \cI(\Nset_0)}$, where 
$\cA_I := \bigvee_{i \in I}\rho(g_0^i)(\cA_0)$ for some von Neumann subalgebra $\cA_0$ of
$\cM_0$.  Further suppose the inclusions 
\[
\begin{matrix}
\cM^{\rho(g_{m+1})} &\subset &\cM \\
\cup  &        & \cup \\
\cA_{[0,m]}   & \subset  & \cA_{[0,\infty)}
\end{matrix}
\]
form a commuting square for all $m \ge 0$.  Then each cell in the following  triangular tower of
inclusions is a commuting square:
\begin{eqnarray*}
\setcounter{MaxMatrixCols}{20}
\begin{matrix}
\cA_{[0,0]}  &\subset&   \cA_{[0,1]} & \subset & \cA_{[0,2]} & \subset &\cA_{[0,3]} & \subset &\cA_{[0,4]} & \subset  & \cdots & \subset & \cA_{[0,\infty)}\\
        &&          \cup  &         & \cup  &         & \cup &         & \cup  &       & & & \cup  \\
        &&   \cA_{[1,1]}&\subset&\cA_{[1,2]}&\subset&\cA_{[1,3]}&\subset&\cA_{[1,4]}&\subset& \cdots & \subset & \cA_{[1,\infty)}\\
         &&               &         & \cup  &         & \cup   &         & \cup  &       & & & \cup \\
              &&&&   \cA_{[2,2]}& \subset &  \cA_{[2,3]}& \subset &   \cA_{[2,4]}
              & \subset  & \cdots & \subset &   \cA_{[2,\infty)}\\
       &&&&&&   \cup           &         & \cup  &       &&  & \cup \\
        &&&&&&   \vdots           &         & \vdots  &       &&  & \vdots 
\end{matrix}
\setcounter{MaxMatrixCols}{10}
\end{eqnarray*}
In particular,  $\cA_\bullet$ is a local Markov filtration.
\end{Theorem}
\begin{proof}
All claimed inclusions in the triangular tower are clear from the definition of $\cA_{[m,n]}$. 
We recall from Proposition \ref{proposition:fixed-point} that
$\alpha_0^{k}(\cA_0)\subset \cM^{\alpha_{n+1}}$ for $0\leq k \leq n$.
Hence $\cA_{[m,n]} \subset \cM^{\alpha_{n+1}}$ for all $0 \le m \le n$.  Next we show that, for
$0 \le k$ and $1 \le m$,  the cell of inclusions
\[
\begin{matrix}
\alpha_0^k(\cA_{[0,m]}) &\subset &\alpha_0^k(\cA_{[0,m+1]}) \\
\cup  &        & \cup \\
\alpha_0^{k+1}(\cA_{[0,m-1]})   & \subset  & \alpha_0^{k+1} (\cA_{[0,m]})
\end{matrix}
\]
forms a commuting square. So, as $P_I$ denotes the normal $\psi$-preserving conditional
expectation from $\cM$ onto $\cA_I$, we need to show 
\[
P_{[k,m+k]} P_{[k+1,m+k+1]} = P_{[k+1,m+k]} 
\]
or, equivalently,
\[
P_{[k,m+k]} \alpha_0^{k+1} P_{[0,m]} =  \alpha_0^{k+1} P_{[0,m-1]}. 
\]
Indeed, we calculate 
\begin{align*}
P_{[k,m+k]} \alpha_0^{k+1}  P_{[0,m]} 
& = P_{[k,m+k]} Q_{m+k+1}\alpha_0^{k+1}  P_{[0,m]} \\
& = P_{[k,m+k]} \alpha_0^{k+1} Q_{m}  P_{[0,m]}\\
& = P_{[k,m+k]} \alpha_0^{k+1} Q_{m}  P_{[0,\infty)}P_{[0,m]} \\
& = P_{[k,m+k]} \alpha_0^{k+1}  P_{[0,m-1]} P_{[0,m]} \\
& = P_{[k,m+k]} \alpha_0^{k+1}  P_{[0,m-1]} \\
&=\alpha_0^{k+1}  P_{[0,m-1]}. 
\end{align*}
Here we have used that $ P_{[k,m+k]} =  P_{[k,m+k]} Q_{m+k+1}$, the intertwining properties of
$\alpha_0$ and the commuting square assumption $  Q_{m}  P_{[0,\infty)}=  P_{[0,m-1]}$.

Since $\alpha_0^k(\cA_{[m,n]}) = \cA_{[m+k,n+k]}$ is evident from the definition of the local
filtration, we have verified that each cell of inclusions in the triangular tower forms a
commuting square.
\end{proof}
More generally, we may consider a probability space which is equipped both with a local filtration and
a representation of the Thompson monoid, and formulate compatiblity conditions between the local
filtration and the representation such that one obtains rich commuting square structures. 
\begin{Corollary}
Suppose the probability space $(\cM,\psi)$ is equipped with a local filtration 
$\cN_\bullet \equiv \{\cN_{I}^{}\}_{I \in \cI(\Nset_0)}^{}$
and a representation $\rho \colon F^+ \to \End(\cM,\psi)$  such that
\begin{enumerate}
\item \label{item:markov1} 
$\rho(g_0) \big(\cN_I\big)=  \cN_{I+1}$ for all $I \in \cI(\Nset_0)$ (compatibility),   
\item \label{item:markov2} 
$\cN_{[0,m]} \subset \cM^{\rho(g_{m+1})}$ for all $m \in \Nset_0$ (adaptedness),
\item \label{item:markov3} 
the inclusions
\[
\begin{matrix}
\cM^{\rho(g_{m+1})} &\subset &\cM \\
\cup  &        & \cup \\
\cN_{[0,m]}   & \subset  & \cN_{[0,\infty)}
\end{matrix}
\]
form a commuting square for all $m \in \Nset_0$.
\end{enumerate}
Then each cell in the following triangular tower of inclusions is a commuting square:
\begin{eqnarray*}
\setcounter{MaxMatrixCols}{20}
\begin{matrix}
\cN_{[0,0]}  &\subset&   \cN_{[0,1]} & \subset & \cN_{[0,2]} &  \subset  & \cdots & \subset & \cN_{[0,\infty)}\\
        &&          \cup  &         & \cup  &                  & & & \cup  \\
        &&   \rho(g_0)\big(\cN_{[0,0]}\big)&\subset&\rho(g_0)\big(\cN_{[0,1]}\big)&\subset& \cdots & \subset & \rho(g_0)\big(\cN_{[0,\infty)}\big)\\
         &&               &         & \cup  &                       & & & \cup \\
              &&&&    \rho(g_0^2)\big(\cN_{[0,0]}\big)& \subset  & \cdots & \subset &  \rho(g_0^2)\big(\cN_{[0,\infty)}\big)\\
       &&&&&&              &             &   \cup \\
        &&&&&                    &   &  \vdots  &   \vdots 
\end{matrix}
\setcounter{MaxMatrixCols}{10}
\end{eqnarray*}
In particular,  $\cN_\bullet$ is a local Markov filtration.
\end{Corollary}
\begin{proof}
Let $P_I$ be the normal $\psi$-preserving conditional expectation onto $\cN_I$. Let
$\alpha_n=\rho(g_n)$ and $Q_n$ be the normal $\psi$-preserving conditional expectation onto
$\cM^{\alpha_{n}}$ as before. We observe that $\cN=\cN_{[0,0]}\subset \cM^{\alpha_1}$  by the
given adaptedness. Adaptedness also gives us $\cN_{[m,n]}\subset\cN_{[0,n]}\subset
\cM^{\alpha_{n+1}}$ for $ 0\leq m\leq n$. Thus $P_{[k,m+k]}=P_{[k,m+k]}Q_{m+k+1}$ as before. 
The rest of the proof follows just as in Theorem \ref{theorem:markov-filtration-2}.
\end{proof}
\subsection{A noncommutative version of de Finetti's theorem}
\label{subsection:ncdf}
Most results of the previous two subsections can be reformulated in terms of sequences of 
random variables associated to stationary processes (see Definition \ref{definition:process-sequence}). 
\begin{Proposition}\label{proposition:definetti-markov}
Given the representation $\rho \colon F^+ \to \End(\cM,\psi)$, let $\cA_0$ be some fixed
$\psi$-conditioned von Neumann subalgebra of $\cM_0 = \bigcap_{k\geq 1} \cM^{\rho(g_k)}$ 
and $\varphi_0 := \psi|_{\cA_0}$.  Then the sequence of random variables 
\[
(\iota_n)_{n \ge 0}\colon  (\cA_0, \varphi_0) \to (\cM,\psi), \qquad \iota_{n} 
                              := \rho(g_0)^n|_{\cA_0}
\]
(associated to the stationary process $(\cM, \psi, \rho(g_0), \cA_0)$) is partially spreadable.
Furthermore this stationary process and its associated sequence of random variables have the same
canonical local filtration 
\[
\cA_\bullet \equiv \big\{\cA_I := \bigvee_{i \in I} \rho(g_0^i)(\cA_0)\big\}_{I \in \cI(\Nset_0)}
\]
which is coarser than the local Markov filtration
\[\cM_\bullet \equiv \big\{\cM_I := \bigvee_{i \in I} \rho(g_0^i)(\cM_0)\big\}_{I \in \cI(\Nset_0)}.
\]  
\end{Proposition}
\begin{proof}
This is immediate from Definition \ref{definition:p-s}, where we introduced partial spreadability as a
distributional symmetry. Clearly the canonical filtration of the stationary process and its
associated sequence of random variables coincide. The inclusion $\cA_0 \subset \cM_0$ ensures
that $\cA_\bullet$ is coarser than $\cM_\bullet$. The Markovianity of $\cM_\bullet$ is inferred
from Corollary \ref{corollary-markov-filtration-0}. 
\end{proof}
We are ready for the proof of a noncommutative version of de Finetti's theorem, as formulated 
in Theorem \ref{theorem:extended-de-Finetti}, and repeat its formulation for the convenience 
of the reader.
\begin{Theorem*}
Let $\iota \equiv (\iota_n)_{n\ge 0} \colon (\cA,\varphi) \to (\cM,\psi)$ be a sequence of
(identically distributed) random variables and consider the following conditions:
\begin{enumerate}
\item[(a)] 
$\iota$ is partially spreadable;  
\item[(b)] 
$\iota$ is stationary and adapted to a local Markov filtration;
\item[(c)] 
$\iota$ is identically distributed and adapted to a local Markov filtration.
\end{enumerate}
Then one has the following implications:
\[
\text{(a)}   \Longrightarrow \text{(b)}   \Longrightarrow \text{(c)}.   
\]
\end{Theorem*}
\begin{proof}[Proof (of Theorem \ref{theorem:extended-de-Finetti})]
(a) $\Longrightarrow$ (b): The stationarity of $\iota$ follows from $\psi \circ \rho(g_0) = \psi$
and $\iota_n=\rho(g_0)^n\iota_0, n\geq 0$. Proposition \ref{proposition:definetti-markov} ensures that
the canonical local filtration $\cA_\bullet$ is coarser than the local Markov filtration $\cM_\bullet$.
Thus the sequence $\iota$ is adapted to the local Markov filtration $\cM_\bullet$, according to
Definition \ref{definition:markovianity}.\\
(b) $\Longrightarrow$ (c): Stationary sequences are identically distributed. 
\end{proof}
Let $(\cM,\psi)$ be a noncommutative probability space, $\rho:F^+\to \End(\cM,\psi)$ be a
representation and $\cA_0$ be a $\psi$-conditioned subalgebra of 
$\cM_0=\bigcap_{k\geq 1}\cM^{\rho(g_k)}$. By Corollary \ref{corollary-markov-filtration-0}, 
we know that if $\cA_0=\cM_0$, then the canonical local filtration 
$\cA_\bullet=\{\cA_I\}_{I\in \cI(\Nset_o)}$ is Markovian. The definition of a partially spreadable
sequence $\{\iota_n\}\colon (\cA,\varphi)\to (\cM,\psi)$ entails the existence of a
representation $\rho \colon F^+\to \End(\cM,\psi)$ such that $\iota_0(\cA)\subset \cM_0=\bigcap_{k\geq 1}\cM^{\rho(g_k})$ (see Definition \ref{definition:p-s}). This motivates to
strengthen the property of partial spreadability as follows. 
\begin{Definition}\label{definition:max-p-s} \normalfont
A partially spreadable sequence of random variables 
$\iota \equiv (\iota_{n})_{n\geq 0}: (\cA,\varphi)\to (\cM,\psi)$ 
is said to be \emph{maximal partially spreadable} if
$\iota_0(\cA)=\cM_0=\cap_{k\geq 1}\cM^{\rho(g_k})$.
Here $\rho$ denotes a representation of $F^+$ as in Definition \ref{definition:p-s}, and $\cM_0$ 
the intersection of the fixed point algebras $\cM^{\rho(g_k)}$ for $k \ge 1$.  
\end{Definition}
This refined notion of partial spreadability allows us to tightly connect the representation
theory of the Thompson monoid $F^+$ and Markovianity. 
\begin{Theorem}\label{theorem:maxps} 
A maximal partially spreadable sequence of random variables $\iota \equiv (\iota_n)_{n\ge 0}
\colon (\cA,\varphi) \to (\cM,\psi)$ is stationary and Markovian.
\end{Theorem}
Thus we have obtained a noncommutative de Finetti theorem for noncommutative 
stationary Markov sequences. One should not expect the converse to be true
in the full generality of our operator algebraic framework, for similar reasons 
as outlined for the noncommutative extended de Finetti theorem in \cite{Ko10}. 
\section{\texorpdfstring{Constructions of representations of the Thompson monoid $F^+$}{}} 
\label{section:Constructions of Reps of F+}
This section is about how to construct  representations of the Thompson monoid $F^+$ as they
naturally arise in noncommutative probability theory. It will be seen that such constructions 
are intimately related to the construction of stationary Markov processes. In particular, this
will establish that a large class of stationary Markov sequences is partially spreadable.  
\subsection{Tensor product constructions}
\label{subsection:tensor-product}
Let $(\cA,\varphi)$ and $(\cC,\chi)$ be probability spaces.  Taking the infinite von Neumann
algebraic tensor product with respect to an infinite tensor product state,    
\[
(\cM, \psi) 
:=  \big(\cA \otimes \cC^{\otimes_{\Nset_0}} , \varphi \otimes \chi^{\otimes_{\Nset_0}}\big)
\]
is a probability space which can be equipped with a representation of the monoid of partial
shifts $S^+$ and the Thompson monoid $F^+$. For $n \in \Nset_0$, let $\beta_n$ denote the partial
shift which acts on the weak*-total set of finite elementary tensors in $\cM$ as 
\[
\beta_n(a \otimes x_0\otimes\cdots\otimes x_{n-1}\otimes x_{n} \otimes x_{n+1} \otimes \cdots)
:= a \otimes x_0 \otimes\cdots\otimes x_{n-1} \otimes \1_\cC \otimes x_{n} \otimes x_{n+1}
    \otimes \cdots.
 \] 
\begin{Proposition}\label{proposition:rep-S+} 
The maps $h_n \mapsto \beta_n =: \varrho(h_n) $, with $n \in \Nset_0$, extend multiplicatively 
to a representation $\varrho\colon S^+ \to \End(\cM,\psi)$ which has the generating property. 
\end{Proposition}
\begin{proof}
Each $\beta_n$ extends to a unital injective *-homomorphism on $\cM$, denoted by the same symbol,
such that $\psi \circ \beta_n = \psi$. As the modular automorphism group of $(\cM,\psi)$ equals
the von Neumann algebraic tensor product of the modular automorphism groups of its tensor
factors, i.e.~$\sigma^\psi_t =  \sigma^\varphi_t \otimes (\sigma^\chi_t)^{\otimes_{\Nset_0}}$, 
it is easily verified that $\beta_n  \sigma^\psi_t =  \sigma^\psi_t  \beta_n$ for all 
$t \in \Rset$.  Thus $\beta_n \in \End(\cM,\psi)$ for all $n \in \Nset_0$.  For $0 \le 
k \le \ell < \infty$, the relations $\beta_k \beta_\ell = \beta_{\ell +1} \beta_k$ are
directly checked on elementary tensors. Consequently the endomorphisms  
$\beta_0, \beta_1, \ldots $ satisfy the relations of the monoid generators 
$h_0, h_1, \ldots \in S^+$.  Finally, the generating property of the representation $\varrho$ 
is inferred from $\cA \otimes \cC^{\otimes_{n}} \otimes \1_{\cC}^{\otimes_{\Nset_0}} \subset
\cM^{\beta_n}$ which ensures that the unital *-algebra $\bigcup_{n \in \Nset_0} \cM^{\beta_n}$ 
is weak*-dense in $\cM$.    
\end{proof}
Let $\epsilon \colon F^+ \to S^+$ be the monoid epimorphism with $\epsilon(g_n) = h_n$ 
for all $n \in \Nset$.  Then
\[
F^+ \ni g  \mapsto  \varrho \circ \epsilon (g) \in \End(\cM,\psi) 
\]
defines a representation of the Thompson monoid $F^+$ which also has the generating property. 
More general representations of $F^+$ can be constructed as follows.  

Given the two random variables $C\colon(\cA, \varphi) \to (\cA\otimes\cC, \varphi \otimes \chi)$ 
and $D \colon  (\cC, \chi) \to  (\cC \otimes \cC, \chi \otimes \chi)$, let $\alpha_n$ denote the
$\Cset$-linear extension of the map defined on a weak*-total subset of $\cM$ by
\begin{align} \label{eq:operator-D}
\alpha_n (a \otimes x_0 \otimes x_1 \otimes \cdots) &:= 
\begin{cases}
C(a) \otimes x_0 \otimes x_1 \otimes  \cdots  & \text{ if $n=0$},\\
a \otimes D(x_0) \otimes x_1 \otimes  \cdots  & \text{ if $n=1$},\\
a  \otimes x_0 \otimes \cdots \otimes D(x_{n-1}) \otimes \cdots  & \text{ if $n>1$.}
\end{cases}
\end{align}
\begin{Proposition}\label{proposition:rep-F+} 
The maps $g_n \mapsto \alpha_n =: \rho(g_n) $, with $n \in \Nset_0$, extend multiplicatively 
to a representation $\rho\colon F^+ \to \End(\cM,\psi)$ which has the generating property. 
\end{Proposition}
\begin{proof}
For $0 \le k < \ell < \infty$, the relations $\alpha_k \alpha_\ell =  \alpha_{\ell+1}  \alpha_k$ 
are verified in a straightforward computation on finite elementary tensors of the form 
$x = a \otimes x_0 \otimes  \cdots \otimes x_n \otimes \1_\cC^{\otimes_\Nset}$. 
We need to verify that each $\alpha_n$ commutes with the modular automorphism group $\sigma^{\psi}
= \sigma^\varphi \otimes (\sigma^\chi)^{\otimes_{\Nset_0}}$. As $\alpha_n$ is a *-homomorphism, 
this modular condition is ensured if
\begin{align}
\alpha_n \circ \sigma_t^\psi (a \otimes \1_{\cC}^{\otimes_{\Nset_0}}) 
&=  \sigma_t^\psi \circ \alpha_n  (a \otimes \1_{\cC}^{\otimes_{\Nset_0}}),
\label{eq-prop-mod-cond-i}
\\
\alpha_n \circ \sigma_t^\psi (\1_\cA \otimes y) 
&=  \sigma_t^\psi \circ \alpha_n  (\1_\cA \otimes y)
\label{eq-prop-mod-cond-ii}
\end{align}
for all $a \in \cA$ and $y \in \cC^{\otimes_{\Nset_0}}$. Let us first consider the case $n=0$ in 
\eqref{eq-prop-mod-cond-i}. Indeed, as the noncommutative random variable $C$ intertwines the 
modular automorphism groups $\sigma^\varphi$ and $\sigma^\varphi \otimes \sigma^\chi$, 
\begin{align*}
\alpha_0 \circ \sigma_t^\psi (a \otimes \1_{\cC}^{\otimes_{\Nset_0}}) 
&=  \alpha_0 (\sigma_t^\varphi( a) \otimes \1_{\cC}^{\otimes_{\Nset_0}}) 
= (C \circ \sigma_t^\varphi( a)) \otimes \1_{\cC}^{\otimes_{\Nset}} \\ 
&=  ((\sigma_t^\varphi \otimes  \sigma_t^\chi) \circ C)(a) \otimes \1_{\cC}^{\otimes_{\Nset}}
 =  \sigma_t^\psi (C(a) \otimes \1_{\cC}^{\otimes_{\Nset}})\\
&=\sigma_t^\psi \circ \alpha_0  (a \otimes \1_{\cC}^{\otimes_{\Nset_0}}).
\end{align*}
Furthermore, again in the case $n=0$, \eqref{eq-prop-mod-cond-ii} is immediate from 
\begin{align*}
\alpha_0 \circ \sigma_t^\psi (\1_\cA \otimes y) = 
\beta_0 \circ \sigma_t^\psi (\1_\cA \otimes y)
=  \sigma_t^\psi  \circ \beta_0  (\1_\cA \otimes y),
\end{align*}
since $\beta_0$ and $\sigma_t^\psi$ commute by Proposition \ref{proposition:rep-S+}. Similar arguments
ensure $\alpha_n \circ \sigma_t^\psi =  \sigma_t^\psi \circ \alpha_n $ for $n \ge 1$.  
Finally, the generating property of $\rho$ is inferred from  
$\cA\otimes \cC^{\otimes_{n-1}}\otimes \1_{\cC}^{\otimes_{\Nset_0}} \subset \cM^{\alpha_n}$ 
for all $n \geq 1$.
\end{proof}
The representation $\rho$ of $F^+$ may be considered as a perturbation of the representation
$\varrho$ of $S^+$ by locally acting operators $C$ and $D$ on the infinite tensor product factors
of $\cM$.  To be more precise, the choice $D(x) = \1_\cC \otimes x$ yields $\alpha_n =
(\beta_{n-1}\beta_n^*) \beta_n  = \beta_{n-1}$ for $n \ge 1$ and  
$\alpha_0 = (\alpha_0 \beta_0^*) \beta_0$.  
\begin{Theorem}
Let $\varrho \colon S^+ \mapsto \End(\cM,\psi)$ be the representation as introduced in
Proposition \ref{proposition:rep-S+}. Then $(\cM, \psi, \beta_0, \cM^{\beta_1})$ is a
(noncommutative) full Bernoulli shift with generator 
$\cM^{\beta_1} = \cA \otimes \cC \otimes \1_{\cC}^{\otimes_{\Nset_0}}$.   
\end{Theorem}
\begin{proof}
Let $\cB_I := \bigvee_{i \in I}\beta_0^i(\cM^{\beta_1})$ for $I\in\cI(\Nset_0)$  and note that
$\cB_{[0,0]} = \cM^{\beta_1}$. It is straightforward to check that $\cM^{\beta_1} = 
\cA \otimes \cC \otimes \1_{\cC}^{\otimes_{\Nset_0}}$ and, more generally,
$\cB_{[m,n]}=\cA\otimes \1_\cC^{\otimes_{m}}\otimes \cC^{\otimes_{n-m+1}}\otimes
\1_\cC^{\otimes_{\Nset_0}}$ for $0 \le m \le n$.  Since $\cB_{\Nset_0}  = \cM$, the stationary
process $(\cM, \psi, \beta_0, \cM^{\beta_1})$ is minimal. We are left to show that this minimal
stationary process is actually a noncommutative Bernoulli shift (in the sense of Definition 
\ref{definition:ncbs}).
Clearly, $ \cM^{\beta_0} \subset \cM^{\beta_1}$ as $\cM^{\beta_0} 
=  \cA \otimes \1_{\cC}^{\otimes_{\Nset_0}}$. We are left to verify the factorization
\[Q_0(xy)=Q_0(x)Q_0(y)\] for any $x\in \cB_I, y\in \cB_J$ whenever $I \cap J = \emptyset$. 
Here $Q_0$ is the $\psi$-preserving normal conditional expectation from $\cM$ onto
$\cM^{\beta_0}$.  As the conditional expectation $Q_0$ is of tensor type, i.e.
\[
Q_0(a \otimes x_0 \otimes x_1 \otimes \cdots \otimes x_k \otimes \1_\cC^{\otimes_{\Nset_0}})   
= a \otimes \chi(x_0) \chi(x_1) \cdots \chi(x_k)  \1_\cC^{\otimes_{\Nset_0}},  
\]
the required factorization easily follows.
\end{proof}
\begin{Theorem} \label{theorem:Markovtensor}
Let $\rho \colon F^+ \mapsto \End(\cM,\psi)$ be a  representation as introduced in Proposition
\ref{proposition:rep-F+}. Then $(\cM, \psi, \alpha_0, \cM^{\alpha_1})$ is a stationary Markov
process with generator $\cM^{\alpha_1}$. Moreover $(\cM, \psi, \alpha_0, \cA_0)$ is a stationary
Markov process with generator $\cA_0 := \cA \otimes \1_{\cC}^{\otimes_{\Nset_0}} 
\subset \cM^{\alpha_1}$.  
\end{Theorem}
\begin{proof}
By Proposition \ref{proposition:rep-F+}, the representation $\rho$ has the generating property.
Now the  Markovianity of the stationary process  $(\cM, \psi, \alpha_0, \cM^{\alpha_1})$ follows from
Corollary \ref{corollary:markov-filtration-MN}. 
We are left to show that the canonical local filtration of the stationary process 
$(\cM, \psi, \alpha_0, \cA_0)$ is Markovian. We note that the definition of the endomorphism
$\alpha_0$ is independent of the choice of the operator $D$ in  \eqref{eq:operator-D}. 
Moreover, the inclusion $\cA_0 := \cA \otimes \1_{\cC^{\Nset_0}} \subset \cM^{\alpha_1}$ is valid
for any choice of the operator $D$.  Now Corollary \ref{corollary:markov-filtration-MN} can again
be applied to ensure Markovianity if there exists some $D \colon \Cset \to \Cset \otimes \Cset$
such that  $\cA_0 =  \cM^{\alpha_1}$.  It is immediately verified that this equality occurs
for the choice $D(x) = \1_\cC \otimes x$.
\end{proof}
These Markov processes may not be minimal.  Note also that $\cA_0$ may be strictly included in 
$\cM^{\alpha_1}$, as the latter depends on the choice of the operator $D$. For example, strict
inclusion occurs for the choice $D(x) = x \otimes \1_\cC$, but equality occurs for the choice  
$D(x) = \1_\cC \otimes x $ in  \eqref{eq:operator-D}.
\begin{Remark}\normalfont 
We remind the reader that the generating property of $\rho$ and the relations of $F^{+}$
guarantee that the fixed point algebras $\cM^{\alpha_{n}}$ form a tower of inclusions, even
though we may not know explicitly what the fixed point algebras are. In particular, we get
$\cM^{\alpha_1}\subset \cM^{\alpha_2}$. It is not obvious to see this directly (without using 
the relations of $F^{+}$) for a general operator $D$, as used in  \eqref{eq:operator-D} for the
definition of the endomorphisms $\alpha_n$. However, the choice $D(x)=\1_\cC\otimes x$ yields
$\alpha_n=\beta_{n-1}$ for all $ n \ge 1$.  We infer from this that $\cM^{\alpha_{1}}
=  \cM^{\beta_{0}}=\cA \otimes \1_{\cC}^{\otimes_{\Nset}}
\subset \cA \otimes \cC \otimes \1_{\cC}^{\otimes_{\Nset}}
=\cM^{\beta_{1}}=\cM^{\alpha_2}$. 
Similarly, choosing $D(x) =x\otimes \1_\cC$, we get $\alpha_{n}=\beta_n$ for all $n \ge 0$;  
and the inclusion $\cM^{\alpha_{1}} \subset \cM^{\alpha_{2}}$ is clear from the inclusion of
fixed point algebras of the partial shifts $\beta_n$. Finally, we note that if $D$ is one of the
above special random variables, then Markovianity can be proved without appealing 
to the Thompson monoid $F^{+}$, see \cite[Section 2.1]{Go04}.
\end{Remark}
Above, we directly constructed some representations of the Thompson monoid $F^+$ on infinite
tensor products of noncommutative probability spaces and invoked some of our general results
about such representations from Section \ref{section:markovianity-from-F+} to obtain 
noncommutative stationary Markov processes. We present next a converse result which starts with 
a certain class of noncommutative stationary Markov processes for which we will make use of
Proposition \ref{proposition:dilation}. Namely, we show that if a Markov map has a tensor dilation 
(in the terminology of K\"ummerer \cite{Ku85}, see Definition \ref{definition:tensordilation}) then this Markov map can be obtained as the
compression of a represented generator of the Thompson monoid $F^+$. 
\begin{Theorem}\label{theorem:TensorMarkov}
Suppose $\gamma \in \End(\cA \otimes \cC, \varphi \otimes \chi)$ and let $\iota_0$ be the canonical 
embedding of $(\cA, \varphi)$ into $(\cA \otimes \cC, \varphi \otimes \chi)$. Then there exists a 
noncommutative probability space $(\cM,\psi)$, a generating representation 
$\rho \colon F^+ \to \End(\cM,\psi)$ and an embedding 
$\kappa \colon (\cA \otimes \cC, \varphi \otimes \chi) \to (\cM,\psi)$ such that
\begin{enumerate}
	\item 
	$\kappa (\cA \otimes \1_\cC) = \cM^{\rho(g_1)}$,
	\item 
	$\iota_0^* \gamma^n \iota_0   = \iota_0^* \kappa^* \rho(g_0^n)\kappa \iota_0$ 
	for all $n \in \Nset_0$. 	
\end{enumerate}	
In particular, $(\cM,\psi,\rho(g_0), \cM^{\rho(g_1)})$ is a unilateral noncommutative stationary Markov process. 
\end{Theorem}
\begin{proof} 
We take 
\[
	(\cM, \psi) 
	:=  \big(\cA \otimes \cC^{\otimes_{\Nset_0}} , \varphi \otimes \chi^{\otimes_{\Nset_0}}\big)
\]
and construct a representation of the Thompson monoid $F^+$ as obtained in Proposition
\ref{proposition:rep-F+}. That is, we define the representation $\rho \colon F^+ \to
\End(\cM,\psi)$ as $\rho(g_n):=\alpha_n$ as in \eqref{eq:operator-D} with 
$C \colon \cA\to\cA \otimes \cC$ and $D \colon \cC\to \cC \otimes \cC$ given by 
$C(a):= \gamma(a\otimes \1_\cC)$ and $D(x):=\1_\cC \otimes x$. 
 Also recall that the partial shifts $\beta_n$ can be constructed and it can be seen that
 $\gamma_0\beta_0= \alpha_0$, where $\gamma_0$ is the natural extension of $\gamma$ to an endomorphism on $(\cM, \psi)$. Let $\kappa$ be the natural
 embedding of $(\cA \otimes \cC, \varphi \otimes \chi)$ into $(\cM,\psi)$.
 The endomorphism $\gamma_0$ satisfies

 \begin{equation}\label{equation:alpha0}
     \kappa^* \gamma_0^n \kappa \iota_0 = \gamma^n \iota_0 \qquad \text{ for all } n\in \Nset_0.
 \end{equation}
Note that for the case $n=1$, the left hand side of this equation can be written as
\begin{equation} \label{equation:alpha0first}
\kappa^*\gamma_0 \kappa \iota_0 
=
\kappa^*\gamma_0 \beta_0\kappa \iota_0 = 
\kappa^*\alpha_0 \kappa \iota_0.
\end{equation}

Now, by Theorem 
\ref{theorem:Markovtensor}, it follows that $(\cM,\psi,\alpha_0,\cM^{\alpha_1})$ is a
noncommutative stationary Markov process, and $\kappa(\cA\otimes \1_{\cC})=\cM^{\alpha_1}$.

We note that $\kappa \iota_0(\kappa \iota_0)^*$ is the $\psi$-preserving normal conditional expectation from $\cM$ onto $\cM^{\alpha_1} =\kappa \iota_0(\cA)$, and by definition, the stationary Markov process $(\cM,\psi,\alpha_0,\cM^{\alpha_1})$ has the transition operator 
\[
T := \kappa \iota_0(\kappa\iota_0)^*\alpha_0 \kappa\iota_0(\kappa \iota_0)^*.
\]

We observe that \eqref{equation:alpha0} and \eqref{equation:alpha0first} allow us to rewrite $T$ as follows:
\begin{align}\label{equation:Talpha}
T &= \kappa \iota_0 (\kappa \iota_0)^* \alpha_0 \kappa \iota_0 (\kappa \iota_0)^* \\ \nonumber
&= \kappa \iota_0\iota_0^*(\kappa^*\alpha_0 \kappa\iota_0)(\kappa\iota_0)^* \\ \nonumber
&= \kappa\iota_0\iota_0^*(\kappa^*\gamma_0\kappa\iota_0)\iota_0^* \kappa^* \\ \nonumber
&=\kappa\iota_0\iota_0^*\gamma\iota_0\iota_0^*\kappa^*
\end{align}

On the other hand, Proposition \ref{proposition:dilation} gives that
$T$ satisfies 
\begin{equation}\label{equation:dilation}
T^n = \kappa \iota_0 (\kappa \iota_0)^* \alpha_0^n \kappa \iota_0 (\kappa \iota_0)^*  \qquad \text{ for all } n \in \Nset_0. 
\end{equation}

Hence by \eqref{equation:Talpha} and \eqref{equation:dilation},
\begin{align*}
(\kappa \iota_0\iota_0^*) \gamma^n (\kappa \iota_0\iota_0^*)^*
&=[(\kappa \iota_0\iota_0^*) \gamma (\kappa \iota_0\iota_0^*)^*]^n \\
&=T^n
= \kappa \iota_0 (\kappa \iota_0)^* \alpha_0^n \kappa \iota_0 (\kappa \iota_0)^*. 
\end{align*}

Simplifying, we get
\[
\iota_0^* \gamma^n \iota_0 = \iota_0^* \kappa^* \alpha_0^n \kappa \iota_0 \qquad \text{ for all } n \in \Nset_0,   
\]
as claimed in (ii) of the theorem. 
\end{proof}
Suppose $(\cA \otimes \cC, \varphi \otimes \chi, \alpha, \cA \otimes \1_\cC)$ is a stationary
Markov process. Then we just showed that $\alpha$ restricted to the generator $\cA\otimes
\1_{\cC}$ can be obtained as the compression of a represented generator of $F^+$. However, it is
unknown in the generality of the present noncommutative setting if there exists a representation
$\rho \colon F^+ \to \End(\cA \otimes \cC, \varphi \otimes \chi)$ such that the Markov shift
$\alpha$ equals the represented generator $\rho(g_0)$ itself. Thus it is unknown if the
canonically associated stationary Markov sequence of random variables 
$\iota \equiv (\iota_n)_{n \in \Nset_0} \colon (\cA,\varphi)  \to 
(\cA \otimes \cC, \varphi \otimes \chi)$ is maximal partially spreadable. We will see in the
Subsection \ref{subsection:constr-class-prob} that this canonically associated sequence 
$\iota$ is maximal partially spreadable in our algebraic framework for classical probability, 
see Theorem \ref{theorem:classical-Markov-Thompson}. Furthermore we will investigate in Subsection
\ref{subsection:op-alg-construction} an operator algebraic setting which allows to deduce that
certain noncommutative stationary Markov sequences are partially spreadable.  
\subsection{Constructions in classical probability}
\label{subsection:constr-class-prob}
The tensor product constructions from Subsection \ref{subsection:tensor-product} apply of course
to commutative von Neumann algebras (with separable predual) as they are of relevance in
classical probability theory: a von Neumann algebra with separable predual is 
isomorphic to the essentially bounded functions on some standard probability space. Constructions
and results on an algebraic reformulation of classical Markov processes were discussed in Section
\ref{section:MarkovianityandF+Classical}. Here we will provide the proof of the classical de
Finetti theorem for stationary Markov sequences, Theorem \ref{theorem:definetti-intro}, in its 
algebraic reformulation, Theorem \ref{theorem:commdeFinetti}.

We recall that a stationary Markov process is completely determined by its transition operator in
classical probability, up to equivalence in distribution. This folklore result can be
reformulated in the present operator algebraic setting as done next.
\begin{Proposition}\normalfont \label{proposition:commutative-Markov-equivalence}
Consider the two stationary Markov sequences $\iota \equiv (\iota_n)_{n \in \Nset_0} \colon $ $(\cA,\varphi) \to (\cM,\psi)$ 
and $\tiota \equiv (\tiota_n)_{n \in \Nset_0}\colon (\cA,\varphi) \to (\tcM,\tpsi)$ with transition operators given by the $\varphi$-Markov maps on
$\cA$, $R$ and $\tR$  respectively. If $\cM$ and $\tcM$ are commutative von Neumann algebras,
then the following are equivalent:   
\begin{enumerate}
	\item[(a)] $\iota \stackrel{\distr}{=}\tiota$;
	\item[(b)] $R = \tR$. 
\end{enumerate}
\end{Proposition}
\begin{proof}
(a) $\Longrightarrow$ (b):  We conclude from Lemma \ref{lemma:qregression} and from 
$\iota \stackrel{\distr}{=}\tiota$ that 
\[
\varphi\big(a R(b)\big) = \psi\big( \iota_0(a)\iota_1(b)\big)  
      = \tpsi\big(\tiota_0(a) \tiota_1(b)\big) = \varphi\big(a \tR(b)\big). 
\]
for all $a,b \in \cA$. But this implies $R = \tR$ by routine arguments. \\
(b) $\Longrightarrow$ (a):  We need to show that $R = \tR$ implies 
\[
\psi\big(\iota_{k_1}(a_1) \cdots \iota_{k_n}(a_n)\big) 
       =  \tpsi\big(\tiota_{k_1}(a_1) \cdots \tiota_{k_n}(a_n)\big)
\]
for any $a_1, \ldots, a_n \in \cA$ and $k_1, \ldots k_n \in \Nset_0$ and $n \in \Nset$. 
Since $\cM$ and $\tcM$ are commutative von Neumann algebras, and
since random variables are (injective) *-homomorphisms, we can assume 
$0 \le k_1 < k_2 < \ldots < k_n$ without loss of generality.  We use Lemma
\ref{lemma:qregression} and $R = \tR$ to calculate
\begin{align*}
\psi\big(\iota_{k_1}(a_1) \cdots \iota_{k_n}(a_n)\big) 
&= \varphi\big (a_1 R^{k_2-k_1}(a_2  \cdots  R^{k_n-k_{n-1}}(a_n)\big)\\ 
& = \varphi\big (a_1 \tR^{k_2-k_1}(a_2  \cdots  \tR^{k_n-k_{n-1}}(a_n)\big)\\
&=\tpsi\big(\tiota_{k_1}(a_1) \cdots \tiota_{k_n}(a_n)\big). 
\end{align*}
\end{proof}
\begin{Notation}\normalfont
As in Subsection \ref{subsection:markov-dilations}, throughout this subsection $\lambda$ 
denotes the Lebesgue measure on the unit interval $[0,1] \subset \Rset$ and the 
(non-)commutative probability space $(\cL, \trace_\lambda)$ is given by 
$\cL := L^\infty([0,1],\lambda)$ and $\trace_\lambda := \int_{[0,1]} \cdot\,  d\lambda$. 
\end{Notation}
We recall from Theorem \ref{theorem:kuemmerer-onesided} that given a
probability space $(\cA,\varphi)$, where $\cA$ is commutative and a
$\varphi$-Markov map $R$ on $\cA$, there exists a Markov dilation
$(\cA\otimes \cL,\varphi\otimes \trace_{\lambda}, \alpha, \iota_0)$ of $R$.
Hence, as already seen before in the context of tensor product constructions,
each Markov map in the present algebraic framework of classical probability
can be obtained as the compression of a represented generator of the 
Thompson monoid $F^+$. 
\begin{Theorem}\label{theorem:F+-gen-compression}
Let $(\cA,\varphi)$  be a probability space where $\cA$ is commutative with separable predual,
and let $R$ be a $\varphi$-Markov map on $\cA$. There exists a probability space $(\cM,\psi)$,
a generating representation $\rho \colon F^+ \to \End(\cM,\psi)$ and an embedding 
$\iota \colon (\cA, \varphi) \to (\cM,\psi)$ such that
\begin{enumerate}
\item 
$\iota(\cA) = \cM^{\rho(g_1)}$,
\item
$R^n  = \iota^* \rho(g_0^n)\iota$ 
for all $n \in \Nset_0$. 	
\end{enumerate}	
\end{Theorem}
\begin{proof}
By Theorem \ref{theorem:kuemmerer-onesided}, there exists $\alpha\in \End(\cA\otimes
\cL,\varphi\otimes \trace_{\lambda})$ such that $(\cA\otimes \cL, \varphi\otimes
\trace_{\lambda}, \alpha, \cA\otimes \1_{\cL})$ is a stationary Markov process, and $R^n=\iota_0^*\alpha^n\iota_0$, for all $n\in \Nset_0$, where $\iota_0 \colon (\cA, \varphi) \to (\cA \otimes \cL, \varphi, \otimes \trace_{\lambda})$ denotes the canonical embedding $\iota_0(a)=a\otimes \1_{\cL}$. The proof then
follows from Theorem \ref{theorem:TensorMarkov} by taking $\iota := \kappa \circ \iota_0$, as we get
\[
R^n = \iota_0^* \alpha^n \iota_0 = \iota_0^* \kappa^* \rho(g_0)^n\kappa \iota_0 = \iota^* \rho(g_0^n)\iota \qquad \text{ for all } n \in \Nset_0.
\]
\end{proof}
Together with our general results from Section \ref{section:markovianity-from-F+}, our next
result paves the road for a de Finetti theorem, Theorem \ref{theorem:commdeFinetti}, for
(recurrent) stationary Markov sequences with values in a standard Borel space.  
\begin{Theorem}\label{theorem:classical-Markov-Thompson}
Let $(\cA,\varphi)$ and $(\cM,\psi)$ be probability spaces such that $\cA$ and $\cM$ are
commutative. A stationary Markov sequence $\iota \equiv (\iota_n)_{n \in \Nset_0} \colon 
(\cA, \varphi) \to (\cM,\psi)$ is (maximal) partially spreadable.
\end{Theorem}
\begin{proof}
We will show that there exists a probability space $(\tcM,\tpsi)$ and a sequence $\tiota \equiv (\tiota_n)_{n \in \Nset_0} 
\colon (\cA, \varphi) \to (\tcM,\tpsi)$ which has the same distribution as the stationary 
Markov sequence $\iota$ and which satisfies, for all $n \in \Nset$, 
\[
\trho(g_n) \tiota_0 = \tiota_0 \quad \text{and} \quad \trho(g_0^n) \tiota_0 = \tiota_n
\]
for some representation $\trho\colon F^+ \to \End(\tcM,\tpsi)$.  

Since $\iota$ is stationary and Markovian there exist a $\varphi$-Markov map $R$ on $\cA$ 
such that
\[
\varphi(a R(b)) = \psi\big(\iota_0(a) \iota_1(b)\big).  
\] 

By Theorem \ref{theorem:F+-gen-compression}, there exists a generating representation
$\trho\colon F^+ \to \End(\tcM,\tpsi)$, where  $(\tcM,\tpsi) = \big(\cA \otimes
\cL^{\otimes_{\Nset_0}},\varphi \otimes \trace_\lambda^{\otimes_{\Nset_0}}\big)$, and an
embedding $\kappa \colon (\cA \otimes \cL, \varphi \otimes \trace_\lambda) \to (\tcM,\tpsi)$ 
such that $\kappa(\cA \otimes \1_\cL) = \tcM^{\trho(g_1)}$.

We infer from Theorem \ref{theorem:Markovtensor} that 
$\big(\tcM,\tpsi,\trho(g_0), \cA \otimes \1_{\cL}^{\otimes_{\Nset_0}}\big)$ is a stationary
Markov process such that 
\[
\varphi(a R(b)) 
= \tpsi\big((a \otimes \1_{\cL}^{\otimes_{\Nset_0}}) \trho(g_0)  
               (b \otimes \1_{\cL}^{\otimes_{\Nset_0}}) \big).
\] 
Consequently, the sequence of random variables $\tiota \equiv (\tiota_n)_{n \in \Nset_0} \colon
(\cA,\varphi) \to (\tcM,\tpsi)$, defined by 
\[
\tiota_0(a) :=  a \otimes \1_\cL^{\otimes_{\Nset_0}} 
\quad \text{and} \quad  
\tiota_n(a) := \trho(g_0)^n\iota_0(a) \quad \text{for $n > 0$}, 
\] 
is Markovian and partially spreadable, both by construction. Additionally, $\tiota$ is
\emph{maximal} partially spreadable as 
$\tiota_0(\cA)=\cA\otimes \1_{\cL}^{\otimes_{\Nset_0}}=\tcM^{\trho(g_1)}=\tcM_0$.
Furthermore, the sequences $\tiota$ and $\iota$ have the same distribution,
as they are stationary Markov sequences with the same Markov map $R$ (see Proposition \ref{proposition:commutative-Markov-equivalence}).      
\end{proof}
Theorem \ref{theorem:maxps} and Theorem \ref{theorem:classical-Markov-Thompson}
consolidate to the following result in the setting of commutative probability
spaces.
\begin{Theorem}\label{theorem:commdeFinetti}
Let $(\cA,\varphi)$ and $(\cM,\psi)$ be probability spaces such that $\cA$
and $\cM$ are commutative with separable predual. Let
$\iota\equiv(\iota_n)_{n\in \Nset_0} \colon (\cA,\varphi) \to (\cM,\psi)
$ be a sequence of random variables. Then the following are equivalent:
\begin{enumerate}
\item[(a)] $\iota$ is a maximal partially spreadable sequence;
\item[(b)] $\iota$ is a stationary Markov sequence.
\end{enumerate}
\end{Theorem}
Thus we have arrived at the algebraic formulation of a de Finetti theorem
for stationary Markov sequences. Its traditional formulation in terms of
classical random variables is available in Theorem \ref{theorem:definetti-intro} and discussed in Subsection \ref{subsection:DIPinCP}.
 
\subsection{Constructions in the framework of  operator algebras}
\label{subsection:op-alg-construction}
K\"ummerer's approach to an operator algebraic theory of stationary Markov processes is based on
the concept of a coupling representation (see \cite{Ku85,Ku93} for example).  Here we adapt and
refine this approach such that it provides a rich operator algebraic framework for the 
construction of representations of the Thompson monoid $F^+$. 

Our investigations are motivated by the elementary observation that the relations of the Thompson
monoid $F^+$ are robust under certain algebraic `perturbations' which we introduce and 
formalize next. 
\begin{Definition}\normalfont \label{definition:EF^+}
The \emph{extended monoid $EF^+$} is presented by the set of generators 
$\{g_n, c_n \mid n \in \Nset_0\}$ subject to the relations
\begin{alignat*}{4}
g_{k}g_{\ell}&=g_{\ell+1}g_{k}, \qquad
c_k c_{\ell+1} &= c_{\ell+1}c_{k}, \qquad 
c_k g_{\ell+1} &=  g_{\ell+1} c_k, \qquad 
g_k c_{\ell} &=  c_{\ell+1} g_k \qquad 
\end{alignat*}
for every $0 \le k < \ell <\infty$.
\end{Definition}
Evidently the first set of generators $\{g_n\}_{n \in \Nset_0}$ satisfies the relations of the
Thompson monoid $F^+$.  
\begin{Proposition}
The submonoid $QF^+:= \langle c_ng_n \mid n \in \Nset_0 \rangle^+ \subset EF^+$ is a quotient of
the monoid $F^+$.
\end{Proposition}
\begin{proof}
An elementary computation, based on all defining relations of the monoid $EF^+$, shows that the
elements of the set $\{\tg_n :=  c_ng_n \mid n \in \Nset_0\}$  satisfy the relations of the
Thompson monoid $F^+$ for $0 \le k <  \ell <\infty$:
\begin{align*}
\tg_{k} \tg_{\ell}     =   c_k g_k   c_\ell g_\ell    
                                 =   c_k   c_{\ell+1}  g_k g_\ell  
                                & =    c_{\ell+1} c_k  g_{k} g_\ell \\ 
                                & =    c_{\ell+1}  c_k   g_{\ell+1}  g_{k}
                                 =    c_{\ell+1} g_{\ell+1}   c_k    g_{k} 
                                  =  \tg_{\ell+1} \tg_{k} . 
\end{align*}
\end{proof}
\begin{Definition}\normalfont\label{definition:ES+} 
The \emph{extended monoid $ES^+$} is presented by the set of generators 
$\{h_n, c_n \mid n \in \Nset_0\}$ subject to the relations 
\begin{alignat*}{4}
h_{k}h_{\ell}&=h_{\ell+1}h_{k}, \qquad
c_k c_{\ell+1} &= c_{\ell+1}c_{k}, \qquad 
c_k h_{\ell+1} &=  h_{\ell+1} c_k,  \qquad
h_k c_l  &=c_{l+1} h_k,
\end{alignat*}
and $h_{k}h_{k} =h_{k+1}h_{k}$ for every $0 \le k < \ell < \infty$.
\end{Definition}
Clearly the monoid $ES^+$ is a quotient of the monoid $EF^+$, due to the additional set of
relations for the $h_k$'s. The extended monoid $ES^+$  algebraically encodes certain local
perturbations of the partial shifts monoid $S^+$ which, roughly phrasing, corresponds to the
perturbation of Bernoulli shifts such that one obtains  Markov shifts in classical probability. 
\begin{Proposition}
The submonoid $QS^+:= \langle c_nh_n \mid n \in \Nset_0 \rangle^+ \subset ES^+$ is a quotient 
of the monoid $F^+$.
\end{Proposition}
\begin{proof}
An elementary computation shows that the elements of the set 
$\{\tg_n :=  c_n h_n \mid n \in \Nset_0\}$  satisfy the relations 
of the Thompson monoid $F^+$.
\end{proof}
\begin{Remark}\normalfont
Actually the relations in Definition \ref{definition:EF^+} have been identified by some reverse
engineering: each $c_k$ should provide a suitable `local perturbation' of $g_k$ such that 
$(c_k g_k)(c_\ell g_\ell)  =(c_{\ell+1} g_{\ell+1})  (c_k g_k)$ for $0 \le k < \ell < \infty$.
An alternative `perturbation' is given by the  extended monoid $FF^+$ which is defined to be 
presented by generators $\{c_n,g_n\}_{n \in \Nset_0}$ subject to the relations
\begin{alignat*}{4}
g_{k}g_{\ell}&=g_{\ell+1}g_{k}, \qquad
c_k c_{\ell} &= c_{\ell+1}c_{k}, \qquad 
c_k g_{\ell+1} &=  g_{\ell+1} c_k, \qquad 
g_k c_{\ell} &=  c_{\ell} g_k \qquad 
\end{alignat*}
for every $(0 \le k < \ell < \infty)$.
Here both the $c_k$'s and the $g_k$'s satisfy the relations of the Thompson monoid $F^+$ and 
the last two sets of relations can be combined to a single set of relations on commutativity:
$c_k g_\ell = g_\ell c_k$ whenever $k \notin \{\ell-1, \ell\}$.   
\end{Remark}\normalfont
\begin{Remark}\normalfont
The results on semi-cosimplicial structures as obtained in \cite{EGK17} make it tempting to also
investigate a  more restrictive perturbed version for the partial shifts monoid $S^+$.  So let
the \emph{extended semi-cosimplicial monoid $ES^+_r$} be presented by the set of generators
$\{h_n, d_n \mid n \in \Nset_0\}$ subject to the relations
\begin{alignat*}{3}
h_{k}h_{\ell}&=h_{\ell+1}h_{k}, \qquad
d_k d_{\ell+1} &= d_{\ell+1}d_{k}, \qquad 
d_k h_{\ell+1} &=  h_{\ell+1} d_k, \qquad
h_k d_l &= d_{l+1}h_k   
\end{alignat*}
for every $0 \le k \le \ell < \infty$.
These relations ensure that the submonoid $QS^+_r:= \langle c_ng_n \mid n \in \Nset_0 \rangle
\subset ES^+_r$ is a quotient of the monoid $S^+$.  In comparison to the extended monoid $EF^+$, 
the additional relations  are more restrictive for possible extensions of  the monoid $S^+$.
Roughly phrasing, these additional relations encode algebraically the perturbative
difference between Markovianity and stochastic independence  in classical probability. We
conjecture that $QS^+_r$ and $S^+$ are isomorphic as monoids.
\end{Remark}
Similarly as it was discovered for the monoid $F^+$ in Section
\ref{section:markovianity-from-F+}, the representation theory of these extended monoids in the
endomorphisms of a noncommutative probability space goes along with very rich structures of 
commuting squares.  Here we restrict ourselves to present a single result, mainly in the
intention to illustrate how Bernoulli shifts and, as their perturbation, Markov shifts can be
simultaneously obtained from the representation theory of the extended monoid $ES^+$.  

Recall from Definition \ref{definition:property} that a noncommutative stationary process
$(\cM,\psi,\beta,\cA_0)$ is said to be spreadable if the canonically associated sequence of random 
variables $(\iota_n)_{n \ge 0} \colon (\cA_0,\psi_0) \to (\cM,\psi)$ is spreadable, where 
$\iota_n := \beta^n |_{\cA_0}$ and $\psi_0 = \psi|_{\cA_0}$.
\begin{Theorem}\label{theorem:ES+}
 Suppose $(\cM,\psi)$ is equipped with a representation $\rho\colon ES^+ \to \End(\cM,\psi)$.  
Let $\cB_0 := \cM^{\rho(h_1)}$ and $(\cB_\infty,\psi_\infty) 
:= \big(\bigvee_{n\in \Nset_0} \rho(h_0^n) (\cB_0), \psi_{|_{\cB_\infty}}\big)$. 
Further let $\cA_0 := \bigcap_{k \ge 1} \cM^{\rho(c_k h_k)}$.
\begin{enumerate}
\item
The restricted represented generator $\beta :=  \rho(h_0)_{|_{\cB_\infty}}$ defines the
spreadable Bernoulli shift $(\cB_\infty, \psi_\infty, \beta,\cB_0)$.
\item 
The represented generator  $\alpha :=  \rho(c_0 h_0)$ defines the (not necessarily minimal)
stationary Markov process $(\cM, \psi, \alpha, \cA_0)$.
\end{enumerate}
If $\cM = \cB_\infty$ and $\cA_0 = \cB_\infty^\beta$, then the stationary Markov process 
$(\cM, \psi, \alpha, \cA_0)$ has the coupling representation
$(\cM, \psi, \gamma \beta, \cA_0)$, with coupling $\gamma := \rho(c_0)$. 
\end{Theorem}
\begin{proof}
(i) Let $\rho_B(h_n):=\beta_n:=\rho(h_n), h_n\in S^+, n\in \Nset_0$. Then $\beta=\beta_0$ and
$\rho_B$ gives a representation of the monoid $S^+$ in $\End(M,\psi)$. Hence $(\cB_{\infty},
\psi_{\infty}, \beta, \cB_0)$ is spreadable (compare Definition \ref{definition:property}). Also observe
that $\cB_\infty^{\beta}=\cB_\infty^{\rho_B(h_0)}\subset \cB_\infty^{\rho_B(h_1)}\subset\cB_0$
due to the relations of $S^{+}$. Now, the fact that it is a Bernoulli shift follows from
\cite[Theorem 8.2]{Ko10}.\\
(ii) Let $\rho_M(g_n):=\alpha_n:=\rho(c_n h_n), g_n\in F^+, n\in \Nset_0$. Then $\alpha=\alpha_0$
and $\rho_M$ gives a representation of the monoid $F^+$ in $\End(\cM,\psi)$. Also observe that
$\cM^{\rho_M}_0:=\bigcap_{k \ge 1} \cM^{\rho_M(g_k)} =\bigcap_{k \ge 1} \cM^{\rho(c_k h_k)}
=\cA_0$, hence by Corollary \ref{corollary-markov-filtration-0}, $(\cM, \psi, \alpha, \cA_0)$ is
a stationary Markov process.
\end{proof}
The significance of this result is that it indicates a promising strategy of how to construct a
representation of the Thompson monoid $F^+$ from a large class of noncommutative stationary
Markov processes. The starting point is the construction of a spreadable noncommutative Bernoulli
shift which is known to be in a bijective correspondence to equivalence classes of spreadable
sequences of noncommutative random variables (see \cite{Ko10,EGK17}). In other words, the
construction of a spreadable Bernoulli shift amounts to the construction of a representation of
the partial shift monoid $S^+$. But as this monoid is a quotient of the Thompson monoid $F^+$,
spreadable Bernoulli shifts correspond to a particular class of representations of the Thompson
monoid $F^+$.  Suitable perturbations of this particular class will provide certain Markov shifts
and wider classes of representations of the Thompson monoid $F^+$.  
\begin{Proposition}\label{proposition:spread-BNS}
Suppose $(\cM,\psi,\beta,\cB_0)$ is a spreadable noncommutative Bernoulli shift. Then there exists a
generating representation $\rho_\beta\colon S^+ \to \End(\cM,\psi)$ such that 
$\beta = \rho_\beta (h_0)$ and $\cB_0 \subset \cM^{\rho_\beta(h_k)}$ for all $k \ge 1$.  
\end{Proposition}
\begin{proof}
If $(\cM,\psi,\beta,\cB_0)$ is spreadable, then as it is a minimal  noncommutative Bernoulli
shift, there exists a representation $\rho_\beta \colon S^+ \to \End(\cM,\psi)$ 
such that for $\lambda_n:=\beta^n|_{\cB_{0}}$ we get $\lambda_n =  \rho_\beta(h_0^n)\lambda_0$
for all $n \in \Nset_0$ and $\rho_\beta(h_k) \lambda_0 = \lambda_0$ for all $k \ge 1$ 
(see \cite[Theorem 4.5]{EGK17}).  The representation $\rho_\beta$ has the generating property 
by construction. 
\end{proof}
\begin{Definition}\normalfont
The representation $\rho_\beta \colon S^+ \to \End(\cM,\psi)$ (as introduced in Proposition
\ref{proposition:spread-BNS}) is said to be \emph{associated} to the spreadable Bernoulli shift 
$(\cM,\psi,\beta,\cB_0)$.
\end{Definition}
\begin{Corollary}\normalfont
A spreadable Bernoulli shift $(\cM,\psi,\beta,\cB_0)$ is partially spreadable. 
\end{Corollary}
\begin{proof}
As in Proposition  \ref{proposition:spread-BNS}, let $\rho_\beta$ be the representation associated to
the spreadable noncommutative Bernoulli shift. Denote by $\epsilon \colon S^+ \to F^+$ the
canonical epimorphism which maps the generator $g_k \in F^+$  to the generator $h_k \in S^+$ for
all $k \in \Nset_0$.  Then $\rho := \rho_\beta \circ \epsilon$ defines a representation of $F^+$
such that the canonically associated sequence of random variables $(\lambda_n)_{n \ge 0}$ (as
used in the proof of Proposition \ref{proposition:spread-BNS}) is partially spreadable.   
\end{proof}
A large class of noncommutative Markov shifts can be obtained as certain perturbations of
noncommutative Bernoulli shifts, as developed and investigated by K\"ummerer in \cite{Ku85,Ku86}.
We refine the notion of a coupling representation so that it applies to spreadable noncommutative
Bernoulli shifts. 
\begin{Definition}\normalfont
A sequence $(\gamma_n)_{n \ge0} \in \End(\cM,\psi)$ is called a \emph{coupling (sequence)} to a
spreadable Bernoulli shift  $(\cM,\psi,\beta,\cB_0)$ with associated representation 
$\rho\colon S^+ \to \End(\cM,\psi)$ if, for all $0 \le k < \ell < \infty$,  
\begin{alignat*}{2}
 \rho(h_k)\, \gamma_{\ell}& =  \gamma_{\ell+1}\,  \rho(h_k),\qquad
 \gamma_{k}\, \rho(h_{\ell+1}) &= \rho(h_{\ell+1}) \,  \gamma_{k},\qquad
 \gamma_k \gamma_{\ell+1} &=  \gamma_{\ell+1} \gamma_k.
\end{alignat*}
\end{Definition}
\begin{Proposition}
Let $(\gamma_n)_{n \ge 0}$ be a coupling sequence to the spreadable Bernoulli shift
$(\cM,\psi,\beta,\cB_0)$ with associated representation  $\rho_\beta \colon S^+ \to
\End(\cM,\psi)$. Then a representation $\rho \colon F^+ \to \End(\cM,\psi)$
is defined by the multiplicative extension of 
\[
F^+ \ni g_k  \mapsto  
\gamma_k \rho_\beta^{} \epsilon(g_k) \in \End(\cM,\psi) \qquad (0 \le k < \infty).   
\]
Here $\epsilon$ denotes the canonical epimorphism from $F^+$ onto $S^+$. 
\end{Proposition}
\begin{proof}
The relations of the Thompson monoid $F^+$ are satisfied by $\gamma_k
\rho_{\beta}\varepsilon(g_k)$ as for $0\leq k< \ell< \infty$, the definition of a coupling
sequence ensures that
\begin{align*}
(\gamma_k \rho_{\beta}\varepsilon(g_k))(\gamma_{\ell} \rho_{\beta}\varepsilon(g_{\ell}))
&=(\gamma_k \rho_{\beta}(h_k))(\gamma_{\ell} \rho_{\beta}(h_{\ell})) \\
&=\gamma_k \gamma_{\ell+1}\rho_{\beta}(h_k)\rho_{\beta}(h_{\ell}) \\
&=\gamma_{\ell+1}\gamma_k \rho_{\beta}(h_{\ell+1})\rho_{\beta}(h_k)\\
&=\gamma_{\ell+1}\rho_{\beta}(h_{\ell+1})\gamma_k\rho_{\beta}(h_k)\\
&=(\gamma_{\ell+1} \rho_{\beta}\varepsilon(g_{\ell+1}))(\gamma_k \rho_{\beta}\varepsilon(g_k)).
\end{align*}
\end{proof}
\begin{Remark}\normalfont
It is known that there exist non-spreadable noncommutative Bernoulli shifts (see \cite{Ko10}). 
We conjecture that there exist also noncommutative Bernoulli shifts without partial
spreadability. An affirmative answer to this conjecture implies that there exist noncommutative
Markov shifts beyond the representation theory of the Thompson monoid $F^+$. 
\end{Remark}
\section*{Acknowledgements}
\label{section:acknowledgements}
The first author would like to thank the Isaac Newton Institute for
Mathematical Sciences in Cambridge, England, for support and hospitality
during the programme \emph{Operator Algebras: Subfactors and their
Applications} in the spring of 2017 where some initial studies for this
paper were undertaken. The second author was supported by a Government of
Ireland Postdoctoral Fellowship (Project ID: GOIPD/2018/498). 
The first and second author acknowledge several helpful discussions
with B.~V.~Rajarama Bhat in an early stage of this project. The first and 
third author would like to thank Persi Diaconis for stimulating discussions 
on Markov chains during his visit at University College Cork in September 2018. 
The first author would like to thank Gwion Evans, Rolf Gohm, Burkhard K\"ummerer 
and Hans Maassen for fruitful discussions on noncommutative Markov processes. 
\label{section:bibliography}

\end{document}